\documentclass[english]{article}
\usepackage[T1]{fontenc}
\usepackage[latin9]{inputenc}
\usepackage{color}
\usepackage{dsfont}
\usepackage{amsmath}
\usepackage{amsthm}
\usepackage{amssymb}

\makeatletter
\theoremstyle{plain}
\newtheorem{thm}{\protect\theoremname}[section]
\theoremstyle{plain}
\newtheorem{lem}[thm]{\protect\lemmaname}
\theoremstyle{plain}
\newtheorem{prop}[thm]{\protect\propositionname}
\newtheorem{corollary}[thm]{\protect\corname}

\usepackage{babel}

\newcommand{\Tr}{\text{Tr}}

\providecommand{\corname}{Corollary}

\providecommand{\lemmaname}{Lemma}
\providecommand{\propositionname}{Proposition}

\providecommand{\theoremname}{Theorem}

\makeatother

\usepackage{babel}
\providecommand{\lemmaname}{Lemma}
\providecommand{\propositionname}{Proposition}
\providecommand{\theoremname}{Theorem}
\title{Large deviation principle for the largest eigenvalue of random matrices
with a variance profile}
\author{Rapha\"el Ducatez, Alice~Guionnet and Jonathan~Husson}

\begin{document}
\maketitle
\begin{abstract}

We establish large deviation principles for the largest eigenvalue of large random matrices with variance profiles. For $N \in \mathbb N$, we consider random  $N \times N$ symmetric matrices $H^N$ which are such that $H_{ij}^{N}=\frac{1}{\sqrt{N}}X_{i,j}^{N}$ for $1 \leq i,j \leq N$, where the $X_{i,j}^{N}$ for $1 \leq i \leq j \leq N$ are independent and centered. We then denote $\Sigma_{i,j} ^N = \text{Var} (X_{i,j}^{N}) ( 1 + \mathds{1}_{ i =j})^{-1}$ the variance profile of $H^N$. Our large deviation principle is then stated under the assumption that the $\Sigma^N$ converge in a certain sense toward a real continuous function $\sigma$ of $[0,1]^2$ and that the entries of $H^N$ are sharp sub-Gaussian. Our rate function is expressed in terms of the solution of a Dyson equation involving $\sigma$. This result is a generalization of \cite{HussonVar} and is new even in the case of Gaussian entries. 

\end{abstract}

\tableofcontents{}

\bigskip

\thanks{This project   has received funding from the European Research Council (ERC) under the European Union
Horizon 2020 research and innovation program (grant agreement No. 884584).
}
\section{Introduction}

Large random matrices appear in a wide variety of problems. They were first introduced into statistics in  Wishart's work  \cite{wishart} to analyze large arrays of noisy data, and are still central in statistics for
 principal component analysis and statistical learning. Wigner \cite{wigner} and Dyson \cite{dyson}  conjectured that the statistics of their eigenvalues model those of high energy levels in heavy nuclei.
They have also been applied to model the stability of large dynamical systems such as food webs \cite{May72} and neural networks \cite{RaAb:neural}. 
Random matrices  have recently played a central role in the study of the complexity of random energy landscapes which are governed by the large deviations of the smallest eigenvalue of their Hessian \cite{ABA,BAMMN,BBMcK}. Random matrices also play a key role in operator algebra and more precisely in free probability since  Voiculescu \cite{voicstflour,voi91}  proved that independent matrices  are asymptotically free.

 In all these problems, random matrices with independent equidistributed entries (modulo the symmetry constraint) were first considered. Indeed, such matrices, called Wigner matrices,  are somehow the easiest to study, in particular because when the entries are Gaussian, the joint law of the eigenvalues is explicit and amenable to analysis \cite{ME,AGZ}. The first and central result of random matrix theory concerns the convergence of the spectrum of such random matrices. Wigner  \cite{wigner} proved that the empirical measure of the eigenvalues of Wigner matrices with entries with finite moments converges towards the so-called semi-circle law. F\"uredi  and Koml\'os  then showed that the largest eigenvalue stick to the bulk in the sense that it converges towards the boundary of the support of the semi-circle law. The Gaussian fluctuations of the empirical measure of the eigenvalues were also studied in great details in a series of papers \cite{PS,LP} whereas local fluctuations were established and shown to be universal \cite{TW99, TW5,TaVu, erdy, TV10, ESY11}. Much less is known about the large deviations for the empirical measure, which were mostly established in the Gaussian case \cite{BAG97}, and in the case where the tails of the entries is heavier than the Gaussians \cite{BordCap}. On can also mention the recent work \cite{Au24} that deals with the large deviations of the empirical measure of adjacency matrices of Erdös-Renyi graphs. However, we have now a fairly complete picture of the large deviations for the largest eigenvalue first in the Gaussian case\cite{BDG}, then in the case of tails heavier than Gaussian's \cite{fanny} and finally when the entries are subGaussian \cite{ HuGu1,HuGu2,AGH, CoDuGu}. In particular, it was shown that large deviations of the largest eigenvalue of Wigner matrices are universal provided the entries are sharp sub-Gaussian in the sense that their Laplace transform is bounded by that of the Gaussian with the same variance. This includes Rademacher or uniformly distributed entries. There has also been multiple works regarding the large deviations for the largest eigenvalue of adjacency matrices of sparse Erdös-Renyi graphs $\mathcal{G}(n,p)$ for different regimes of sparsity when $N$ goes to $+\infty$ including some recent advances
 (see \cite{ChaVar11,ChaVar12,LubZha15} in the dense regime $p \sim 1$, \cite{ChaDem16,BhaGan20} in the regime $ \frac{1}{\sqrt{n}} \ll p \ll 1$ and \cite{BhaBhaGan21,GanNam22,Gan22} in the regime $p \sim \frac{1}{n}$).

Most of the previous  questions are very relevant for more complex models, for instance for the so-called structured matrices or Wigner matrices with a variance profile where the entries are still independent but with a variance which depends on the position.  In this case, the joint law of the eigenvalues is not  explicit, even when the entries are Gaussian. The results about these matrices are much more recent.  The convergence of the spectrum was established by Girko \cite{girko}, see also \cite{Sh98}. It can also be seen that the eigenvalue stick to the bulk (see the discussion in Section \ref{convsec}). Fluctuations of the empirical measure could be studied \cite{AZ04,guionnet-clt}, as well as local fluctuations \cite{AEK2,Alt18,AlErKr, AEK1,AEK3, CEKS1,CEKS2}.
In this article, our goal is to derive large deviations for the largest eigenvalue of such matrices. 
We prove a large deviation principle for the largest eigenvalue of
these matrices under the condition that the variables $(X_{i,j}^{N})_{i\le j}$
have sharp sub-Gaussian taills. The same  question was studied  by the second author in \cite{HussonVar} but the results required a quite restrictive
assumption on the profile, see \cite[Assumptions 2.4, 2.5 and 2.8]{HussonVar},
which we show is not necessary provided the rate function is conveniently
generalized. To simplify the notations, we will assume that the entries
are real but it is straightforward to generalize our results to the complex case.

\subsection{Model }
In this article, we consider Wigner matrices  with a variance profile.  We will denote these matrices by 
$H^{N}$, $N\in\mathbb{N}$ and will assume that the entries which are of the form $H_{ij}^{N}=\frac{1}{\sqrt{N}}X_{i,j}^{N}$
where $(X_{i,j}^{N})_{i\le j}$ are centered independent variables
with non trivial variance profile $\Sigma_{i,j}^{N} $.

More precisely, the model will satisfy the following hypothesis: \label{thehyp} 
\begin{itemize}
\item $H^{N}=\frac{1}{\sqrt{N}}X^{N}$ belongs to the set $\mathcal{S}_{N}(\mathbb{R})$
of $N\times N$ real symmetric matrices. $(X_{ij}^{N})_{i\le j}$
are independent centered random variables whose variance profile is
uniformly bounded that is there exists a finite constant $C$ independent
of $N$ such that $\max_{i,j}\Sigma_{i,j}^{N}\leq C$. 
\item The entries are \emph{sharp subgaussian} variables \cite{HuGu1},
namely for every real number $t$, {
\begin{equation}\label{SSH}
\frac{1}{t^{2}}\log\mathbb{E}e^{tX_{ij}^{N}}\leq\frac{1}{2^{1_{i\neq j}}}\Sigma_{i,j}^{N}
\end{equation}
for all $i,j\in \{1,\ldots, N\}$. 
Observe that this implies that $\mathbb{E}[(X_{i,j}^{N})^{2}]=\Sigma_{i,j}^{N}$ when $i\neq j$ but  when $j=i$, we have $\mathbb{E}[(X_{i,i}^{N})^{2}]=2\Sigma_{i,i}^{N}$.   }
\item $\Sigma_{i,j}^{N}$ is \emph{asymptotically piecewise continuous}
: A nonnegative symmetric real-valued function $\sigma$ on $[0,1]^{2}$
is piecewise continuous if there exists an integer number $p$ and
a collection of intervals $I_{1},\dots I_{p}\subset[0,1]$ that forms
a partition of $[0,1]$ such that $\sigma$ is continuous on each
$I_{k}\times I_{l}$ and can be extended as a continuous function
on $\overline{I_{k}}\times\overline{I_{l}}$ for every $k,l\in\{1,\dots,p\}$.
We say that $\Sigma_{i,j}^{N}$ is \emph{asymptotically piecewise
continuous} iff there exists a sequence $0\le t_{1}^{N}<t_{2}^{N}\le\cdots<t_{N}^{N}\le1$
such that $\frac{1}{N}\sum_{i=1}^{N}\delta_{t_{i}^{N}}$ converges weakly towards
the uniform law on $[0,1]$ and a piecewise continuous function $\sigma$
such that
\begin{equation}\label{convcov}
\lim_{N\rightarrow\infty}\sup_{1\leq i,j\leq N}|\Sigma_{i,j}^{N}-\sigma(t_{i}^{N},t_{j}^{N})|=0\,.
\end{equation}
\end{itemize}
In the next section, we recall known results about the convergence of the spectrum and the largest eigenvalue of $H^{N}$. 
\subsection{Limit of the empirical measure of the eigenvalues of $H^{N}$}\label{convsec}

For a symmetric real matrix $M$, we will denote in the rest of this
paper $\lambda_{1}(M)\geq\dots\geq\lambda_{N}(M)$ its ordered eigenvalues, {
$v_{1}(M),\dots, v_{N}(M)$ the associated eigenvectors and $\mu_{M}$
the empirical measure of its eigenvalues defined by } 
\[
\mu_{M}=\frac{1}{N}\sum_{i=1}^{N}\delta_{\lambda_{i}(M)}.
\]
In the case where $M=H^{N}$, we will often just use  the notation $\lambda_{i},v_{i}$
instead of $\lambda_{i}(H^{N}),v_{i}(H^{N})$. For any measure $\nu$
on $\mathbb{R}$, we denote by $G_{\nu}$ its Stieljes transform given by
as
\[
G_{\nu}(z):=\int_{\mathbb{R}}\frac{1}{z-y}d\nu(y)
\]
for any $z\in\mathbb{C}\setminus\text{Supp}(\nu)$.

The spectral measure $\mu_{H^{N}}$ does not converge to the semi-circle
law in general but to some  measure $\mu_{\sigma}$ that depends
on $\sigma$, see \cite{AlErKr,Sh98}, see \cite[Theorem 1.4]{HussonVar}. We denote by $r_{\sigma}$ the right hand point of the support of $\mu_{\sigma}$. 
We denote by $d_{W}$ the Wassertein distance on $\mathcal{P}(\mathbb{R})$
the space of probability measures on $\mathbb{R}$: 
\[
d_{W}(\mu,\nu)=\sup_{\|f\|_{L}\le1}|\int f(d\mu-d\nu)|
\]
 with $\|f\|_{L}$ the Lipschitz constant. We will also use this distance
for non negative measures with different masses. We have the following. 
\begin{lem}
\label{lem:Measure-Concentration} There exists a family of probability
measures $(\mu_{\sigma,t})_{t\in[0,1]}$ supported on a compact interval
$[l_{\sigma},r_{\sigma}]\subset\mathbb{R}$ so that if we set 
\[
\mu_{\sigma}:=\int_{0}^{1}\mu_{\sigma,t}dt,
\]
we have 
\begin{enumerate}
\item \label{enu:(Dyson-equation)}(Dyson equation)  For all $z\in\mathbb{C},$
$\Im z>0$
\begin{equation}
\frac{1}{G_{\mu_{\sigma,t}}(z)}=z-\int_{0}^{1}\sigma(t,s)G_{\mu_{\sigma,s}}(z)ds\label{eq:Dyson}
\end{equation}
for Lebesgue-almost all $t\in[0,1]$.
\item \label{enu:(Weak-convergence)}(Weak convergence of the specral measure)
For every $\delta>0$,
\[
\limsup_{N\rightarrow\infty}\frac{1}{N}\log\mathbb{P}(d_{W}(\mu_{\sigma},\mu_{H^{N}})\geq\delta)=-\infty.
\]
\item (Convergence of the largest eigenvalue) \label{enu:edge}For every $\delta>0$,
\[
\lim_{N\rightarrow\infty}\mathbb{P}(|\lambda_{1}(H^{N})-r_{\sigma}|\geq\delta)=0.
\]
\end{enumerate}
\end{lem}

The first part of this theorem is basically a consequence of Theorem
1.1 in \cite{girko2012} which states that the empirical measure is close to the solution of the corresponding finite $N$ Dyson equation, and proofs are given in \cite{AEK2,Alt18,AlErKr, AEK1,AEK3, CEKS1,CEKS2,HussonVar}
For completness,{ we give  in Proposition \ref{prop:muconv}  a proof that the solution of this finite $N$ Dyson equation converges to the continuous version of this equation.}
Remark that the empirical mesure $\mu_{H^{N}}$ concentrates
with speed larger than $N$ according to Lemma \ref{lem:Measure-Concentration} (2) by \cite{GZ00}. The convergence of the largest eigenvalue stated in   Lemma  \ref{lem:Measure-Concentration} (3)  is known but because we could not find a self-contained
proof in the literature, we  give a full proof in Appendix
\ref{convlambda} for completeness.

\subsection{Main results}

\label{subsec:Main-results}
The main result of our article  is a large deviation principle for the largest eigenvalue of $H^{N}$.  To state this result we first define the rate function of the large deviation principle.
We denote  by $S$ the following bilinear form on $\mathcal{P}([0,1])$ :
\[
\langle\phi,S\phi\rangle=\int_{0}^{1}\int_{0}^{1}\sigma(s,t)d\phi(s)d\phi(t).
\]
For a probability measure $\psi$ on $[0,1]$ which is absolutely continuous
with respect to  Lebesgue measure,   a non-negative real number $\hat\theta$ and a real number $x$ larger than $r_{\sigma}$, we set
\begin{equation*}
	\widehat{\mathcal{F}}( \widehat{\theta},x,\psi,\sigma) := \widehat{\theta}(x-\langle\varphi_{\star}(x) dt ,S\psi\rangle)-\widehat{\theta}^{2}\langle\psi ,S\psi\rangle-\frac{1}{2}\int_{0}^{1}\log\left(\frac{2\widehat{\theta}}{\varphi_{\star}(x)}\frac{d\psi}{dt}+1\right)dt
\end{equation*} 
where $\varphi_{\star}(x)$ is the function $t \mapsto G_{\mu_{\sigma,t}}(x)  $. Using the Dyson equation satisfied by  $\varphi_*$, we also have 

\begin{equation}\label{defF2}
	\widehat{\mathcal{F}}( \widehat{\theta},x,\psi,\sigma) := \widehat{\theta}\langle\frac{dt}{\varphi_{\star}(x)},\psi\rangle-\widehat{\theta}^{2}\langle\psi ,S\psi\rangle-\frac{1}{2}\int_{0}^{1}\log\left(\frac{2\widehat{\theta}}{\varphi_{\star}(x)}\frac{d\psi}{dt}+1\right)dt.
\end{equation} 


We can now state the main result of this article, where we recall that $\mu_{\sigma}$ is the almost sure limit of $\mu_{H^{N}}$ and $r_{\sigma}$,  the right boundary of its support, the limit of the largest eigenvalue of $H^{N}$.
\begin{thm}
\label{maintheo} The law of $\lambda_{1}(H^{N})$ satisfies a large
deviation principle with speed $N$ and rate function $I_{\sigma}$
which is infinite if $x<r_{\sigma},$ $I_{\sigma}(r_{\sigma})=0$
and is  given for $x>r_{\sigma}$ by

\[ I_{\sigma}(x)=\inf_{\substack{\psi\in{\cal P}([0,1])\\
		}
	}\sup_{\widehat{\theta}\geq0} \widehat{\mathcal{F}}( \widehat{\theta},x,\psi,\sigma) \]


 Moreover $I_{\sigma}$ is a good rate function,
namely its level sets are compact, and it is increasing. 
\end{thm}

Notice that the rate function only depends on the variance profile
$\text{\ensuremath{\sigma}.}$ In particular we have that the large
deviation principle is the same as for a random Gaussian matrix with
the same variance profile. This result is also new in the Gaussian
case as soon as $\sigma$ does not satisfy the conditions of \cite{HussonVar}
for instance in the case where $\sigma$ is block diagonal.  

In the rest of the paper however, we will use another expression of our rate function that is more suited to our needs but slightly more involved. 
In order to do so, we first introduce a few notations. 
Let $\psi$ be an element of the set $\mathcal{P}([0,1])$ of probability measures on $[0,1]$, $x$ be a real number larger than $ r_{\sigma}$ and $\theta$ be a non-negative real number.
We define two non negative measures $\varphi,\phi$ on $[0,1]$ by
\[
d\varphi(\theta,x)(t):=\frac{1}{2\theta}G_{\mu_{\sigma,t}}(G_{\mu_{\sigma}}^{-1}(2\theta)\vee x)dt,
\]
\begin{equation}
	d\phi(\theta,x,\psi)(t):=d\varphi(\theta,x)(t)+\Big(1-\frac{G_{\mu_{\sigma}}(x)}{2\theta}\Big)_{+}d\psi(t).\label{eqn:phicontinuous-1}
\end{equation}
Notice that $\phi(\theta,x,\psi)\in\mathcal{P}([0,1])$. We also introduce
the function $J_{\mu_{\sigma}}(x,\theta)$ given by 
\begin{equation}
	\hspace{-0.185cm}
	\begin{cases}
		\theta x-\frac{1}{2}-\frac{1}{2}\log2\theta-\frac{1}{2}\int_{\mathbb{R}}\log|x-y|d\mu_{\sigma}(y) & \text{if }2\theta\geq G_{\mu_{\sigma}}(x),\\
		\theta G_{\mu_{\sigma}}^{-1}(2\theta)-\frac{1}{2}-\frac{1}{2}\log2\theta-\frac{1}{2}\int_{\mathbb{R}}\log(G_{\mu_{\sigma}}^{-1}(2\theta)-y)d\mu_{\sigma}(y) & \text{if }2\theta<G_{\mu_{\sigma}}(x).
	\end{cases}\label{defJ-1}
\end{equation}

We then define
\begin{equation}\label{defKH}
K(\theta,\phi):=\theta^{2}\int_{0}^{1}\int_{0}^{1}\sigma(s,t)d\phi(s)d\phi(t)- H(\phi),\,
H(\phi):=- \frac{1}{2}\int_{0}^{1}\log\frac{d\phi}{dt}dt\,.
\end{equation}

Then we have the following equality, which is proven in Appendix \ref{equalityRF}:

\begin{lem}\label{lem:equalityRF}
	
	For $x > r_{\sigma}$:
\[
I_{\sigma}(x)=\inf_{\substack{\psi\in{\cal P}([0,1])\\
		\langle\psi,S\psi\rangle\neq0
		}
	}\sup_{\theta\geq0}\{J_{\mu_{\sigma}}(x,\theta)-K(\theta,\phi(\theta,x,\psi))\}\,.
\]
	\end{lem}

We easily retrieve the well known large deviation principle  when $\sigma=\mathds{1}_{[0,1]^{2}}$ in which case $K(\theta,\phi)=\theta^{2}$, as studied in \cite{HuGu1}. In the case where $\Sigma_{N}$ is block diagonal, we find that our fomula "decouples"  as would be expected, see Proposition \ref{prop:block}.
Furthermore, in the Wishart case, we also recover the already known large deviation principle in this case (see Theorem 1.6 and 1.7 in \cite{HuGu1}). We prove that our formula for $I_{\sigma}$ does match the one in those two results in Appendix \ref{sec:Wishart}. In the more  general case where $\psi\rightarrow \langle \psi,S\psi\rangle$ is concave, we show in Lemma \ref{lem:concave} that the infimum over $\psi$ and the maximum over $\theta$ can be exchanged by Sion's minmax theorem. One can work out the resulting supremum over $K$ to recover the result of \cite{HussonVar}, modulo the use of some arguments from that paper. 
  
\subsection{Piecewise constant case}
\label{subsec:Piecewise-constant-notation} 

We first work with a stronger assumption on the model and we will
assume that $\Sigma_{i,j}^{N}$ is\emph{ piecewise constant}. This
will imply the general theorem by approximating  the  piecewise continuous function $\sigma$
 by piecewise constant functions and constructing exponential approximations of our problem, see Section \ref{sec:approximation}.
More precisely, we say that   $\Sigma_{i,j}^{N}$ is \emph{piecewise constant}
if there exists a collection of intervals $I_{1},\dots I_{p}\subset[0,1]$ that forms
a partition of $[0,1]$,  a $p\times p$ matrix   $\sigma\in\mathbb{R}^{p\times p}$ and 
 a sequence $0\le t_{1}^{N}<t_{2}^{N}\le\cdots<t_{N}^{N}\le1$ such that $\frac{1}{N}\sum_{i=1}^{N}\delta_{t_{i}^{N}}$ converges weakly towards
the uniform law on $[0,1]$ and 
 so that if we set
$I_{k}^{N}:=\{i\in\{1,\cdots,N\}:t_i^N \in I_{k}\}$,
$$ \Sigma_{i,j}^{N}=\sigma_{k,l}\qquad \forall i\in I_{k}^{N},
j\in I_{l}^{N}\,.$$
Observe that $\Sigma_{i,j}^{N}$ converges in the sense of \eqref{convcov} towards 
the asymptotic variance profile $\tilde{\sigma}$ given by
\[
\tilde{\sigma}(s,t)=\sum_{k,l=1}^{p}\mathds{1}_{I_{k}}(s)\mathds{1}_{I_{l}}(t)\sigma_{k,l}.
\]
In this setting all the mathematical objects look quite simpler. We
have a canonical bijection that associates a vector $\psi\in\mathbb{R}^{p}$,
resp. a family of measure $(\nu_{k})_{k\leq p}$ and a family of complex
function $(g_{k}(z))_{k\leq p}$, to a function $\tilde{\psi}\in\mathbb{R}^{[0,1]}$,
resp. a family of measure $(\tilde{\nu}_{t})_{t\in[0,1]}$ and a family
of complex function $(\tilde{g}_{t}(z))_{t\in[0,1]}$, that is constant
on each interval $I_{k}$ :

\[
d\tilde{\psi}(t)=\sum_{k=1}^{p}\mathds{1}_{I_{k}}(t)\frac{\psi_{k}}{|I_{k}|}dt,\quad\tilde{\nu}_{t}=\sum_{k=1}^{p}\mathds{1}_{I_{k}}(t)\frac{\nu_{k}}{|I_{k}|},\quad\tilde{g}_{t}(z)=\sum_{k=1}^{p}\mathds{1}_{I_{k}}(t)\frac{g_{k}(z)}{|I_{k}|}
\]
The factors $|I_{k}|^{-1}$ have been added here by convention so
that the measure  $\tilde{\nu}$ belongs to ${\cal P}([0,1])$ if and only if  $\nu$ is a vector in $\{\mathbb{R}_{+}^{p}$ such that $\sum\nu_{k}=1\}$, namely $\nu$ is an element of the set ${\cal P}_{p}$
 of probability measures on $\{1,\cdots,p\}$. Lemma \ref{lem:Measure-Concentration}
specifies in this setting as follows : we have a finite family of
non negative measures $(\mu_{\sigma,k})_{k\leq p}$ with mass $(|I_{k}|)_{k\leq p}$,
which satisfies the following Dyson system of equations
\[
\frac{|I_{k}|}{G_{\mu_{\sigma,k}}(z)}=z-\sum_{l=1}^{p}\sigma_{k,l}G_{\mu_{\sigma,l}}(z).
\]
for all $k\in\{1,\cdots,p\}$. We set $\mu_{\sigma}=\sum_{k=1}^{p}\mu_{\sigma,k}$.
Similarly we define for $x,\theta\in\mathbb{R}_{+}$, $\psi\in{\cal P}_{p}$
\begin{equation}
\varphi(\theta,x)_{k}=\frac{G_{\mu_{\sigma,k}}(G_{\mu_{\sigma}}^{-1}(2\theta)\vee x)}{2\theta},\,\phi(\theta,x,\psi)_{k}=\varphi(\theta,x)_{k}+\Big(1-\frac{G_{\mu_{\sigma}}(x)}{2\theta}\Big)_{+}\psi_{k},\label{eq:Def_Phi_theta_bis}
\end{equation}
$\langle\phi,S\phi\rangle=\sum_{k,l=1}^{p}\sigma_{k,l}\phi_{k}\phi_{l}$
and the function

\begin{equation}
K(\theta,\phi)=\theta^{2}\sum_{k,l=1}^{p}\sigma_{k,l}\phi_{k}\phi_{l} + \frac{1}{2}\sum_{k=1}^{p}|I_{k}|\log\frac{\phi_{k}}{|I_{k}|}.\label{eq:Def_K_bis}
\end{equation}
In the rest of the paper with the exception of Section \ref{sec:Piecewise-continuous}
we work in the piecewise constant case and use the above definitions
instead of the ones from Section \ref{subsec:Main-results}. For $\psi\in{\cal P}_{p}$ 
\begin{equation}\label{defF0}
\mathcal{F}(\theta,x,\psi,\sigma):=J_{\mu_{\sigma}}(x,\theta)-K(\theta,\phi(\theta,x,\psi)).
\end{equation} 
and we set  
$$I_{\sigma}(x)=\inf_{\substack{\psi\in{\cal P}_{p}\\
\langle\psi,S\psi\rangle\neq0
}
}\sup_{\theta\geq0}\mathcal{F}(\theta,x,\psi,\sigma)\,.$$
It is easily seen that this definition corresponds to the previous one when $\sigma$ is piecewise constant, see Section \ref{ratef}. 
We set
$$P_{\delta}^{N}(x):=\frac{1}{N}\log\mathbb{P}(\{|\lambda_{1}-x|\leq\delta\})\,.$$
We will prove the following result: 
\begin{thm}
\label{maintheo-constant} In the piecewise constant case, we have
the weak large deviation principle : for $x>r_{\sigma}$ 

\[
\lim_{\delta\rightarrow0}\liminf_{N\rightarrow\infty}P_{\delta}^{N}(x)=\lim_{\delta\rightarrow0}\limsup_{N\rightarrow\infty}P_{\delta}^{N}(x)=-I_{\sigma}(x)\,.
\]
\end{thm}
We prove this Theorem in Sections \ref{sec:LDupperbound} and \ref{sec:LDlowerbound}.

\subsection{Outline of the paper}

The main tool we use in this paper is the spherical integral defined,
for every real number $\theta$ and every symmetric $N\times N$ matrix
$M$, by 
\[
I_{N}(M,\theta)=\int_{\mathbb{S}^{N-1}}e^{\theta N\langle u,Mu\rangle}du
\]
where the integral is taken over the uniform law on the $N$ dimension
sphere $\mathbb{S}^{N-1}$. We see $u$ as a random vector in $\mathbb{S}^{N-1}$
and also use the notation $\mathbb{E}_{u}[.]:=\int_{\mathbb{S}^{N-1}}.du$.
The strategy of the proof is to tilt the measures by the spherical integral as in \cite{HuGu1,HuGu2,HussonVar}
but to localize them and to compute their expectation by using the Fubini trick.
\begin{equation}
\mathbb{E}_{H^{N}}\mathbb{E}_{u}(1_{\Theta}e^{\theta N\langle u,H^{N}u\rangle})=\mathbb{E}_{u}\mathbb{E}_{H^{N}}(1_{\Theta}e^{\theta N\langle u,H^{N}u\rangle})\label{eq:Fubini_trick}
\end{equation}
for some particular event $\Theta$. The following Lemma can be used
to estimate the left hand side of \eqref{eq:Fubini_trick}. It was
shown in \cite{Ma07,GuMa05} that 
\begin{lem}
\label{lem:J_spherical_integral}For a sequence of $N\times N$ self-adjoint deterministic  matrices $M^{N}$
so that $\mu_{M^{N}}$ converges towards $\nu$, $\sup_{N}\|M^{N}\|<\infty$
and whose largest eigenvalue $\lambda_{1}(M^{N})$  converges towards $y$ we have 
\[
\lim_{N\rightarrow\infty}\frac{1}{N}\log I_{N}(M^{N},\theta)=J_{\nu}(y,\theta).
\]
Moreover this convergence is uniform on small neighborhoods of $\mu_{M^{N}}$
for the weak topology and $\lambda_{1}(M^{N})$.
\end{lem}

For $1\le k\le p$ , we denote by $\Pi_{k}$ the orthogonal projection
on the vector space generated by the elements of the canonical basis
indexed by $I_{k}^{N}=\{i\in\{1,\cdots,N\}:t_i^N \in I_{k}\}$.
For any $v\in\mathbb{R}^{N}$, 
\begin{equation}
(\Pi_{k}v)_{i}:=v_{i}\mathds{1}_{i\in I_{k}^{N}}.\label{defpik}
\end{equation}
and we define for $v\in\mathbb{R}^{N}$ the associated profile $\rho(v)\in(\mathbb{R}_{+})^{p}$
as 
\begin{equation}
\rho(v)_{k}:=\|\Pi_{k}v\|^{2}=\sum_{j\in I_{k}^{N}}v_{j}^{2}.\label{eq:def_rho(u)}
\end{equation}
Remark that for $u\in\mathbb{S}^{N-1}$ we have $\rho(u)\in{\cal P}_{p}$
. For the right hand side of \eqref{eq:Fubini_trick} we will use
\begin{prop}
\label{ldpv-1-1} In the piecewise constant case, for any vector $\phi\in\mathcal{P}_{p}$
we have {
\begin{eqnarray*}
K(\theta,\phi)&=&\lim_{\delta\rightarrow0}\limsup_{N\rightarrow\infty}\frac{1}{N}\log\mathbb{E}_{u}\mathbb{E}_{H^{N}}(e^{N\theta\langle u,H^{N}u\rangle}\mathds{1}_{\{||\rho(u)-\phi||\leq\delta\}})\\
&=&\lim_{\delta\rightarrow0}\liminf_{N\rightarrow\infty}\frac{1}{N}\log\mathbb{E}_{u}\mathbb{E}_{H^{N}}(e^{N\theta\langle u,H^{N}u\rangle}\mathds{1}_{\{||\rho(u)-\phi||\leq\delta\}})\end{eqnarray*}}.
\end{prop}

To use both Lemma \ref{lem:J_spherical_integral} and Proposition
\ref{ldpv-1-1}, the new key idea is to consider jointly the largest
eigenvalue $\lambda_{1}(H^{N})$ and the profile of the eigenvector
$v_{1}(H^{N})$. We define for any $\psi\in\mathcal{P}_{p}$ and $\delta>0$
\[
P_{\delta}^{N}(x,\psi):=\frac{1}{N}\log\mathbb{P}(\{|\lambda_{1}-x|\leq\delta\}\cap\{\|\rho(v_{1})-\psi\|_{1}\leq\delta\})
\]
with $\|.\|_{1}$ the $L^{1}$ norm on $\mathbb{R}^{p}$. The motivation
to introduce $P_{\delta}^{N}(x,\psi)$ is the following remark : Because
$\mathcal{P}_{p}$ is compact, it can be covered by a finite union
of open balls of radius $\delta$. As a consequence,  we have :
\begin{lem}
\label{lem:zero_entropy} For any $\delta>0$, there exists a finite
set ${\cal J}\subset{\cal P}_{p}$ such that 
\[
P_{\delta}^{N}(x)=\sup_{\psi\in{\cal J}}P_{\delta}^{N}(x,\psi)+\eta_{\delta}(N)
\]
where $\eta_{\delta}(N)$ goes to $0$ as $N$ goes to infinity.
\end{lem}

The proof of Theorem \ref{maintheo-constant} is then divided into
the proof of the upper bound and the lower bound.
\begin{prop}
\label{prop:LDPUB2} (Theorem \ref{maintheo-constant} upper bound)
In the piecewise constant case, for every $x>r_{\sigma}$ and $\psi\in\mathcal{P}_{p}$,

\[
\lim_{\delta\to0}\limsup_{N\rightarrow\infty}P_{\delta}^{N}(x,\psi)\leq\inf_{\theta\geq0}-\mathcal{F}(\theta,x,\psi,\sigma).
\]
\end{prop}

\begin{prop}
\label{prop:LDPLB} (Theorem \ref{maintheo-constant} lower bound)
In the piecewise constant case, for every $x>r_{\sigma}$,

\[
\lim_{\delta\to0}\liminf_{N\rightarrow\infty}P_{\delta}^{N}(x)\geq\sup_{\psi\in{\cal P}_{p}\atop
\langle \psi,S\psi\rangle \neq 0}\inf_{\theta\geq0}-\mathcal{F}(\theta,x,\psi,\sigma)
\]
\end{prop}

Proposition \ref{prop:LDPUB2} and Lemma \ref{lem:zero_entropy} will
be enough to conclude the upper bound in Theorem \ref{maintheo-constant} { (together with Section \ref{seceps} where it is shown that the $\psi$ which optimize the upper bound are such that $\langle \psi,S\psi\rangle \neq 0$)}
and then Proposition \ref{prop:LDPLB} will give the complementary
large deviation lower bound. This finishes the proof of Theorem \ref{maintheo-constant}.
We then obtain Theorem \ref{maintheo} in the piecewise continuous
case approximating the variance profile by piecewise constant functions
and we will prove that the rate functions of the approximations converge
toward the actual definition of Theorem \ref{maintheo} thus concluding
the proof.

In Section \ref{sec:Annealed_spherical} we prove Proposition \ref{ldpv-1-1}
that gives an asymptotic formula for the annealed spherical integral.
In Section \ref{sec:LDupperbound} we prove the lower bound of Theorem
\ref{maintheo-constant} stated as in Proposition \ref{prop:LDPUB2}.
Section \ref{sec:LDlowerbound} is devoted to the proof of the lower
bound of Theorem \ref{maintheo-constant} stated as in Proposition
\ref{prop:LDPLB}. We check that the proposed function $I_{\sigma}$
is a good rate function in Section \ref{sec:rate}. The generalization
from piecewise constant to piecewise continuous is done in Section
\ref{sec:Piecewise-continuous}.

\subsection{Notations}
We recall the following notations used throughout this paper:
\begin{itemize}
\item $N$ the size of the matrices, $p$ the number of intervals for the piecewise
continuous/constant functions.
\item $H_{ij}^{N}=\frac{1}{\sqrt{N}}X_{ij}^{N}$ the random matrix of size
$N\times N$. 
\item $\Sigma_{ij}^{N}=2^{-1_{i=j}}\mathbb{E}[(X_{ij}^{N})^{2}]$ the variance profile
of $H^{N}$ and  $A=\sup_{N}\max_{i,j}\Sigma_{ij}^{N}$.  $\sigma$ denotes  the limiting
variance profile. it is a matrix $\sigma_{k,l}$ in the  piecewise constant case, and a function $\sigma(s,t)$ on $[0,1]^{2}$
in the continuous case. 
\item $I_{1},\cdots,I_{p}$ are  intervals that form a partition of $[0,1]$,
$|I_{k}|$ for the length of the interval $0<t_{1}^{N} <t^{N}_{2}<\cdots <t_{N}^{N}<1$ are sequences whose empirical distribution converges towards the uniform law on $[0,1]$ and we set $I_{k}^{N}:=\{i: t_{i}^{N}\in I_{k}\}$. 
\item $M$ is
a symmetric matrix in $\mathbb{R}^{N\times N}$,
$\lambda_{1}(M)\geq\cdots\geq\lambda_{N}(M)$ the eigenvalues of $M$
and $v_{1}(M),\cdots,v_{N}(M)$ the associated eigenvectors. \item $\mu_{M}$ the empirical measure of  the eigenvalues of $M$. 
\item $G_{\nu}(z)$ the Stieljes transform of a measure $\nu$, $G_{\mu_{\sigma}}^{-1}$
its inverse. 
\item ${\cal P}([0,1])$ the set of probability measures on $[0,1]$
and ${\cal P}_{p}={\cal P}(\{1,\cdots,p\})$. 
\item $\mu_{\sigma}$, $\mu_{\sigma,t}$ and $\mu_{\sigma,k}$ the asymptotic
empirical measures. $[l_{\sigma},r_{\sigma}]$ the support of $\mu_{\sigma}$
.
\item $u$ and element of the $N$-dimensional sphere $\mathbb{S}^{N-1}$, $\int_{\mathbb{S}^{N-1}}...du$ and  $\mathbb{E}_{u}$
the integral on $\mathbb{S}^{N-1}$. $\mathbb{E}_{H}\mathbb{E}_{u}$ or $\mathbb{E}_{H,u}$ when we integrate both on the matrix $H$ and $u$. 
\item $I(M,\theta)$ the spherical integral and  $J_{\mu}(x,\theta)$ the limit of the spherical integral. 
\item $I_{\sigma}(x)$ the large deviation rate function,  $K(\theta,\psi)$ the limit of the annealed integral. 
\item ${\cal F}(\theta,x,\psi,\sigma)=J_{\mu_{\sigma}}(x,\theta)-K(\theta,\phi(\theta,x,\psi))$ 
\item $\rho(u)\in\mathbb{R}_{+}^{p}$ for the profile of $u$.
\item $P_{\delta}^{N}(x)$ the probability for the largest eigenvalue  to be localized near $x$ and
$P_{\delta}^{N}(x,\psi)$  the probability  for both the eigenvalue  to be localized near $x$  and the profile of the
eigenvector to be localized near $\psi$. 
\item $\Pi$ general orthogonal projection, $\Pi_{k}$ the orthogonal   projection onto the vector space
$\{ e_{i}, i\in I_{k}^{N}\}$  where $e_{i}$ is the canonical basis.
\end{itemize}

\section{Convergence of the annealed spherical integral, Proof of Proposition
\ref{ldpv-1-1}}

\label{sec:Annealed_spherical}We next prove the convergence of the
annealed restricted spherical integral of Lemma \ref{ldpv-1-1}. For
$u\in\mathbb{S}^{N-1}$, we recall the definition of  the profile $\rho(u)\in{\cal P}_{p}$
from \eqref{eq:def_rho(u)} and introduce the family of vectors $\frak{u}^{k}\in\mathbb{R}^{I_{k}^{N}}$
for $1\leq k\leq p$ 
\[
\frak{u}^{k}:=(u_{i}/\sqrt{\rho(u)_{k}})_{i\in I_{k}^{N}}.
\]

It is straightforward to check the following Lemma :
\begin{lem}
\label{lem:Uniform_sphere_family}For $u$ following the uniform law
on $\mathbb{S}^{N-1}$, the random vector $\rho(u)$ follows a Dirichlet
law of parameters $(\#I_{1}^{N},\dots,\#I_{p}^{N})/2$ and $(\frak{u}^{k})_{1\leq k\leq p}$
form a family of independent random uniform unit vectors that are
independent of $\rho(u)$.
\end{lem}

Moreover we have the following classical large deviation principle.

\begin{lem}
\label{lem:rho_LD}For $u$ with the uniform law on $\mathbb{S}^{N-1}$,
the law of $\rho(u)$ satisfies a large deviation principle with speed
$N$ and rate function given by the Kullback--Leibler divergence
:
\[
d_{KL}(|I|\|\phi):=-\frac{1}{2}\sum_{k=1}^{p}|I_{k}|(\log(\phi_{k})-\log(|I_{k}|)).
\]
\end{lem}

\begin{proof}
[Proof of Lemma \ref{lem:rho_LD}]Because $\lim_{N\to\infty}\#I_{k}^{N}/N=|I_{k}|$ since the empirical measure of the $t_{i}^{N}$ converges toward the uniform law, 
it follows from \cite[Theorem 11]{HuGu2}.
\end{proof}

We denote, for $1\leq i\leq j\leq N$, $L_{ij}^{N}(t)=\log\mathbb{E}[e^{tX_{ij}^{N}}]$
and we have 
\begin{equation}
\mathbb{E}_{H}[\exp(\theta N\langle u,H^{N}u\rangle)]=\exp\Big(\sum_{1\leq i<j\leq N}L_{ij}^{N}(2\sqrt{N}\theta u_{i}u_{j})+\sum_{i=1}^{N}L_{ii}^{N}(\sqrt{N}\theta u_{i}^{2})\Big)\label{eq:Annealed_H-1}.
\end{equation}

\begin{proof}
[Proof of Proposition \ref{ldpv-1-1}] For the sake of this theorem,
we can replace the norm $\|.\|_{L^{1}}$ by the uniform norm $\|.\|_{\infty}$.
By norm equivalence, the result remains unchanged. Then, using our
sharp sub-Gaussian assumption \eqref{SSH} in  \eqref{eq:Annealed_H-1}, we find that : 
\begin{eqnarray*}
\mathbb{E}_{H}[\exp(\theta N\langle u,H^{N}u\rangle)] & \leq & \exp\Big(N\theta^{2}\Big(2\sum_{1\leq i<j\leq N}u_{i}^{2}u_{j}^{2}\Sigma_{i,j}^{N}+\sum_{i=1}^{N}u_{i}^{4}\Sigma_{i,i}^{N}\Big)\Big)\\
 & = & \exp\Big(N\theta^{2}\sum_{k,l=1}^{p}\rho(u)_{k}\rho(u)_{l}\sigma_{k,l}\Big).
\end{eqnarray*}
Therefore, with Lemma \ref{lem:rho_LD}, since $\{\|\rho(u)-\psi\|_{\infty}\leq\delta\}=\bigcap_{i=1}^{p}\{|\rho(u)_{i}-\psi_{i}|\leq\delta\}$,
we have by Varadhan's lemma that 
\[
\limsup_{N\to\infty}\frac{1}{N}\log\mathbb{E}_{u}[\prod_{1\le i\le p}\mathds{1}_{\{|\rho_{i}(u)-\psi_{i}|\leq\delta\}}\exp\Big(N\theta^{2}\sum_{k,l=1}^{p}\rho(u)_{k}\rho(u)_{l}\sigma_{k,l}\Big)]\qquad\qquad
\]
\[
\qquad\qquad\leq\theta^{2}\sum_{k,l=1}^{p}\psi_{k}\psi_{l}\sigma_{k,l} { + } \frac{1}{2}\sum_{i=1}^{p}|I_{k}|(\log(\psi_{i})-\log(|I_{i}|))+O(\delta)=K(\theta,\psi)+O(\delta)
\]
where $O(\delta)$ goes to zero when $\delta$ goes to zero. 
Therefore, letting $\delta$ going to zero, we conclude that

\[
\lim_{\delta\rightarrow0}\limsup_{N\rightarrow\infty}\frac{1}{N}\log\mathbb{E}_{u}\mathbb{E}_{H}(e^{N\theta\langle u,H^{N}u\rangle}\mathds{1}_{\{\|\psi-\rho(u)\|\leq\delta\}})\leq K(\theta,\psi).
\]

For the lower bound, we use Lemma \ref{lem:Uniform_sphere_family}.
We denote,  for $1\leq k\leq p$, $E_{k}:=\{\forall i\in I_{k}^{N},|\frak{u}_{i}^{k}|\leq N^{-3/8}\}$
and $E=\cap_{k=1}^{p}E_{k}$. We get the following lower bound:
\begin{eqnarray*}
\mathbb{E}_{u}\mathbb{E}_{H}(e^{N\theta\langle u,H^{N}u\rangle}\mathds{1}_{\{\|\psi-\rho(u)\|\leq\delta\}}) & \geq & \mathbb{E}_{u}\mathbb{E}_{H}(e^{N\theta\langle u,H^{N}u\rangle}\mathds{1}_{\{\|\psi-\rho(u)||\leq\delta\}\cap E}).
\end{eqnarray*}
Let $\epsilon>0$. Using the sub-Gaussian character of the $H_{ij}^{N}$,
and the fact that the $\Sigma_{i,j}^{N}$ are bounded, one can find
$\epsilon'>0$ such that for $|t|\leq\epsilon'$, $L_{i,j}^{N}(t)\geq(1-\epsilon)\Sigma_{i,j}^{N}t^{2}/2$.
Therefore, on the event $E$, $2\sqrt{N}\theta|u_{i}||u_{j}|\leq2\theta\sqrt{N}(N^{-3/8})^{2}\leq\epsilon'$
for $N$ large enough and we have by \eqref{eq:Annealed_H-1} : 
\begin{eqnarray*}
\mathbb{E}_{H}[\exp(\theta N\langle u,H^{N}u\rangle)] & \geq & \exp\Big(N(1-\epsilon)^{2}\theta^{2}\Big(\sum_{k,l=1}^{p}\rho(u)_{k}\rho(u)_{l}\sigma_{k,l}\Big)\Big)\,.
\end{eqnarray*}
As a consequence, one can write: 
\begin{eqnarray*}
 &  & \mathbb{E}_{u}\mathbb{E}_{H}(e^{N\theta\langle u,H^{N}u\rangle}\mathds{1}_{\{\|\psi-\rho(u)\|\leq\delta\}})\\
 & \geq & \mathbb{E}_{u}\left[\exp\Big(N(1-\epsilon)\theta^{2}\Big(\sum_{k,l=1}^{p}\sigma_{k,l}\rho(u)_{k}\rho(u)_{l}\Big)\Big)\mathds{1}_{\{\|\psi-\rho(u)\|\leq\delta\}\cap E}\right]\\
 & \geq & \mathbb{P}_{u}[E\cap\{\|\psi-\rho(u)\|\leq\delta\}]\exp\Big(N(1-\epsilon)\theta^{2}\Big(\sum_{k,l=1}^{n}\sigma_{k,l}\psi_{k}\psi_{l}\Big)\Big)\exp(-N\eta(\delta))
\end{eqnarray*}
with $\eta(\delta)\rightarrow0$ as $\delta\rightarrow0$. Furthermore,
one can see that since $\mathbb{P}_{u}[E]=\prod_{k=1}^{p}\mathbb{P}_{u}[E_{k}]$
because the $\frak{u}^{k}$ are independent and $\lim_{N\to\infty}N^{-1}\log\mathbb{P}_{u}[E_{k}]=0$
for all $k$ (see e.g. \cite[Lemma 3.3]{HuGu1}) then 
\[
\lim_{N\to\infty}\frac{1}{N}\log\mathbb{P}_{u}[E]=0.
\]
Therefore, since $\rho$ is independent from $E$, we have by Lemma
\ref{lem:rho_LD} that for any $\epsilon>0$,

\[
\lim_{\delta\to0}\liminf_{N\to\infty}\frac{1}{N}\log\mathbb{E}_{u}\mathbb{E}_{H}(e^{N\theta\langle u,H^{N}u\rangle}\mathds{1}_{\{||\psi-\rho(u)||\leq\delta\}})\qquad
\]
\[
\geq(1-\epsilon)\theta^{2}\sum_{k,l=1}^{p}\psi_{k}\psi_{l}\sigma_{k,l}+ \frac{1}{2}\sum_{i=1}^{p}|I_{i}|(\log(\psi_{i})-\log(|I_{i}|)).
\]
Since this is true for any $\epsilon>0$, we conclude by taking $\epsilon$ to zero that
\[
\lim_{\delta\to0}\liminf_{N\to\infty}\frac{1}{N}\log\mathbb{E}_{u}\mathbb{E}_{H}(e^{N\theta\langle u,H^{N}u\rangle}\mathds{1}_{\{\|\psi-\rho(u)\|\leq\delta\}})\geq K(\theta,\psi)\,.
\]
\end{proof}

\section{Proof of Theorem \ref{maintheo-constant}, Large deviation upper
bound}

\label{sec:LDupperbound}
As
mentioned before we restrict ourselves in this section to the piecewise
constant case and use the notations from Section \ref{subsec:Piecewise-constant-notation}. 

\subsection{Proof of Proposition \ref{prop:LDPUB2}}
For Proposition \ref{prop:LDPUB2} we actually only need to prove
that for every $x\ge r_{\sigma}$ and $\psi\in\mathcal{P}_{p}$, every
$\theta>G_{\mu_{\sigma}}(x)/2$, 
\begin{equation}
\lim_{\delta\to0}\limsup_{N\to\infty}P_{\delta}^{N}(x,\psi)\leq-\mathcal{F}(\theta,x,\psi,\sigma)\,.\label{ldub}
\end{equation}
Indeed, recalling  the definitions in \eqref{eq:Def_K_bis} and \eqref{eq:Def_Phi_theta_bis}, we will prove in  Appendix \ref{secsimp} that
\begin{lem}
\label{lem:theta} For any piecewise continuous variance profile $\sigma$,
any $\psi\in\mathcal{P}([0,1])$, any $x>r_{\sigma}$, and $\theta$
such that $G_{\mu_{\sigma}}^{-1}(2\theta)>x$, we have 
\[
K(\theta,\phi(\theta,x,\psi))=J_{\mu_{\sigma}}(x,\theta)\,.
\]
\end{lem}
As a consequence,  we see that $\mathcal{F}(\theta,x,\psi,\sigma)=0$
for $\theta\leq G_{\mu_{\sigma}}(x)/2$ so we only need to optimize
over $\theta>G_{\mu_{\sigma}}(x)/2$ rather than all positive $\theta$ in \eqref{ldub}.
We thus deduce Proposition \ref{prop:LDPUB2} readily from \eqref{ldub} by optimizing over $\theta>G_{\mu_{\sigma}}(x)/2$. To prove \eqref{ldub},
we may assume without loss of generality that 
\begin{equation}
\lim_{\delta\to0}\limsup_{N}P_{\delta}^{N}(x,\psi)>-\infty\,.\label{infty}
\end{equation}

The main  step to prove  \eqref{ldub}, and the central result of Section \ref{sec:LDupperbound}, is the following
Proposition.
\begin{prop}
\label{localem-1}There is a function $o(\delta)$ with $\lim_{\delta\rightarrow0}o(\delta)=0$
such that for every  $x\ge r_{\sigma}$, $\delta>0$, $\theta>G_{\mu_{\sigma}}(x)/2$
and $\psi\in{\cal P}_{p}$, 
\[
P_{\delta}^{N}(x,\psi)\le\frac{1}{N}\log\mathbb{E}_{H^{N}}\mathbb{E}_{u}[\mathds{1}_{\{\|\rho(u)-\phi(\theta,x,\psi)\|<o(\delta)\}}e^{N\theta\langle u,H^{N}u\rangle}]-J_{\mu_{\sigma}}(x,\theta)+\eta_\delta(N)
\]
where $\eta_\delta(N)$ goes to zero when $N$ goes to infinity and
then $\delta$ goes to zero. 
\end{prop}

We first check that the large deviation upper bound in Proposition
\ref{prop:LDPUB2} indeed follows from Proposition \ref{localem-1}.
\begin{proof}
[Proof of Proposition \ref{prop:LDPUB2}]With Proposition \ref{localem-1}
and Lemma \ref{ldpv-1-1} we deduce that for $\theta>G_{\mu_{\sigma}}(x)/2$
\begin{eqnarray*}
 &  & \limsup_{N\rightarrow\infty}P_{\delta}^{N}(x,\psi)\\
 &  & \le\limsup_{N\rightarrow\infty}\frac{1}{N}\log\mathbb{E}_{H}\mathbb{E}_{u}[\mathds{1}_{\{\|\rho(u)-\phi(\theta,x,\psi)\|<o(\delta)\}}e^{N\theta\langle u,Hu\rangle}]-J_{\mu_{\sigma}}(x,\theta)+\eta(\delta)\\
 &  & =K(\theta,\phi(\theta,x,\psi))-J_{\mu_{\sigma}}(x,\theta)+\eta(\delta)
\end{eqnarray*}
with $\eta(\delta)$ going to zero with $\delta$. This proves \eqref{ldub} and hence Proposition \ref{prop:LDPUB2}.
\end{proof}

\subsection{Localization of the spherical integral, Proof of Proposition \ref{localem-1}}

For any symmetric matrix $M$ and orthogonal projection $\Pi$ in
$\mathbb{R}^{N}$ we denote $\mu_{M}(\Pi)$ the non negative measure
on $\mathbb{R}$ given by 
\[
\mu_{M}(\Pi):=\frac{1}{N}\sum_{i=1}^{N}\langle v_{i}(M),\Pi v_{i}(M)\rangle\delta_{\lambda_{i}(M)}
\]
where we recall that $(v_{i}(M),\lambda_{i}(M))_{i\leq N}$ are the
eigenvectors and eigenvalues of $M$. In particular, for every bounded
continuous function $f$ we have

\begin{equation}
\mu_{M}(\Pi)(f):=\int_{\mathbb{R}}fd\mu_{M}(\Pi)=\frac{1}{N}Tr(\Pi f(M)\Pi)\label{eq:Trace_mu_M}
\end{equation}
Remark that $\mu_{M}(\Pi)$ generalizes the empirical measure of the
eigenvalues since $\mu_{M}(I_{N})=\mu_{M}$. In the case where $\Pi=uu^{*}$ is a rank
one projection, $\mu_{M}(uu^{*})=\frac{1}{N}\sum_{i=1}^{N}\langle u,v_{i}(M)\rangle^{2}\delta_{\lambda_{i}(M)}$ is the usual spectral measure associated
to the vector $u$.
The following result improves Lemma \ref{lem:Measure-Concentration}
so that we have concentration of measure not only for $\mu_{M}$ but
for every $\mu_{M}(\Pi_{k})$ as well. We recall that $d_{W}$ denote
the Wassertein distance and the projections $\Pi_{k}$ are defined in \eqref{defpik}.
\begin{thm}
\label{prop:G-measure-concentration} For  every $\delta>0$, 
\[
\limsup_{N\rightarrow\infty}\frac{1}{N}\log\mathbb{P}(\exists k\leq p,\,d_{W}(\mu_{\sigma,k},\mu_{H^{N}}(\Pi_{k}))\geq\delta)=-\infty.
\]
\end{thm}

The proof of this Theorem is given in Section \ref{sec:LDspectralmeasure}.
{ For $(\epsilon_N)_{N \geq 0}$ a  sequence  of positive real numbers going to $0$}, let us denote by 
\begin{equation}\label{defomegan}
\Omega_{N}:=\{H\in\mathcal{S}_{N}:\forall k\in[1,p],d_{W}(\mu_{H}(\Pi_{k}),\mu_{\sigma,k})\leq\epsilon_{N},||H||_{op}\leq{\cal C}\}\,.
\end{equation}
We see that $\Omega_{N}$ has a large probability:
\begin{lem}
\label{prop:exptightness} For any $K>0$, there exists ${\cal C}>0$
and a sequence $\epsilon_{N}$ going to $0$ as $N$ goes to $\infty$
such that:

\[
\limsup_{N\to\infty}\frac{1}{N}\log\mathbb{P}[H^{N}\notin\Omega_{N}]\leq-K.
\]
\end{lem}

\begin{proof}
[Proof of Lemma \ref{prop:exptightness}]The fact that we can choose
an adequate sequence $\epsilon_{N}$ comes from Theorem \ref{prop:G-measure-concentration}.
The choice of ${\cal C}>0$ comes from the fact that $\|H^{N}\|_{op}$
is exponentially tight and more precisely, see for instance \cite[Lemma 3.2]{HussonVar},
that there exists a constant $C>0$ such that for $A:=\max_{N}\max_{i,j}\Sigma_{i,j}^{N}$,
for every ${\cal C}>0$ 
\begin{equation}
\frac{1}{N}\log\mathbb{P}[||H^{N}||_{op}\geq{\cal C}]\leq C-\frac{{\cal C}^{2}}{A}\,.\label{expt}
\end{equation}
\end{proof}
We then notice that
\begin{align*}
\mathbb{P}[\{H^{N} & \in\Omega_{N}\}\cap\{|\lambda_{1}-x|\leq\delta\}\cap\{||\rho(v_{1})-\psi||<\delta\}]\\
 & \geq\mathbb{P}[\{|\lambda_{1}-x|\leq\delta\}\cap\{||\rho(v_{1})-\psi||<\delta\}]-\mathbb{P}[\Omega_{N}^{c}]\\
 & =\exp(NP_{\delta}^{N}(x,\psi))(1-\mathbb{P}[\Omega_{N}^{c}]\exp(-NP_{\delta}^{N}(x,\psi)))
\end{align*}
and then
\[
P_{\delta}^{N}(x,\psi)\leq\frac{1+o_{N}(1)}{N}\log\mathbb{P}[\{H^{N}\in\Omega_{N}\}\cap\{|\lambda_{1}-x|\leq\delta\}\cap\{\|\rho(v_{1})-\psi\|<\delta\}]\,,
\]
where $1+o_{N}(1)=(1-\mathbb{P}(\Omega_{N}^{c})\exp(-NP_{\delta}^{N}(x,\psi)))^{-1}$
goes to one as $N$ goes to infinity by \eqref{infty} provided  ${\cal C}$ and $\epsilon_N$
are chosen so that  \eqref{expt} holds with $K$ is large enough. 
We denote in the following 
\[
B_{\delta}(x,\psi)=\{H^{N}\in\Omega_{N}\}\cap\{|\lambda_{1}-x|\leq\delta\}\cap\{\|\rho(v_{1})-\psi\|<\delta\}\,.
\]

We next follow \cite{HuGu1,HussonVar} and tilt the measure by spherical
integrals. As a consequence of Lemma \ref{lem:J_spherical_integral},
for all $\theta\in[\frac{1}{2} G_{\mu_{\sigma}}(x-\delta),\infty)$ , for $N$
large enough,

\begin{eqnarray}
P_{\delta}^{N}(x,\psi) & \leq & \frac{1}{N}\log\mathbb{E}[\mathds{1}_{H^{N}\in B_{\delta}(x,\psi)}\frac{I_{N}(H^{N},\theta)}{I_{N}(H^{N},\theta)}]\nonumber \\
 & \leq & \frac{1}{N}\log\mathbb{E}[\mathds{1}_{H^{N}\in B_{\delta}(x,\psi)}\mathbb{E}_{u}[e^{N\theta\langle u,H^{N}u\rangle}]]-J_{\mu_{\sigma}}(x,\theta)+\eta_\delta(N)\label{tobound}
\end{eqnarray}
where $\eta_\delta(N)$ goes to zero as $N$ goes to infinity and
then $\delta$ to zero by uniform convergence of the spherical integral. 

Let us call $u'$ the projection of $u$ on the orthocomplement of
the eigenvector $v_{1}$ for the largest eigenvalue $\lambda_{1}$
of $H$ and recall the definition \eqref{eq:Def_Phi_theta_bis} of $\varphi$
for $G_{\mu_{\sigma}}(\lambda)\leq2\theta$ : $\varphi(\lambda,\theta)_{k}=\frac{1}{2\theta}G_{\mu_{\sigma,k}}(\lambda)$.
We also denote $\varsigma(u,v)_{k}:=\langle u,\Pi_{k}v\rangle$
for $1\leq k\leq p$ and for $\epsilon>0$
\begin{equation}
S_{\epsilon}(v_{1},\lambda_{1}):=\{u\in\mathbb{S}^{N-1}\colon\|\rho(u')-\varphi(\lambda_{1},\theta)\|_{\infty}<\epsilon\text{ and }\|\varsigma(u',v_{1})\|_{\infty}<\epsilon\}.\label{eq:S_epsilon_def}
\end{equation}

\begin{prop}
\label{prop:Gaussian-Concentration-1}Let $\delta,\epsilon>0$. For
any $H\in {B_{\delta}(x, \psi)}$ such that $\lambda_{1}(H)-3\delta>r_{\sigma}$ and
$G_{\mu_{\sigma}}(\lambda_{1}(H)-3\delta)<2\theta$ : 
\[
\frac{\mathbb{E}_{u}[\mathds{1}_{S_{\epsilon}(v_{1}(H),\lambda_{1}(H))}e^{N\theta\langle u,Hu\rangle}]}{\mathbb{E}_{u}[e^{N\theta\langle u,Hu\rangle}]}\geq e^{-o(N)}
\]
where $o(N)/N$ goes to zero as $N$ goes to infinity.
\end{prop}

We will prove this proposition in Section \ref{subsec:Tilted-probability-}
and now finish the proof of Proposition \ref{localem-1}.
\begin{proof}
[Proof of Proposition \ref{localem-1}] Because of Proposition \ref{prop:Gaussian-Concentration-1}
we can upper bound \eqref{tobound} by 
\begin{equation}
P_{\delta}^{N}(x,\psi)\leq\frac{1}{N}\log\mathbb{E}[\mathds{1}_{H^{N}\in B_{\delta}(x,\psi)}\mathbb{E}_{u}[\mathds{1}_{S_{\epsilon}(v_{1},\lambda_{1})}e^{N\theta\langle u,H^{N}u\rangle}]]-J_{\mu_{\sigma}}(x,\theta)+\eta_{\delta}(N)\label{tobound2}
\end{equation}
where $\eta_{\delta}(N)$ goes to zero as $N$ goes to
infinity and then $\delta$ goes to zero. Write $u=\langle v_{1},u\rangle v_{1}+u'$
so that we have 
\[
\rho(u)=\langle v_{1},u\rangle^{2}\rho(v_{1})+2\langle v_{1},u\rangle\varsigma(u',v_{1})+\rho(u').
\]
On the event $S_{\epsilon}(v_{1},\lambda_{1})\cap\{\|\rho(v_{1})-\psi\|\leq\delta\}$,
we have 
\[
\langle v_{1},u\rangle^{2}=1-\|u'\|^{2}=1-\sum_{k=1}^{p}\rho(u')_{k}=1-\sum_{k=1}^{p}\varphi(\lambda_{1},\theta)_{k}+O(\epsilon)=1-\frac{G_{\mu_{\sigma}}(\lambda_{1})}{2\theta}+O(\epsilon).
\]
We deduce that on $S_{\epsilon}(v_{1},\lambda_{1})\cap\{\|\rho(v_{1})-\psi\|\leq\delta\}$,
\[
\rho(u)=\left(1-\frac{G_{\mu_{\sigma}}(\lambda_{1})}{2\theta}\right)\psi+\varphi(\lambda_{1},\theta)+O(\epsilon+\delta).
\]
Finally we recall the definition of $\phi(\theta,x,\psi)$ from \eqref{eq:Def_Phi_theta_bis}
and because $\lambda\rightarrow G_{\mu_{\sigma}}(\lambda)$ is smooth
on $(r_{\sigma},\infty)$ we have 
\begin{align*}
S_{\epsilon}(v_{1},\lambda_{1})\cap B_{\delta}(x,\psi) & \subset S_{\epsilon}(v_{1},\lambda_{1})\cap\{\|\rho(v_{1})-\psi\|\leq\delta\}\cap\{|\lambda_{1}-x|<\delta\}\\
 & \subset\{\|\rho(u)-\phi(\theta,x,\psi)\|\leq o(\delta,\epsilon)\}.
\end{align*}
where $o(\delta,\epsilon)$ is a quantity going to $0$ as $\delta$
and $\epsilon$ go to $0$ and comes from the modulus of continuity
of the functions $G_{\mu_{\sigma,k}}$. Therefore, choosing $\epsilon=\epsilon_{N}^{(2)}$
going to zero sufficiently slowly such that Proposition \ref{prop:Gaussian-Concentration-1}
remains true and we deduce from \eqref{tobound2} that
\[
P_{\delta}^{N}(x,\psi)\leq\frac{1}{N}\log\mathbb{E}_{H}\mathbb{E}_{u}[\mathds{1}_{\{\|\rho(u)-\phi(\theta,x,\psi)\|<o(\delta)\}}e^{N\theta\langle u,Hu\rangle}]-J_{\mu_{\sigma}}(x,\theta)+\eta_\delta(N).
\]
\end{proof}

\subsection{Tilted probability on the sphere, Proof of Proposition \ref{prop:Gaussian-Concentration-1}}

\label{subsec:Tilted-probability-}

We fix the sequence $(\epsilon_{N})_{N\in\mathbb{N}}$ of positive
real numbers that goes to $0$ as $N$ goes to $+\infty$ and ${\cal C}>0$
so that Lemma \ref{prop:exptightness} is satisfied. The proof of
Proposition \ref{prop:Gaussian-Concentration-1} will be based on
localization of the following probability measure $\mathbb{Q}_{H,\theta}$
on the sphere 
\[
d\mathbb{Q}_{H,\theta}(u)=\frac{e^{N\theta\langle u,Hu\rangle}du}{I_{N}(\theta,H)}\,,
\]
where $H$ is a deterministic self-adjoint matrix and $\theta$ a real number.
Recalling the definition  \eqref{eq:S_epsilon_def} of $S_{\epsilon}$,  Proposition \ref{prop:Gaussian-Concentration-1} is then equivalent
to showing that for
any $H\in {B_{\delta}(x, \psi)}$ such that $\lambda_{1}(H)-3\delta>r_{\sigma}$ and
$G_{\mu_{\sigma}}(\lambda_{1}(H)-3\delta)<2\theta$, 
\begin{equation}
\mathbb{Q}_{H,\theta}(S_{\epsilon}(v_{1},\lambda_{1}))\geq e^{-o(N)}\,.\label{eq:Gaussian_concentration}
\end{equation}
We will first prove a weaker version of this result : 
\begin{prop}
\label{prop:Gaussian-Concentrationbis} Inequality \eqref{eq:Gaussian_concentration}
is true under the additional assumption that $x-\lambda_{2}>3\delta$
where $\lambda_{2}$ is the second largest eigenvalue of $H$. 
\end{prop}

%
\begin{proof}
Under this assumption, if $\lambda_{1}\ge\lambda_{2},\ldots\ge\lambda_{N}$
are the eigenvalues of $H$, $\lambda_{1}-\lambda_{i}$ is bounded
below by $\delta$ on $\{|\lambda_{1}-x|\le\delta\}$ and therefore
we can define 
\[
u\text{\textquoteright}=\sum_{i=2}^{N}\frac{Y_{i}}{\sqrt{2N\theta(\lambda_{1}-\lambda_{i})}}v_{i}
\]
where $Y_{i}$ are iid Gaussian ${\cal N}(0,1)$. We denote by $\tilde{\mathbb{Q}}$
the law of $(\frac{Y_{i}}{\sqrt{2N\theta(\lambda_{1}-\lambda_{i})}})_{i\ge2}$.
Then it is not hard to see that

\[
\mathbb{Q}_{H,\theta}(u'\in A)=\frac{\int_{A\cap\mathbb{B}_{N-1}}(1-||u||^{2})^{-1/2}d\tilde{\mathbb{Q}}(u)}{\int_{\mathbb{B}_{N-1}}(1-||u||^{2})^{-1/2}d\tilde{\mathbb{Q}}(u)}
\]
where $\mathbb{B}_{N-1}$ is the unit ball in $\mathbb{R}^{N-1}$.
With the spectral gap assumption, using the same arguments as in the
proof of \cite[Proposition 4.3]{CoDuGu}, especially the proof of
inequality (4.17) we find that

\begin{equation}
\limsup_{N\to\infty}\frac{1}{N}\log\int_{\mathbb{B}_{N-1}}(1-||u||^{2})^{-1/2}d\tilde{\mathbb{Q}}(u)\leq0\,.\label{eq:tg}
\end{equation}
Since $\int_{A\cap\mathbb{B}_{N-1}}(1-||u||^{2})^{-1/2}d\tilde{\mathbb{Q}}(u)\geq\tilde{\mathbb{Q}}(A)$
it only remains to show that $\mathbb{\tilde{Q}}(S_{\epsilon}(v_{1},\lambda_{1}))\geq e^{-o(N)}$.
Denoting by $\tilde{\mathbb{E}}$ the expectation under $\tilde{\mathbb{Q}}$,
we have that

\[
\tilde{\mathbb{E}}[\|\Pi_{k}u\text{\textquoteright}\|^{2}]=\sum_{i=2}^{N}\frac{1}{2N\theta(\lambda_{1}-\lambda_{i})}\langle v_{i},\Pi_{k}v_{i}\rangle
\]
and we recognize the Stieljes transform of $\mu_{H}(\Pi_{k})$ on
which the mass on $\delta_{\lambda_{1}}$ has been removed. Given
the definition of $\Omega_{N}$ and $\varphi(\theta,\lambda_{1})$
and because $z\rightarrow(\lambda_{1}-z)^{-1}$ is $\delta^{-2}$
Lipschitz on $(-\infty,\lambda_{1}-\delta)$ we have

\[
\lim_{N\to+\infty}\left|\tilde{\mathbb{E}}[\|\Pi_{k}u\text{\textquoteright}\|^{2}]-\varphi(\theta,\lambda_{1})_{k}\right|=0\,.
\]
Since $u'$ is a Gaussian vector whose covariance matrix is bounded
in spectral radius by $1/(2N\theta\delta)$, the same is also true
for $\Pi_{k}u'$, therefore one can write with $Z:=||\Pi_{k}u'||^{2}$
and $\sigma_{1},\dots,\sigma_{N}$ the eigenvalues of the covariance
matrix of $\Pi_{k}u'$.

\[
\text{Var}(Z)=2\sum_{i=1}^{N}{\sigma_{i}^{2}}\leq\frac{1}{2N\theta^{2}\delta^{2}}\,.
\]
We conclude then using Chebyshev's inequality that 
\begin{eqnarray}
\tilde{\mathbb{Q}}(\exists k\leq p:\Big|\|\Pi_{k}u\text{\textquoteright}\|^{2}-\tilde{\mathbb{E}}\|\Pi_{k}u\text{\textquoteright}\|^{2}\Big|\ge\frac{\epsilon}{2})&\le&\frac{4p}{\epsilon^{2}}\max_{1\leq k\leq p}\mbox{Var}(\|\Pi_{k}u'\|^{2})\nonumber\\
&\le&\frac{2p}{N\theta^{2}\delta^{2}\epsilon^{2}}\,.\label{cont1}
\end{eqnarray}
Furthermore, for $N\geq(2\delta\theta)^{-1}$, since $v_{1}$ is a
unit vector, then $\langle\Pi_{k}u',v_{1}\rangle=\langle u',\Pi_{k}v_{1}\rangle$
is a centered Gaussian vector whose variance is bounded by the spectral
radius of the covariance matrix of $u'$. As a consequence, we deduce
that 
\[
\tilde{\mathbb{E}}[\langle\Pi_{k}u',v_{1}\rangle^{2}]\leq\frac{1}{N\theta\delta}
\]
Therefore, Chebyshev's inequality directly gives the result: 
\begin{equation}
\tilde{\mathbb{Q}}(\exists k\le p:|\langle\Pi_{k}u',v\rangle|\ge\epsilon)\le\frac{p}{N\theta\delta\epsilon^{2}}.\label{cont2}
\end{equation}
Putting together \eqref{cont1} and \eqref{cont2} and for $N$ large
enough so that 

\noindent
$\left|\tilde{\mathbb{E}}[\|\Pi_{k}u\text{\textquoteright}\|^{2}]-\varphi(\theta,\lambda_{1})_{k}\right|$ is smaller than $\frac{\epsilon}{2}$,
we prove
\begin{equation}
\tilde{\mathbb{Q}}(S_{\epsilon}(v_{1},\lambda_{1}))\geq1-\frac{1}{N}\left(\frac{p}{\theta\delta\epsilon^{2}}+\frac{2p}{\theta^{2}\delta^{2}\epsilon^{2}}\right).\label{cont00}
\end{equation}
With \eqref{eq:tg}, the proof of Proposition \ref{prop:Gaussian-Concentrationbis}
is complete since it implies that 
\[
\mathbb{Q}_{H,\theta}(S_{\epsilon}(v_{1},\lambda_{1}))\ge e^{-o(N)}\tilde{\mathbb{Q}}(S_{\epsilon}(v_{1},\lambda_{1}))=e^{-o(N)}\,.
\]
\end{proof}
To deal with the case where $\lambda_{1}$ is not an outlier, let
us first define a generalization of $\mathbb{Q}_{H,\theta}$. For
$V$ a Euclidean vector space and $H$ a self-adjoint endomorphism
of $V$, we define : 
\[
\mathbb{Q}_{H,\theta}^{V}(u):=\frac{e^{\text{dim}(V)\theta\langle u,Hu\rangle}}{\mathbb{E}[e^{\text{dim}(V)\theta\langle u,Hu\rangle}]}du
\]

Let $k_{N}$ be a sequence of integer numbers  such that $\frac{k_{N}}{N}\rightarrow0$
as $N\rightarrow\infty$ { and $(\epsilon_N)_{N \geq 0}$ be  a sequence of positive real numbers  that goes to $0$. We define $\Omega'_{N}$ by}
\[
\Omega'_{N}:=\{H\in\mathcal{S}_{N},\forall k\in[1,p],d_{W}(\mu(\Pi_{k}),\mu_{\sigma,k})\leq\epsilon_{N},\lambda_{k_{N}}(H)<x-2\delta,\|H\|_{op}\leq{\cal C}\}.
\]
{ We can find such a sequence $(k_N)_{N \geq 0}$ that will be such that $$B_{\delta}(x, \psi) \subset  \Omega'_N \cap \{ |\lambda_1 -x| , ||\rho(u) - \psi || \leq \delta  \}\,.$$
Indeed let us consider some Lipschitz function $\phi$ with support in $[r_{\sigma}, \mathcal{C} +1 ]$  such that $0 \leq \phi \leq 1$ and $\phi$ is equal to $1$ on $[ r_{\sigma} + \delta , \mathcal{C}]$. Then, since $d_{W}( \mu_H, \mu_{\sigma}) \leq p \epsilon_N$, we find that

\[ \Big| \int \phi(x) d \mu_H(x) - \int \phi(x) d \mu_{\sigma}(x) \Big| \leq ||\phi||_{Lip}  p \epsilon_N\]
By construction the second term of the difference is $0$ and the first term is bounded below by $\# \{ i \in[1,N], \lambda_i \geq r_{\sigma} + \delta \}/N$. Since $x -3 \delta \geq r_{\sigma}$ by assumption, this implies that 

\[ \# \{ i \in[1,N], \lambda_i \geq x - 2 \delta \} \le \# \{ i \in[1,N], \lambda_i \geq r_{\sigma} + \delta \}
\leq N ||\phi||_{Lip}  p \epsilon_N .\]

In particular, this implies that for any  $k_N $ larger than $  N ||\phi||_{Lip}  p \epsilon_N / 2 $, $\lambda_{k_N} < x - 2 \delta$. Therefore we only need to prove Proposition \ref{prop:Gaussian-Concentration-1} assuming $H \in \Omega'_N$.

 }
We denote $\Pi_{\text{out}}$ the orthogonal projector on $\text{Span}(v_{2}(H),\cdots,v_{k_{N}}(H))$,
$\Pi_{\text{cut}}:=I-\Pi_{\text{out}}$ that is the orthogonal projector
on $V=\text{Span}(v_{1}(H),v_{k_{N}+1}(H),\cdots,v_{N}(H))$, 
\[
u_{\text{cut}}:=\frac{\Pi_{\text{cut}}u}{\|\Pi_{cut}u\|},\quad H^{\text{cut}}:=\frac{N}{N-k_{N}+1}H|_{V}
\]
with $H|_{V}$ the matrix $H$ restricted on $V$ and $u$ a unit  vector.
\begin{lem}
\label{SpectralGap} Let $\delta>0$ and $x>r_{\sigma}+3\delta$ and
$2\theta>G_{\mu_{\sigma}}(x-3\delta)$. For any Borelian $A\subset\mathbb{R}^{N}$,
$\epsilon>0$ and $H\in\Omega'_{N}$ such that $|\lambda_{1}(H)-x|\leq\delta$
we have that
\begin{align*}
 & \frac{1}{N}\log\mathbb{Q}_{H,\theta}(\|\Pi_{\text{out}}(u)\|^{2}\leq\epsilon,u_{\text{cut}}\in A)\\
 & \quad\geq\frac{1}{N}\log\mathbb{P}_{u}[\|\Pi_{\text{out}}(u)\|^{2}\leq\epsilon]+\frac{1}{N}\log\mathbb{Q}_{H^{cut},\theta}^{V}(u\in A)+O({\cal C}\theta\epsilon)+\eta(N)
\end{align*}
where $\eta(N)$ goes to zero when $N$ goes to infinity. { Moreover, $\mathbb{Q}_{H^{cut},\theta}^{V}$ is defined as $\mathbb{Q}_{H,\theta}$ except that $u$ follows the uniform law on the sphere in $V$ (and $H$ is restricted to $V$).}
\end{lem}

\begin{proof}[Proof of Lemma \ref{SpectralGap}]

We have, since $H$ maps $V$ into $V$, that
\[
\langle u,Hu\rangle=\langle u,\Pi_{\text{out}}H\Pi_{\text{out}}u\rangle+\langle u,\Pi_{\text{cut}}H\Pi_{\text{cut}}u\rangle
\]
and $\Pi_{\text{cut}}u=(1-\|\Pi_{\text{out}}u\|^{2})^{1/2}u_{\text{cut}}$.
We recall that on $\Omega_{N}$ we have $\|H\|_{op}\le{\cal C}$ and
then 
\begin{eqnarray*}
\mathbb{Q}_{H,\theta}(\|\Pi_{\text{out}}u\|^{2}\leq\epsilon,u_{\text{cut}}\in A) & \geq & \frac{e^{-2N{\cal C}\theta\epsilon}}{I(H,\theta)}\int_{\mathbb{S}^{N-1}}e^{N\theta\langle u_{\text{cut}},Hu_{\text{cut}}\rangle}\mathds{1}_{\|\Pi_{\text{out}}u\|^{2}\leq\epsilon,\,u_{\text{cut}}\in A}du.
\end{eqnarray*}
One may use now that  for $u$ following the uniform law on $\mathbb{S}^{N-1}$,
$u_{\text{cut}}$ is a uniform vector on the unit sphere of $V$ and
is independent from $\|\Pi_{\text{out}}u\|^{2}$. Therefore 

\begin{eqnarray*}
 &  & \int_{\mathbb{S}^{N-1}}e^{N\theta\langle u_{\text{cut}},Hu_{\text{cut}}\rangle}\mathds{1}_{\|\Pi_{\text{out}}u\|^{2}\leq\epsilon,\,u_{\text{cut}}\in A}du\\
 &  & \quad=\mathbb{P}_{u}[\|\Pi_{\text{out}}u\|^{2}\le\epsilon]\int_{\mathbb{S}^{N-1}}e^{N\theta\langle u_{\text{cut}},Hu_{\text{cut}}\rangle}\mathds{1}_{u_{\text{cut}}\in A}du\\
 &  & \quad=\mathbb{P}_{u}[\|\Pi_{\text{out}}u\|^{2}\le\epsilon]\mathbb{Q}_{H^{\text{cut}},\theta}^{V}(u\in A)I(H^{\text{cut}},\theta).
\end{eqnarray*}
It now only remains to show that
\[
\lim_{N\to\infty}\sup_{H\in\Omega_{N}}\frac{1}{N}\log\frac{I(H^{\text{cut}},\theta)}{I(H,\theta)}=0
\]
and this follows from Lemma \ref{lem:J_spherical_integral}, the definition
of $\Omega'_{N}$ and the continuity of $J_{\mu_{\sigma}}(.,\theta)$.
\end{proof}
\begin{proof}
[Proof of Proposition \ref{prop:Gaussian-Concentration-1}]If we denote
$${\nu}(\Pi):=\frac{1}{N-k_{N}+1}\sum_{i=k_{N}+1}^{N}\langle v_{i},\Pi v_{i}\rangle\delta_{\lambda_{i}}+\langle v_{1},\Pi v_{1}\rangle\delta_{\lambda_{1}},$$
we also have that for $H\in\Omega'_{N}$, 
\[
\lim_{N\to\infty}{\nu}(\Pi_{k})=\mu_{\sigma,k}
\]
and furthermore we have that for some { sequence of positive real numbers  $(\epsilon''_{N})_{N \geq 0}$ that goes to $0$}
when $N\to\infty$, $d(\nu(\Pi_{k}),\mu_{\sigma,k})\leq\epsilon''_{N}$.
Last, we have that for $N$ large enough $\|H^{\text{cut}}\|<2{\cal C}$.
For $N\in\mathbb{N}$, $V\subset\mathbb{R}^{N}$ a linear subspace,
let us denote $\mathcal{S}_{N}(V)$ the set of self-adjoint endomorphisms
of $V$ for a fixed $V\subset\mathbb{R}^{N}$ such that $\text{dim}(V)=N-k_{N}+1,\Omega''_{N}$
is given as follows :
\[
\Omega''_{N}:={ \bigcup_{ V \subset {\mathbb R}^N \atop V  \text{ is a subspace}} }\{H\in\mathcal{S}_{N}(V):\|H\|\leq2{\cal C},\forall k\in[1,p],d_{W}(\mu_{H}(\Pi_{k}),\mu_{\sigma,k})\leq\epsilon''_{N}\}.
\]

So we have that $H\in\Omega'_{N}$ implies that $H^{\text{cut}}\in\Omega''_{N}$ for $N$ large enough.
Furthermore,for $N$ large enough, we have that for all {$H\in\Omega'_{N}$,}
$|\lambda_{1}(H)-x|\leq\delta$ implies $|\lambda_{1}(H^{\text{cut}})-x|\leq\delta$
and by construction, we have that for $N$ large enough, for all {$H\in\Omega'_{N}$,}
$H^{\text{cut}}$ has the spectral gap, in the sense that {$\lambda_{2}(H^{\text{cut}})\leq x-2\delta$.}
Hence we can apply Proposition \ref{prop:Gaussian-Concentrationbis}
to conclude that if 
\[
A:=\{u\in\mathbb{S}^{N-1}\cap V\colon\|\rho(u')-\varphi(\lambda_{1},\theta)\|_{\infty}<\epsilon\text{ and }\|\varsigma(u',v_{1})\|_{\infty}<\epsilon\}
\]
 then

\[
\frac{1}{N}\log\mathbb{Q}_{H^{\text{cut}},\theta}^{V}(u\in A)\ge o(1)\,.
\]
Moreover, it is easy to see that for any $\epsilon>0$ 
\[
\lim_{N\rightarrow\infty}\frac{1}{N}\log\mathbb{P}_{u}[\|\Pi_{\text{out}}u\|^{2}\le\epsilon]=0
\]
so with Lemma \ref{SpectralGap} we can find a sequence $\epsilon_{N}^{(1)}$
going to $0$ such that for $\epsilon=\epsilon_{N}^{(1)}$

\[
\mathbb{Q}_{H,\theta}(\|\Pi_{\text{out}}u\|^{2}\leq\epsilon_{N}^{(1)},u_{\text{cut}}\in A)\geq e^{-o(N)}\,.
\]
Furthermore we can see that for $\epsilon>0$ and $N$ large enough:

\[
\{\|\Pi_{\text{out}}u\|^{2}\leq\epsilon_{N}^{(1)},u_{\text{cut}}\in A\}\subset S_{2\epsilon}(v_{1},\lambda_{1})
\]
so $\mathbb{Q}_{H,\theta}(S_{2\epsilon}(v_{1},\lambda_{1}))\geq e^{-o(N)}$
which finishes the proof of Proposition \ref{prop:Gaussian-Concentration-1}. 
\end{proof}

\subsection{Concentration for the spectral measures, proof of Theorem \ref{prop:G-measure-concentration}}

\label{sec:LDspectralmeasure}

We first prove the convergence of the expectation and then will use concentration of measure to conclude. 
\begin{prop}\label{prop:expectation-convergence}
\label{prop:weak-muk}For $1\leq k\leq p$, $\mathbb{E}[\mu_{H^{N}}(\Pi_{k})]$
converges weakly toward $\mu_{\sigma,k}$. 
\end{prop}

\begin{proof}
[Proof of Proposition \ref{prop:weak-muk}] To prove this convergence,
one only has to prove that for every $z\in\mathbb{C}\setminus\mathbb{R}$,
the convergence in probability of the Stieljes transform
\begin{equation}\label{convSt}
G_{\mu_{H^{N}}(\Pi_{k})}(z)\rightarrow G_{\mu_{\sigma,k}}(z)
\end{equation}
 since by boundedness of $x\mapsto(z-x)^{-1}$ we will then have the
convergence in expectation. {This will imply the vague convergence of  $\mathbb{E}[\mu_{H^{N}}(\Pi_{k})]$
 toward $\mu_{\sigma,k}$. It can be easily strenghtened into a weak convergence by using the tightness established in \eqref{expt}. }
 One then notices, using \eqref{eq:Trace_mu_M},

\[
G_{\mu_{H^{N}}(\Pi_{k})}(z)=\frac{1}{N}\sum_{i\in I_{k}^{N}}(z-H^{N})_{i,i}^{-1}\,.
\]
For the case of a variance profile bounded below by some positive
number (i.e $\sigma_{i,j}>0$ for all $i,j\in[1,p]$) one can apply
the local law \cite[Theorem 1.7]{AEK} and \cite[Theorem A.6]{HussonVar}
to prove the statement \eqref{convSt}. For variance profiles that can vanish, one
can then proceed by approximation, for instance approximating $H^{N}$
by $H^{N}+\epsilon W^{N}$ where $W^{N}$ is GOE and $\epsilon$ a
positive real number. The previous proof shows that $\mathbb{E}[\mu_{H^{N}+\epsilon W^{N}}(\Pi_{k})]$
converges weakly toward $\mu_{(\sigma^{2}+\epsilon^{2})^{1/2},k}$.
One can then use that if $\|.\|_{op}$ denotes the operator norm,
because $\|W^{N}\|_{op}$ is bounded by $3$ with probability greater
than $1-e^{-cN}$ for some positive $c$, 
\[
\|(z-H^{N}-\epsilon W^{N})^{{-1}}-(z-H^{N})^{{-1}}\|_{op}\le\frac{3}{\Im z^{2}}\epsilon
\]
to see that $\mu_{H^{N}+\epsilon W^{N}}(\Pi_{k})$ is close to $\mu_{H^{N}}(\Pi_{k})$
when $\epsilon$ goes to zero with overwhelming probability. Finally,
one can use Theorem A.6 in \cite{HussonVar} again to prove that $\mu_{\sqrt{\sigma^{2}+\epsilon^{2}},k}$
goes to $\mu_{\sigma,k}$ when $\epsilon$ goes to $0$. The conclusion
follows in the general case. 
\end{proof}

\begin{proof}
[Proof of Theorem \ref{prop:G-measure-concentration}]We now obtain an estimate on the  concentration of measure of the law of the spectral measures to conclude the proof of 
Theorem \ref{prop:G-measure-concentration}, namely we show that 
for every  orthogonal projection $\Pi$,
every $\delta>0$,
\begin{equation}\label{top}
\limsup_{N\rightarrow\infty}\frac{1}{N}\log \mathbb P(d_{W}(\mu_{H^{N}}(\Pi),\mathbb E[\mu_{H^{N}}(\Pi)])>\delta)=-\infty\,.\end{equation}
This readily concludes the proof of Theorem \ref{prop:G-measure-concentration}  by taking $\Pi=\Pi_{k}$, $1\le k\le p$, and combining this estimate with the weak  convergence of $\mathbb{E}[\mu_{H^{N}}(\Pi_{k})]$  towards $\mu_{\sigma,k}$ derived in Proposition \ref{prop:expectation-convergence}. 
So, let $\Pi$ a
orthogonal projection and $f$ be a Lipschitz function. By definition,
\[
\mu_{H}(\Pi)[f]=\frac{1}{N}\Tr(f(H)\Pi)
\]
is Lipschitz in the entries of $X=\sqrt{N} H$ for the Euclidean norm. Indeed,
since $\Pi$ is a projection, its operator norm is bounded above by
one so that 
\begin{eqnarray*}
|\mu_{H}(\Pi)[f]-\mu_{H'}(\Pi)[f]| & \le & \frac{1}{N}\Tr(\Pi|f(H)-f(H')|)\\
 & \le & \frac{1}{N}\|f\|_{L} \left(\frac{1}{N}\Tr\Pi\right)^{1/2}\left(\Tr(X-X')^{2}\right)^{1/2}\,.
\end{eqnarray*}
Let us first assume that the $X_{i,j}^{N}$ are all bounded by some
$K$ finite. Then we can use apply Talagrand's bound \cite[Theorem 6.6]{talagrand}
to insure that for $f$ convex and Lipshitz, for all $\delta>0$ 
\[
\mathbb{P}\left(|\mu_{H^{N}}(\Pi)[f]-\text{Median}(\mu_{H^{N}}(\Pi)[f])|\ge\delta\|f\|_{L} \left(\frac{1}{N}\Tr\Pi\right)^{1/2}\right)\le4e^{-(4K)^{-2}N^{2}\delta^{2}}\,.
\]
As a consequence, we find that
\[
|\text{Median}(\mu_{H^{N}}(\Pi)[f])-\mathbb{E}[\mu_{H^{N}}(\Pi)[f]]|\le16\sqrt{\pi}\frac{K}{N}\|\sigma\|_{\infty}\|f\|_{L}\left(\frac{1}{N}\Tr\Pi\right)^{1/2}=:\delta_{N}[f]\,.
\]
Therefore 
\[
\mathbb{P}\left(|\mu_{H^{N}}(\Pi)[f]-\mathbb{E}[\mu_{H^{N}}(\Pi)[f]]|\ge\delta\|f\|_{L}  \left(\frac{1}{N}\Tr\Pi\right)^{1/2}+\delta_{N}[f]\right)\le4e^{-(4K)^{-2}N^{2}\delta^{2}}
\]
By optimization and approximation, this can be generalized to the
$d_{W}$ distance in the spirit of \cite[Corollary 1.4]{GZ} by approximating
any Lipschitz function $f$ by a sum of at most $K/\delta$ convex
Lipschitz function to deduce that 
\[
\mathbb{P}\left(d_{W}(\mu_{H^{N}}(\Pi),\mathbb{E}[\mu_{H^{N}}(\Pi)])>(\delta+16\sqrt{\pi}\frac{K}{N})\|\sigma\|_{\infty}\left(\frac{1}{N}\Tr\Pi\right)^{1/2}\right)\le\frac{32K}{\delta}e^{-(4K)^{-2}N^{2}\delta^{2}}
\]
This completes the proof of \eqref{top} when $X_{ij}^{N}$ is bounded.
To deal with the case where $X_{ij}^{N}$ is not bounded
(but still has subgaussian tail), we denote $a_{ij}=X_{ij}^{N}/(\Sigma_{ij}^{N})^{1/2}$
write 
\[
H^{N}=H^{\le}+H^{>},\mbox{ with }H_{ij}^{\le}=N^{-1/2}\sqrt{\Sigma_{ij}^{N}}\frac{[\mathds{1}_{|a_{ij}|\le K}a_{ij}-\mathbb{E}[\mathds{1}_{|a_{ij}|\le K}a_{ij}]]}{\mathbb{E}[(\mathds{1}_{|a_{ij}|\le K}a_{ij}-\mathbb{E}[\mathds{1}_{|a_{ij}|\le K}a_{ij})^{2}]^{1/2}}
\]
Because $a_{ij}$ are sharp subgaussian, $P(|a_{ij}|\ge K)\le2e^{-K^{2}/2}$
and therefore $\mathbb{E}[\mathds{1}_{|a_{ij}|\le K}a_{ij}]], \mathbb{E}[\mathds{1}_{|a_{ij}|>K}a_{ij}]]$
and $\mathbb{E}[(\mathds{1}_{|a_{ij}|\le K}a_{ij}-\mathbb{E}[\mathds{1}_{|a_{ij}|\le K}a_{ij})^{2}]^{-1/2}-1$
are at most of order $e^{-K^{2}}$. We can apply the previous concentration
result to $H^{\le}$ with any $K\ll\sqrt{N}$ the probability that
$\mu_{H^{\le}}(\Pi)$ deviates from its mean is going to zero faster
than exponentially in $N$. We next take $K=N^{1/2-\epsilon}$ for
$\epsilon>0$. We see that 
\[
H_{ij}^{>}=N^{-1/2}\sqrt{\Sigma_{i,j}}\mathds{1}_{|a_{ij}|>K}a_{ij}+m_{ij}^{K}
\]
where $m_{ij}^{K}$ is a variable coming from the recentering and
rescaling of $a_{ij}\mathds{1}_{|a_{ij}|\le K}$ which is bounded
by $e^{-K^{2}/2}(1+|a_{ij}|)$. Therefore, $\Tr(m^{K})^{2}\le Ne^{-K^{2}/2}\Tr(H^{2})$
is at most of order $Ne^{-K^{2}/2}$ and therefore is very small.
For the first term, we claim that $H^{>}-m^{K}$ has small rank with
overwhelming probability. To see this, it is enough to show that with
high probability, $H^{>}-m^{K}$ has a small rank, say smaller than
$\kappa N$ with $\kappa$ small, and in fact it has less than $\kappa N$
columns with a non zero entry. To prove this we can consider the families
of Bernoulli random variable $\mathds{1}_{|a_{i,j}|>K}$ for $1\leq i\leq j\leq N$.
Using the sub-Gaussian property of the $a_{i,j}$, these are i.i.d.
Bernoulli random variables of parameter 
\[
P(|a_{i,j}|>K)\le p_{N}=O(e^{-cK^{2}})
\]
As a consequence of Tchebychev's inequality, for any $\beta>0$,

\[
\mathbb{P}(\sum_{i\leq j}\mathds{1}_{|a_{i,j}|>K}\ge t)\le e^{-\beta t}(e^{\beta}p_{N}+1)^{\frac{N(N-1)}{2}}\le e^{-\beta t+N^{2}e^{\beta}p_{N}}
\]
Hence, if $t>ep_{N}N^{2}$, taking $\beta=\log(\frac{t}{p_{N}N^{2}})>1$,
we deduce 
\[
\mathbb{P}(\sum_{i\leq j}\mathds{1}_{|a_{i,j}|>K}\ge t)\le e^{-(\log(\frac{t}{p_{N}N^{2}})-1)t}
\]
We conclude by taking $t=\delta N$ and $K=N^{1/2-\epsilon}$ that
for $N$ large enough and $\delta>e^{-cN^{1-4\epsilon}}$ {and $\epsilon < 1/10$}

\[
\mathbb{P}\Big[\sum_{i\leq j}\mathds{1}_{|a_{i,j}|>K}\geq { \delta N }\Big]\leq e^{-cN^{1+2\epsilon}\delta}\,.
\]
Therefore, taking $\delta=N^{-\epsilon}$, we conclude that up to
an error negligible on the exponential scale, the matrix $H^{>>}:=(\mathds{1}_{|a_{ij}|>K}a_{ij})_{i,j}$
has at most $N^{1-\epsilon}$ non-zero entries, and therefore rank
bounded by $N^{1-\epsilon}$, with probability greater than $1-e^{-cN^{1+\epsilon}}$.
We conclude that 
\begin{eqnarray*}
d_{W}(\mu_{H^{N}}(\Pi),\mu_{H^{\le}}(\Pi)) & \le & \left(\frac{1}{N}\Tr(m^{K})^{2}\right)^{1/2}+\frac{\mbox{rank}(H^{>>})}{N}\\
 & \le & N^{-\epsilon}+(e^{-N^{1-\epsilon}/2}\Tr(H^{2}))^{1/2}
\end{eqnarray*}
is bounded by $2N^{-\epsilon}$ with probability greater than $1-e^{-N^{1+\epsilon}}$. this is enough to conclude.
\end{proof}

\subsection{Adding the condition $\langle\psi,S\psi\rangle>\epsilon$}\label{seceps}

The last section completed the proof of  the large deviation upper bound stated in Proposition \ref{prop:LDPUB2}
In view of proving the corresponding lower bound, we improve this
upper bound by showing that we can restrict the infimum over $\psi$.
Namely, we remind that for $\psi\in\mathcal{P}_{p}$, we denote $\langle\psi,S\psi\rangle=\sum_{k,l=1}^{p}\sigma_{k,l}\psi_{k}\psi_{l}$.
We shall prove the following improvement upon Proposition \ref{prop:LDPUB2}.
We denote $A:=\max_{k,l}\sigma_{k,l}$ and $a:=\langle|I|,S|I|\rangle=\sum_{k,l=1}^{p}\sigma_{k,l}|I_{k}||I_{l}|$.
\begin{prop}
\label{prop:epsilon} For  every $x>r_{\sigma}$, there exists $\varepsilon_{x}>0$
such that
\[
\limsup_{\delta\downarrow0}\limsup_{N\rightarrow\infty}\frac{1}{N}\log\mathbb{P}(|\lambda_{1}-x|\le\delta)\le-\inf_{\psi:\langle\psi,S\psi\rangle\ge\varepsilon_{x}}\sup_{\theta>0}\mathcal{F}(\theta,x,\psi,\sigma)\,.
\]
Moreover, we can assume that $x\mapsto\epsilon_{x}$ is decreasing
and only depending on $a,A$ and $r_{\sigma}$. 
\end{prop}

A key step to prove this proposition is to show that we can exclude
the eigenvectors $v_{1}$ for the eigenvalue $\lambda_{1}$ such that
$\langle\rho(v_{1}),S\rho(v_{1})\rangle$ is small. This is the content of   the following
lemma. 
\begin{lem}
\label{lem:vectexcl} There exists constants $C,C'>0$ such that for
any $M>0$ and every $x>r_{\sigma}$ :
\[
\frac{1}{N}\log\mathbb{P}[|\lambda_{1}-x|\leq\delta,\langle\rho(v_{1}),S\rho(v_{1})\rangle\leq\epsilon]\leq\max\Big(C-\frac{M^{2}}{A},C'-\log r+\log M-\frac{r^{2}}{4\epsilon}\Big)
\]
with $r=x-\delta$. 
\end{lem}

This Lemma is a direct consequence of \eqref{expt} and the following
lemma:
\begin{lem}
\label{lemb} For $\epsilon>0$, denote by $\mathcal{S}_{\epsilon}:=\{u\in\mathbb{S}^{N-1},\langle\rho(u),S\rho(u)\rangle\leq\epsilon\}.$
Then, there exists a constant $C'>0$ such that for all $M>0$ 
\[
\frac{1}{N}\log\mathbb{P}[||H^{N}||\leq M,\lambda_{1}\geq r,v_{1}\in\mathcal{S}_{\epsilon}]\leq C'+\log M+\log r-\frac{r^{2}}{4\epsilon}.
\]
\end{lem}

\begin{proof}
For a given $\epsilon>0$, we can built a $r/4M$-net $\mathcal{N}_{\epsilon}(r/M)$
of $\mathcal{S}_{\epsilon}$ of cardinality less than $(K'M/r)^{N}$
for some $K'$ finite independent of $M,N,r$. Thus,  we find that 
\begin{eqnarray*}
\mathbb{P}[\|H^{N}\|\leq M,\lambda_{1}\geq r,v_{1}\in\mathcal{S}_{\epsilon}] & \leq & \mathbb{P}[\|H^{N}\|\leq M,\exists u\in\mathcal{S}_{\epsilon},\langle u,H^{N}u\rangle\ge r]\\
 & \leq & \mathbb{P}[\|H^{N}\|\leq M,\exists u\in\mathcal{N}_{\epsilon}(r/M),\langle u,H^{N}u\rangle\ge r/2]\\
 & \leq & (K'M/r)^{N}\max_{u\in\mathcal{S}_{\epsilon}}\mathbb{P}[\langle u,H^{N}u\rangle>r/2].
\end{eqnarray*}
We conclude by noticing that, as a linear combination of independent
sharp sub-Gaussian random variables, $\langle u,H^{N}u\rangle$ is
itself sharp sub-Gaussian of variance 
\[
\frac{1}{N}\sum_{i,j}\Sigma_{i,j}^{N}u_{i}^{2}u_{j}^{2}=\frac{1}{N}\langle\rho(u),S\rho(u)\rangle\le\frac{\epsilon}{N}
\]
since $u$ belongs to   $\mathcal{S}_{\epsilon}$. This 
 implies that
\[
\mathbb{P}[\langle u,H^{N}u\rangle>r/2]\leq e^{-Nr^{2}/4\epsilon}\,.
\]
\end{proof}
We use the following upper bound on $I_{\sigma}$ whose proof is given
in Section \ref{secsimp}: 
\begin{lem}
\label{lem:Ratefunctionbound} For every $x>r_{\sigma}$, 
\[
I_{\sigma}(x)\leq\frac{x^{2}}{2a}.
\]

\end{lem}

We can finally complete the proof of Proposition \ref{prop:epsilon} 
\begin{proof}[Proof of Proposition \ref{prop:epsilon}]
 It is enough to show that $-I_{\sigma}(x)>\max\{C'+\log M+\log r-\frac{r^{2}}{4\epsilon},C-M^{2}/A\}$
with $r=x-\delta$ and $\epsilon$ small enough since then this implies
that the infimum over $\psi\in\mathcal{P}_{p}$ is taken over $\psi$
so that $\langle\psi,S\psi\rangle\ge\epsilon$. This is easily done thanks to Lemma \ref{lem:Ratefunctionbound} by 
choosing $M>0$ such
that $\frac{M^{2}}{A}-C\geq\frac{x^{2}}{4a}+1$ and then $\epsilon>0$
small enough such that $\log r_{\sigma}+\log M+\frac{r_{\sigma}^{2}}{4\epsilon}-C'\geq\frac{x^{2}}{4a}+1$. 

\end{proof}

\section{Proof of Theorem \ref{maintheo-constant}, Large deviation lower
bound}

\label{sec:LDlowerbound}

In this section we will prove Proposition \ref{prop:LDPLB}. { It will rely on the following key Proposition. }
\begin{prop}
\label{prop:Add_lambda_1}For every  $x>r_{\sigma}$, every $\varepsilon,\delta>0$  and every $\psi\in\mathcal{P}_{p}$
such that $\langle\psi,S\psi\rangle\ge\varepsilon $ there exist
$\theta_{*}=\theta(x,\psi)\geq0$ such that 
\begin{align*}
 & \frac{1}{N}\log\mathbb{E}_{H}\mathbb{E}_{u}[\mathds{1}_{|\lambda_{1}-x|\leq\delta,||\rho(u)-\phi(\theta_{*},x,\psi)||\leq\kappa,H^{N}\in\Omega_{N}}e^{N\theta_{*}\langle u,H^{N}u\rangle}]\\
 & \quad=\frac{1}{N}\log\mathbb{E}_{H}\mathbb{E}_{u}[\mathds{1}_{||\rho(u)-\phi(\theta_{*},x,\psi)||\leq\kappa}e^{\theta_{*}N\langle u,H^{N}u\rangle}]+\eta_{\delta,\varepsilon}(N,\kappa).
\end{align*}
where 
for every  $\delta,\varepsilon>0$,  $\eta_{\delta,\varepsilon}(N,\kappa)$ goes to $0$ as $N\rightarrow\infty$ and then $\kappa\rightarrow0$.
\end{prop}

The main part of this section is to prove this proposition but for
the moment we assume it is true and finish the proof of the lower
bound. 
\begin{proof}
[Proof of Theorem \ref{maintheo-constant}, lower bound] First, using
the same definition for $\Omega_{N}$ as in section \ref{sec:LDupperbound},
for any $x>r_{\sigma}$,$\delta>0$, $\kappa>0$, and for any $\psi\in\mathcal{P}_{p}$
and $\langle\psi,S\psi\rangle\ge\varepsilon$ and $\theta\geq0$ as in Proposition
\ref{prop:Add_lambda_1} : 
\begin{align*}
 & \frac{1}{N}\log\mathbb{P}[|\lambda_{1}-x|\leq\delta]\\
 & \quad\geq\frac{1}{N}\log\mathbb{E}[\mathds{1}_{|\lambda_{1}-x|\leq\delta,H^{N}\in\Omega_{N}}\frac{I_{N}(H,\theta)}{I_{N}(H,\theta)}]\\
 & \quad=\frac{1}{N}\log\mathbb{E}[\mathds{1}_{|\lambda_{1}-x|\leq\delta,H^{N}\in\Omega_{N}}I_{N}(H,\theta)]-J_{\mu_{\sigma}}(\theta,x)+ {\eta_\delta(N)'}\\
 & \quad=\frac{1}{N}\log\mathbb{E}_{H}\mathbb{E}_{u}[\mathds{1}_{|\lambda_{1}-x|\leq\delta,H^{N}\in\Omega_{N}}e^{N\theta\langle u,H^{N}u\rangle}]-J_{\mu_{\sigma}}(\theta,x)+ {\eta_\delta(N)'}\\
 & \quad\geq\frac{1}{N}\log\mathbb{E}_{H}\mathbb{E}_{u}[\mathds{1}_{|\lambda_{1}-x|\leq\delta,H^{N}\in\Omega_{N}}\mathds{1}_{||\rho(u)-\phi(\theta,x,\psi)||\leq\kappa}e^{N\theta\langle u,H^{N}u\rangle}]-J_{\mu_{\sigma}}(\theta,x)+ {\eta_\delta(N)'}
\end{align*}
where $ { \eta_\delta(N)'}$ comes from the uniform convergence of $N^{-1} \log I(H^N,\theta)$ toward $J_{\mu_{\sigma}}(\theta,x)$. 
 In the case where $\theta_{*}=\theta(x,\psi)$ we get by Proposition \ref{prop:Add_lambda_1} that
\begin{align*}
 & \frac{1}{N}\log\mathbb{P}[|\lambda_{1}-x|\leq\delta]\\
 & \quad\geq\frac{1}{N}\log\mathbb{E}_{H}\mathbb{E}_{u}[\mathds{1}_{||\rho(u)-\phi(\theta_{*},x,\psi)||\leq\kappa}e^{N\theta_{*}\langle u,H^{N}u\rangle}]-J_{\mu_{\sigma}}(\theta_{*},x)+ { \eta_\delta(N)'} + \eta_\kappa(N)\\
 & \quad=K(\theta_{*},\phi(\theta_{*},x,\psi))-J_{\mu_{\sigma}}(\theta_{*},x)+ { \eta_\delta(N)'} + \eta_\kappa(N)\\
 & \quad\geq\inf_{\theta\geq0}-\mathcal{F}(\theta,x,\psi,\sigma)+{ \eta_\delta(N)'}+ \eta_{\delta,\varepsilon}(N,\kappa)
\end{align*}
where $\eta_{\delta,\varepsilon}(N,\kappa)$ goes to zero when $N$ goes to infinity
and then $\kappa$ goes to zero. Because this inequality
holds for every $\psi\in\mathcal{P}_{p}$ so that  $\langle\psi,S\psi\rangle\ge\varepsilon$, every $\varepsilon>0$ and $\delta>0$, 
we finally get the lower bound of Theorem \ref{maintheo-constant}.
\end{proof}
We finally prove Proposition \ref{prop:Add_lambda_1}. Let $\phi\in\mathbb{{\cal P}}_{p}$,
and let $v$ be a  random vector uniformly distributed on the set $\{v\in\mathbb{S}^{N-1}\colon\rho(v)=\phi\}$.
This random vector can also be constructed as follow : we set 
$v=\sum_{i=1}^{p}v_{k}$ where $(\phi_{k}^{-1/2}v_{k})$ are independent
vectors distributed uniformly on 
\[
\mathbb{S}^{I_{k}^{N}}:=\mathbb{S}^{N-1}\cap\{x\in\mathbb{R}^{N}:\forall i\notin I_{k}^{N},x_{i}=0\}\,.
\]
respectively. We denote this law $\mathbb{P}_{v}^{\phi}$ and $\mathbb{E}_{v}^{\phi}$
the associated expectation. 
\begin{lem}
\label{lem:Phii-uniform}For any $\theta\geq0$, $\phi\in{\cal P}_{p}$
and any measurable set $A$ of $\mathcal{S}_{N}$, we have 

\begin{align*}
 & \frac{1}{N}\log\mathbb{E}_{H}\mathbb{E}_{u}[e^{\theta N\langle u,H^{N}u\rangle}\mathds{1}_{H^{N}\in A\cap\Omega_{N},||\rho(u)-\phi||\leq\kappa}]\\
 & \qquad=\frac{1}{N}\log\mathbb{E}_{H}\mathbb{E}_{v}^{\phi}[e^{\theta N\langle v,H^{N}v\rangle}\mathds{1}_{H^{N}\in A\cap\Omega_{N}}]-d_{KL}(|I|\|\phi)+\eta_\kappa(N)
\end{align*}
where $\eta_\kappa(N)$ goes to $0$ for $N$ going to $\infty$ and then $\kappa$ going to $0$.
\end{lem}

\begin{proof}
For $u$ uniform on $\mathbb{S}^{N-1}$ we can write $u=\sum_{k=1}^{p}\sqrt{\rho(u)_{k}}\frak{u}^{k}$
with $\frak{u}^{k}:=(u_{i}/\sqrt{\rho(u)_{k}})_{i\in I_{k}^{N}}$
random uniform unit vectors on $\mathbb{S}^{I_{k}^{N}}$, which are
jointly independent with the random variable $\rho(u)$. If $v_{k}:=\phi_{k}^{1/2}\frak{u}^{k}$
, then the vector $v$ is distributed according to  $\mathbb{P}_{v}^{\phi}$ and we have
\[
\Big|\langle u,H^{N}u\rangle-\langle v,H^{N}v\rangle\Big|=\Big|\sum_{k,l=1}^{p}(\sqrt{\rho(u)_{k}\rho(u)_{l}}-\sqrt{\phi_{k}\phi_{l}})\langle\frak{u}^{k},H^{N}\frak{u}^{l}\rangle\Big|\leq{\cal C}p^{2}\kappa.
\]
with ${\cal C}$ as in the definition of $\Omega_{N}$. The Lemma
then follows from Lemma \ref{lem:rho_LD}.
\end{proof}
For $v$ a unit vector we define the distribution  $\mathbb{P}^{\theta,v}$ on ${\cal S}_{N}$
as

\[
d\mathbb{P}^{\theta,v}(H)=\frac{e^{\theta N\langle v,Hv\rangle}}{\mathbb{E}_{H}[e^{\theta N\langle v,H^{N}v\rangle}]}d\mathbb{P}(H).
\]
We find in the spirit of \cite[Section 5.1]{HuGu1}, that when the
vector $v$ has sufficiently small uniform norm, $\mathbb{P}^{\theta,v}$
is approximately the law of the original matrix tilted by a rank one
perturbation, namely: 
\begin{lem}
\label{lem2-1} \cite[Proof of Lemma 5.2]{HussonVar} For $v$ a unit
vector with $\|v\|_{\infty}\leq N^{-3/8}$ under $\mathbb{P}^{\theta,v}$
we can write
\[
H^{N}=\widetilde{H}^{N}+2\theta E+\Delta_{N}\,,\quad E_{ij}:=\Sigma_{ij}^{N}v_{i}v_{j}
\]
where $\widetilde{H}^{N}$ is a centered matrix with the same variance
profile $\Sigma^{N}$ as $H^{N}$ under $\mathbb{P}$, and $\Delta_{N}$
is such that for  every $\epsilon>0$ 
\[
\lim_{N\to\infty}\mathbb{P}^{\theta,v}[\|\Delta_{N}\|_{op}>\epsilon]=0.
\]
\end{lem}

We write $\phi(\theta):=\phi(\theta,x,\psi)$ and consider the following
equation of unknown $z$ on $[r_{\sigma},+\infty[$ 
\begin{equation}
\det(I_{p}-2\theta SD(\theta,z))=0\label{eqn:lambda-1}
\end{equation}
where $S\in\mathbb{R}^{p\times p}$ with $S_{kl}:=\sigma_{kl}$ and
$D(\theta,z)=\text{diag}((G_{\mu_{\sigma,k}}(z)\phi(\theta)_{k})_{1\leq k\leq p})$
and $I_{p}$ is the identity matrix. We define $z(\theta)$ as the
following

\[
z(\theta):=\begin{cases}
\text{largest solution to equation \eqref{eqn:lambda-1} in [\ensuremath{r_{\sigma}},\ensuremath{\infty}) if there is one}\\
\text{\ensuremath{r_{\sigma}} if no such solution exists.}
\end{cases}
\]
 
\begin{lem}
\label{lem:z(theta)-prediction} \cite[Proof of Lemma 5.3]{HussonVar}
Let $v$ a unit vector with $\|v\|_{\infty}\leq N^{-3/8}$ and $\rho(v)=(\phi(\theta)_{1},\dots,\phi(\theta)_{p})$.
For every $\epsilon>0$, we have

\[
\lim_{N\to\infty}\mathbb{P}^{\theta,v}[|\lambda_{1}(H^{N})-z(\theta)|\geq\epsilon]=0.
\]
\end{lem}
Hence, our main question is to check that for every $\psi$ we can find $\theta=\theta(x,\psi)$ so that $z(\theta)=x$. This will prove 
the existence of $\theta(x,\psi)$ in Proposition \ref{prop:Add_lambda_1}. It 
will be obtained by an intermediate value Theorem. We need the following.
\begin{lem}
\label{lem:continuous-find_theta}The function 
\[
\nu(\theta):=2\theta\lambda_{1}\left(\sqrt{D(\theta,x)}S\sqrt{D(\theta,x)}\right)
\]
 is continuous and satisfies $\nu(0)=0$ and $\lim_{\theta\rightarrow\infty}\nu(\theta)=\infty$.
\end{lem}

\begin{proof}
[Proof of Lemma \ref{lem:continuous-find_theta}]

This is accomplished by noticing that since $\phi(\theta):=\phi(\theta,x,\psi)$
is continuous in $\theta$, so is $D(\theta,x)$ and so is $\theta\mapsto2\theta\rho(\theta,x)$.
For $\theta=0$, we simply have $0\times\lambda_{1}(\sqrt{D(\theta,x)}S\sqrt{D(\theta,x)})=0$.
We now consider the limit $\theta\rightarrow\infty$. Since $\langle\psi,S\psi\rangle>0$,
there must be at least one couple $(k,l)$ such that $\sigma_{kl}\psi_{k}\psi_{l}>0$.
Finally, for $2\theta>G_{\mu_{\sigma}}(x)$, since $SD(\theta,x)$
has positive entries, we have

\[
\text{Tr}(\sqrt{D(\theta,x)}S\sqrt{D(\theta,x)})\geq\sigma_{kl}\sqrt{G_{\mu_{\sigma,k}}(x)G_{\mu_{\sigma,l}}(x)}\sqrt{\phi(\theta)_{k}\phi(\theta)_{l}}
\]
and by definition of $\phi(\theta)$, 
\[
\phi(\theta)_{k}\geq(1-\frac{G_{\sigma}(x)}{2\theta})\psi_{k}
\]
so we obtain
\[
\text{Tr}(\sqrt{D(\theta,x)}S\sqrt{D(\theta,x)})\geq\frac{1}{2}\sigma_{kl}\sqrt{\psi{}_{k}\psi_{l}}(1-\frac{G_{\sigma}(x)}{2\theta})\min_{i}G_{\mu_{\sigma,i}}(x).
\]
Last, we use that the largest eigenvalue is larger than $p^{-1}\text{Tr}(\sqrt{D(\theta,x)}S\sqrt{D(\theta,x)})$
and we have: 
\[
\nu(\theta)\geq\frac{\theta}{2p}\sigma_{kl}\sqrt{\psi{}_{k}\psi_{l}}(1-\frac{G_{\sigma}(x)}{2\theta})\min_{k}G_{\mu_{\sigma,k}}(x)
\]
and we can conclude that $\lim_{\theta\rightarrow\infty}\nu(\theta)=\infty.$
\end{proof}
We can finally finish the proof of Proposition \ref{prop:Add_lambda_1}.
\begin{proof}
[Proof of Proposition \ref{prop:Add_lambda_1}] With Lemma \ref{lem:continuous-find_theta}
by the  intermediate value Theorem there is a  $\theta_{*}$ such that $\nu(\theta_{*})=1$. { Therefore, by definition, $\theta_*$ is a solution of the equation \eqref{eqn:lambda-1}. Furthermore, one can notice that entrywise the function $z \mapsto \sqrt{D( \theta_*, z)} S \sqrt{D( \theta_*, z)}$ is positive and strictly decreasing. Since $\sqrt{D( \theta_*, z)} S \sqrt{D( \theta_*, z)}$ is symmetric, it implies that $2 \theta_* \lambda_1( \sqrt{D( \theta_*, z)} S \sqrt{D( \theta_*, z)}) <1 $ for $z > x$ and therefore $z(\theta_{*})=x$}. Let $W_{N}:=\{u\in\mathbb{S}^{N-1}:\|u\|_{\infty}\leq N^{-3/8}\}$
using \cite[Lemma 4.3]{HussonVar} we have 
\begin{equation}
\frac{\mathbb{E}_{H}\mathbb{E}_{v}^{\phi}[\mathds{1}_{v\in W_{N}}e^{\theta_{*}N\langle v,H^{N}v\rangle}]}{\mathbb{E}_{H}\mathbb{E}_{v}^{\phi}[e^{\theta_{*}N\langle v,H^{N}v\rangle}]}=1-o(1)\label{eq:annealed_delocalized}
\end{equation}
and one can refer to \cite[Lemma 5.1]{HuGu1} to prove that one can
choose ${\cal C}$ large enough in $\Omega_{N}$ so that

\begin{equation}
\mathbb{P}^{\theta_* }[H^{N}\in\Omega_{N}]=1-o(1).\label{eq:Omega_theta}
\end{equation}
Therefore one can write : 
\begin{align*}
 & \mathbb{E}_{H}\mathbb{E}_{v}^{\phi}[e^{\theta_{*}N\langle v,H^{N}v\rangle}\mathds{1}_{|\lambda_{1}-x|\leq\delta,H^{N}\in\Omega_{N}}]\\
 & \qquad\geq\mathbb{E}_{v}^{\phi}\left[\mathds{1}_{v\in W_{N}}\mathbb{E}_{H}[e^{\theta_{*}N\langle v,H^{N}v\rangle}\mathds{1}_{|\lambda_{1}-x|\leq\delta,H^{N}\in\Omega_{N}}]\right]\\
 & \qquad=\mathbb{E}_{v}^{\phi}\left[\mathds{1}_{v\in W_{N}}\mathbb{P}^{\theta_{*},v}(|\lambda_{1}-x|\leq\delta,H^{N}\in\Omega_{N})\mathbb{E}_{H}[e^{\theta_{*}N\langle v,H^{N}v\rangle}]\right]\\
 & \qquad=\mathbb{E}_{v}^{\phi}[\mathds{1}_{v\in W_{N}}\mathbb{E}_{H}[e^{\theta_{*}N\langle v,H^{N}v\rangle}]]e^{o(N)}\\
 & \qquad=\mathbb{E}_{v}^{\phi}\mathbb{E}_{H}[e^{\theta_{*}N\langle v,H^{N}v\rangle}]e^{o(N)}
\end{align*}
where we use first the definition of $\mathbb{P}^{\theta_{*},v}$
then Lemma \ref{lem:z(theta)-prediction}, \eqref{eq:Omega_theta}
and finally \eqref{eq:annealed_delocalized}. We now use Lemma \ref{lem:Phii-uniform}
to conclude Proposition \ref{prop:Add_lambda_1}.
\end{proof}

\section{Study of the rate function}\label{ratef}
Recall the rate function $I_{\sigma}$ defined in Theorem \ref{maintheo}. Note that when $\sigma$ is piecewise constant, this definition agrees with the definition from  Theorem \ref{maintheo-constant} since in this case we can take $\psi$ with constant density on each intervals by concavity of $H$.
\label{sec:rate} 
\begin{prop}
\label{prop:rate} $\,$
\begin{enumerate}
\item $x\mapsto I_{\sigma}(x)$ is a good rate function. 
\item $x\mapsto I_{\sigma}(x)$ is continuous on $(r_{\sigma},+\infty[$. 
\item Assume that $\sigma$ is piecewise constant. Then $x\mapsto I_{\sigma}(x)$ is strictly increasing on $[r_{\sigma},+\infty[$. 
\end{enumerate}
\end{prop}

To prove the continuity of Proposition \ref{prop:rate}(2), we shall prove the following : 
\begin{lem}
For $B>0$, there exists some  finite real number $\theta_{B}$ and a positive real number $\epsilon_{B}$ such that for  every $x\leq B$, we have
\begin{eqnarray*}
I_{\sigma}(x) & = & \inf_{{\psi\in\mathcal{P}([0,1])\atop \langle\psi,S\psi\rangle\geq\epsilon_{B}}}\sup_{\theta\in[0,\theta_{B}]}\mathcal{F}(\theta,x,\psi,\sigma)
\end{eqnarray*}
Here, $\mathcal F$ is defined in \eqref{defF0}. 
\end{lem}
 
\begin{proof}{
We first notice that for $\theta\ge \theta_{x}=G_{\mu_{\sigma}}(x)/2 $, 
$$d\phi(x,\theta,\psi)= \frac{1}{2\theta} G_{\mu_{\sigma,t}}(x)dt+(1-\frac{\theta_{x}}{\theta})d\psi(t),$$
 so that

$$\theta^{2}\iint \sigma(s,t) d \phi(\theta,x,\psi)(s)d \phi(\theta,x,\psi)(t)=
a(\theta-\theta_{x})^{2}+2b(\theta-\theta_{x}) +c$$
with $a: =\iint_{[0,1]^{2}}\sigma(u,t)d\psi(u)d\psi(t)$ and
$$ b:=\iint_{[0,1]^{2}}\sigma(t,u)G_{\mu_{\sigma,t}}dt d\psi(u),\mbox{ and }c:=\frac{1}{4}\iint \sigma(s,t) G_{\mu_{\sigma,t}}(x) dt G_{\mu_{\sigma,s}}ds\,.$$
As a consequence, we find that 
\[
\frac{\partial\mathcal{F}}{\partial\theta}(\theta,x,\psi,\sigma)=x-\frac{1}{2\theta}-2a(\theta-\theta_{x})-2b+\frac{\partial(H\circ\phi)}{\partial\theta}(\theta,x,\psi)
\]
with $H$ defined in \eqref{defKH}. 
 Using that $\mathcal{F}(\theta_{x},x,\psi,\sigma)=0$ by Lemma 
\ref{lem:theta}, we deduce that for $\theta\geq \theta_{x}$
\begin{eqnarray}
\mathcal{F}(\theta,x,\psi,\sigma) & = & -a(\theta-\theta_{x})^{2}-2b(\theta-\theta_{x})+(\theta-\theta_{x})x-\frac{1}{2}\log(\theta/\theta_{x})\label{eqn:Fupperbound}\\
 &  & +(H(\phi(\theta,x,\psi))-H(G_{\sigma,t}(x)dt))\,.\nonumber 
\end{eqnarray}

Moreover, for $\theta\geq\theta_{x}$, we find that 
\[
H(\phi(\theta,x,\psi))-\frac{1}{2} \log(2\theta)=-\frac{1}{2}\log\int_{0}^{1}\Big(G_{\sigma,t}(x)+(2\theta-G_{\mu_{\sigma}}(x))\frac{d\psi}{dt}(t)\Big)dt
\]
so that $\theta\mapsto H(\phi(\theta,x,\psi))-\frac{1}{2}\log(2\theta)$
is decreasing since it is convex, with increasing derivative going to zero when $\theta$ goes to infinity. We therefore deduce  that
\begin{equation} \label{eqn:Fupperbound}
\mathcal{F}(\theta,x,\psi,\sigma)\leq-a(\theta-\theta_{x})^{2}+(\theta-\theta_{x})x=-a(\theta-\theta_{x})(\theta-\theta_{x}-\frac{x}{a}).
\end{equation}}
 For $\psi$
such that $a=\langle\psi,S\psi\rangle\geq\epsilon$, and  $\theta\geq\theta_{x}+\frac{\epsilon}{x}=:\theta_{*}(x,\epsilon)$, we deduce that $\mathcal{F}(\theta,x,\psi,\sigma)$ is non-positive.
Since we saw that we can take $\epsilon_{x}$ decreasing and $\theta_{*}(x,\epsilon)$, we conclude that we can take $\theta_{*}=\sup_{x\le B }\theta_{*}(x,\epsilon_{x})$ bounded and the conclusion follows that for all $x\le B$, there exists $\epsilon_{B}>0$ and $\theta_{B}$ finite such that:
\begin{eqnarray*}
I_{\sigma}(x) & = \inf_{{\psi\in\mathcal{P}([0,1])\atop \langle\psi,S\psi\rangle\geq\epsilon_{B}}}\sup_{\theta\geq0}\mathcal{F}(\theta,x,\psi,\sigma)=\inf_{{\psi\in\mathcal{P}([0,1])\atop \langle\psi,S\psi\rangle\geq\epsilon_{B}}}\sup_{\theta\in[0,\theta_{B}]}\mathcal{F}(\theta,x,\psi,\sigma)\,.
\end{eqnarray*}
\end{proof}
\begin{proof}[Proof of the second point of Proposition \ref{prop:rate}]
One only needs to notice that the family of functions $x\mapsto\mathcal{F}(\theta,x,\psi,\sigma)$
defined on $[r_{\sigma},B]$ are equi-continuous for $\theta\in[0,\theta_{*}]$.
Hence since we can take $\langle\psi,S\psi\rangle\geq\epsilon_{B}>0$
uniform for $x$ in compact sets of $(r_{\sigma},+\infty)$, the conclusion
follows. 
\end{proof}
\begin{proof}[Proof of the first point of Proposition \ref{prop:rate}]
By { Lemma \ref{lem:Measure-Concentration}, $I_{\sigma}(r_{\sigma})=0$} whereas Theorem
\ref{lem:Measure-Concentration} implies that $I_{\sigma}(x)=+\infty$
on $(-\infty,r_{\sigma})$. Moreover, $I_{\sigma}$ is non negative.
As a consequence of the continuity of $I_{\sigma}$ on $(r_{\sigma},+\infty)$,
we conclude that $I_{\sigma}$ is lower semi-continuous. Moreover,
taking $\theta=\kappa x$ and using that $\sigma$ is uniformly bounded
above, we find that there exists $\chi>0$ and $C$ finite such that
for every real number $x$, we have 
\[
I_{\sigma}(x)\ge\chi x^{2}+C.
\]
As a consequence, the level sets of $I_{\sigma}$ are included in
compacts and therefore $I_{\sigma}$ is a good rate function. 
\end{proof}

We now turn to the last point of Proposition  \ref{prop:rate}. It is restricted for the time being to the case where $\sigma$ is piecewise constant because we rely on Theorem \ref{maintheo-constant}, but generalizes to the piecewise continuous case as soon as we have the associated weak large deviation principle. Indeed, the weak large deviation principle  of Theorem \ref{maintheo-constant} and the exponential tightness \eqref{expt} imply readily the full large deviation principle and in particular, since $I_{\sigma}$ is continuous on $(r_{\sigma},\infty)$, that
\begin{corollary}\label{cor} Assume that $\sigma$ is piecewise constant. Then, for every $x>r_{\sigma}$,
\[
\limsup_{N\rightarrow\infty}\frac{1}{N}\log\mathbb{P}(\lambda_{1}\geq x)=-\inf_{y\geq x}{I}_{\sigma}(y)=:-\check{I}_{\sigma}(x).
\]
\end{corollary}
The monotonicity of $I_{\sigma}(.)$ will follow from the following.
\begin{lem}
\label{lem:inc} 
For any $r_{\sigma}\le x\leq y$, we have $\check{I}_{\sigma}(x)\leq\frac{x^{2}}{y^{2}}\check{I}_{\sigma}(y)$. 
\end{lem}

\begin{proof}
For $s\in[0,1]$, we consider the submatrix $(H^{N})^{(s)}\in\mathbb{R}^{t(N)\times t(N)}$
of $H^{N}$ where for each $k$ we erase $\lfloor(1-s)\#I_{k}^{N}\rfloor$
lines and columns of indices $i\in I_{k}^{N}$. $\lim_{N\to\infty}N^{-1}t(N)=s$.
As $(H^{N})^{(s)}$ is  a submatrix of $H^{N}$, we have that $\lambda_{1}(H^{N})\geq\lambda_{1}((H^{N})^{(s)})$
and therefore, for any $x\in\mathbb{R}$,
\[
\mathbb{P}(\lambda_{1}(H^{N})\geq x)\geq\mathbb{P}(\lambda_{1}((H^{N})^{(s)})\geq x)
\]
Furthermore $(\sqrt{s}^{-1}(H^{N})^{(s)})_{N\geq0}$ is also a Wigner
matrix with the same profile variance as $H^{N}$ and we can apply
our large deviation principle. So we have 
\begin{eqnarray*}
-\check{I}_{\sigma}(x) & = & \limsup_{N\rightarrow\infty}\frac{1}{N}\log\mathbb{P}(\lambda_{1}(H^{N})\geq x)\\
 & \ge & \limsup_{N\rightarrow\infty}\frac{1}{N}\log\mathbb{P}(\lambda_{1}((H^{N})^{(s)})\geq x)\\
 & = & \limsup_{N\rightarrow\infty}\frac{t(N)}{N}\frac{1}{t(N)}\log\mathbb{P}(\lambda_{1}(\sqrt{s}^{-1}(H^{N})^{(s)})\geq\sqrt{s}^{-1}x)\\
 & = & -s\check{I}_{\sigma}(\sqrt{s}^{-1}x)
\end{eqnarray*}
And so we obtain 
\[
\check{I}_{\sigma}(x)\leq s\check{I}_{\sigma}(\frac{x}{\sqrt{s}})
\]
As a consequence, for every $x\leq y$, 
\[
\check{I}_{\sigma}(x)\leq\frac{x^{2}}{y^{2}}\check{I}_{\sigma}(y).
\]
\end{proof}
\begin{proof}[Proof of the last point of Proposition \ref{prop:rate}]
Since the rate function doesn't depend on the law of the entries,
we can assume here that they are for instance Gaussian. We can then
use the concentration result for $\lambda_{1}(H^{N})$ and then we
have, for every $\epsilon>0$, the existence of $c>0$ such that 
\[
P[\lambda_{1}(H^{N})\geq r_{\sigma}+\epsilon]\leq e^{-cN}\,.
\]
Therefore, by Corollary \ref{cor},  $\check{I}_{\sigma}(x)>0$ for every $x>r_{\sigma}$. With
Lemma \ref{lem:inc} we deduce that {
\[
\check{I}_{\sigma}(y)-\check{I}_{\sigma}(x)\ge (\frac{y^{2}}{x^{2}}-1)\check{I}_{\sigma}(x)>0
\]}
as soon as $y>x>r_{\sigma}$. Hence, $\check{I}_{\sigma}(.)$ is strictly
increasing on $(r_{\sigma},+\infty[$. Since

\[
\mathbb{P}[|\lambda_{1}(H^{N})-x|\leq\delta]=\mathbb{P}[\lambda_{1}(H^{N})\geq x-\delta]-\mathbb{P}[\lambda_{1}(H^{N})\geq x+\delta]
\]
Therefore if $x>r_{\sigma}$, one can find $\delta>0$ small enough
such that $\check{I}_{\sigma}(x-\delta)<\check{I}_{\sigma}(x+\delta)$.
Then we have

\[
\lim_{N\to\infty}\frac{1}{N}\log\mathbb{P}[|\lambda_{1}(H^{N})-x|\leq\delta]=-\check{I}_{\sigma}(x-\delta)
\]
This implies that $I_{\sigma}(x)=\lim_{\delta\to0}\check{I}_{\sigma}(x-\delta)$
is strictly increasing in $x>r_{\sigma}$. 
\end{proof}

\section{The piecewise continuous case}

\label{sec:Piecewise-continuous}

\label{sec:approximation} If $\sigma$ is piecewise continuous, we
can approximate it by piecewise constant functions for the uniform
norm. Then we can prove that the largest eigenvalue of these approximations provide exponential approximations
for the largest eigenvalue under study \cite{DZ}. We proceed as in \cite[section 6]{HussonVar}. Let $H^{N}=(\frac{X_{i,j}}{\sqrt{N}})$
be a sequence of Wigner type matrices such that for any $N\in{\mathbb{N}}$,
$1\leq i,j\leq N$, 
\[
\Sigma_{i,j}^{N}:=\text{Var}(X_{i,j})2^{-\mathds{1}_{i=j}}\,.
\]

{ Let us build a sequence (in $p \in \mathbb N$) of uniform approximations $(H^N)^{(p)}$ of $H^N$ such that each $(H^N)^{(p)}$ has a piecewise constant variance profile. For this, let $\sigma_p$ be any sequence of piecewise constant function of $[0,1]\times [0,1]$ such that $(\sigma_p)_{p \in\mathbb{N}}$ goes uniformly to $\sigma$. We let $t_{i}^{N}$ be a sequence satisfying \eqref{convcov} and set
\[ \Sigma_{i,j}^{N,p} = \sigma_p ( t_i^N, t_j^N) \]

We set 
\[
(H^{N})^{(p)}(i,j)=\left(\frac{\Sigma_{i,j}^{N,p}}{\Sigma_{i,j}^{N}}\right)^{\frac  1 2 }H^{N}(i,j)
\]
By Lemma \ref{lem:Measure-Concentration}, the empirical measure of the eigenvalues of $(H^{N})^{(p)}$ converges weakly  towards $\mu_{\sigma_{p}}$ almost surely.
Moreover, we can notice the following convergence of  the measures $\mu_{\sigma_p,t}$. 

\begin{prop}\label{prop:Stieltjesconv} 
	For any compact set $K\subset \mathbb{C} \setminus \mathbb{R}$, the sequence of functions $G_{\sigma_p} : [0,1] \times \mathbb{C} \setminus \mathbb{R} \to \mathbb{C}$ that maps $(t,z)$ to $G_{\sigma_p, t}(z)$ converges toward $G_{\sigma}$ for the uniform distance.

	\end{prop} 
This is in fact a simple application of the proof of Theorem 1.4 in \cite{HussonVar}. In particular one only has to notice that the $\mathbb{H}_{\eta}^+$ for $\eta>0$ are a compact exhaustion of $\mathbb{H}^+$ and that the topology induced by the distance $d_{\mathbb{H}_{\eta}^+}$ is equivalent to the topology induced by the uniform distance on the set of measurable functions from  $[0,1] \times \mathbb{H}_{\eta}^+$ to $\mathbb{C}$. 
}
Then we have the following approximation lemma, which easily follows
from the uniform convergence for instance by using \cite[Lemma 5.6]{HuGu1}: 
\begin{lem}
\label{approx} We have for any $\epsilon>0$ 
\[
\limsup_{n\to\infty}\limsup_{N}\frac{1}{N}\log\mathbb{P}[||(H^{N})^{(p)}-H^{N}||_{\infty}>\epsilon]=-\infty\,.
\]
As a consequence, 
\[
\limsup_{N}\frac{1}{N}\log\mathbb{P}[|\lambda_{\max}((H^{N})^{(p)})-\lambda_{\max}(H^{N})|>\epsilon]=o_{p}(\epsilon)
\]
where $o_{p}(\epsilon)$ goes to $-\infty$ when $p$ goes to $+\infty$
for any $\epsilon>0$. 
\end{lem}

This implies that the $(\lambda_{\max}((H^{N})^{(p)}))_{N\in{\mathbb{N}}}$
are  exponential approximations of $(\lambda_{\max}(H^{N}))_{N\in{\mathbb{N}}}$
(see definition 4.2.14 in \cite{DemZei2010}). On the other hand,
$\Sigma^{N,p}$ is a piecewise constant covariance. We can then apply
Theorem \ref{maintheo-constant} to $(H^{N})^{(p)}$ with piecewise
constant profile corresponding to $\sigma_p$.

\begin{lem}
The  law of the sequence of random variables $\lambda_{\max}((H^{N})^{(p)})$
satisfies a large deviation principle with speed $N$ and  good rate function $I_{\sigma_p}$\,. 
\end{lem}

Then we have the following Theorem, which is simply an application
of Theorem 4.2.14 from \cite{DemZei2010}. 
\begin{thm}
The law of $\lambda_{\max}(H^{N})$ satisfies a weak large deviation
principle with speed $N$ with rate function $\hat{I}$ defined for
$x\in\mathbb{R}$ by: 
\[
\hat{I}(x):=\sup_{\delta>0}\liminf_{p\to\infty}\inf_{|y-x|\leq\delta}I_{\sigma^{(p)}}(y).
\]
\end{thm}

We will prove the following result: 
\begin{thm}
\label{thm-ratefuct} For all $x\in\mathbb{R}$, 
\[
\hat{I}(x)=I_{\sigma}(x).
\]
\end{thm}

Then, to get the full LDP from the weak LDP, we only have to notice
that the sequence of random variables $\lambda_{\max}(H^{N})$ is
exponentially tight, which is well known, see \eqref{expt}.
The rest of this section is devoted to the proof of Theorem \ref{thm-ratefuct}
We recall that

\[
\mathcal{F}(\theta,x,\psi,\sigma)=J_{\mu_{\sigma}}(\theta,x)-K(\theta,\phi(\theta,x,\psi)).
\]

\begin{thm}
\label{thmconv} Let $\sigma$ be a piecewise continuous function.
Let $(\sigma_{n})_{n\in\mathbb{N}}$ be a sequence of piecewise constant
functions such that $(\sigma_{n})_{n\in\mathbb{N}}$ converges uniformly
toward $\sigma$ and $\sigma\neq0$. Let $x_{1},x_{2}$ be two real
numbers such that $r_{\sigma}<x_{1}<x_{2}$, and $M$ be a positive
real number. Then the sequence of functions $(\theta,x,\psi)\mapsto\mathcal{F}(\theta,x,\psi,\sigma_{n})$
converges uniformly toward $(\theta,x,\psi)\mapsto\mathcal{F}(\theta,x,\psi,\sigma)$
on $[0,M]\times[x_{1},x_{2}]\times\mathcal{P}([0,1])$. 
\end{thm}

\begin{proof}
For the sake of simplicity, we will denote $G_{t}^{(n)}$ and $G_{t}$
the following functions: 
\[
G_{t}^{(n)}:=G_{\sigma_{n},t}\text{ and }G_{t}:=G_{\sigma,t}.
\]
and $r_{\sigma}$ and $l_{\sigma}$ respectively the right and left
edges of the support of $\mu_{\sigma}$. As  the Stieltjes transform of a
probability measure with support included in $[l_{\sigma},r_{\sigma}]$,
we have for every $t\in[0,1]$, $z>r_{\sigma}$ 
\[
\frac{1}{z-l_{\sigma}}\leq G_{t}(z)\leq\frac{1}{z-r_{\sigma}}
\]
(the same is true for $\sigma_{n}$ and $G_{t}^{(n)}$). First, we
recall the following lemma, which shows the continuity of the spectral
distributions in terms of the variance profiles, see e.g. \cite[Lemma 6.4]{HussonVar}: 
\begin{lem}
\label{convn} $\mu_{\sigma_{n}}$ converges weakly towards $\mu_{\sigma}$
and 
\[
\lim_{n\to\infty}r_{\sigma_{n}}=r_{\sigma}\,.
\]
\end{lem}

Note that the $\mu_{\sigma}$'s are symmetric probability measures
(because the law of $H^{N}$ is symmetric). Therefore the left hand
side $l_{\sigma}$ of the support of $\mu_{\sigma}$ is equal to $-r_{\sigma}$
and is also continuous. As a consequence, we are going to prove that 
\begin{lem}
\label{lem:Gconvergence} Let $x_{1}$ be a real number such that
$x_{1}>r_{\sigma}$. Under the hypotheses of Theorem \ref{thmconv},
the sequence of functions $(t,z)\mapsto G_{t}^{(n)}(z)$ converges
uniformly on $[0,1]\times[x_{1},+\infty]$ toward $(t,z)\mapsto G_{t}(z)$
and the sequence of functions $(t,z)\mapsto\frac{G_{t}^{(n)}(z)}{G^{(n)}(z)}$
converges uniformly on $[0,1]\times[x_{1},+\infty]$ toward $(t,z)\mapsto\frac{G_{t}(z)}{G(z)}$. 
\end{lem}

\begin{proof}[Proof of the Lemma]
Since for $n$ large enough, Lemma \ref{convn} implies that $r_{\sigma_{n}}<(r_{\sigma}+x_{1})/2$,
we have that for $n$ large enough, for every $t$, $z\mapsto G_{t}^{(n)}(z)$
is $4/(x_{1}-r_{\sigma})^{2}$-Lipschitz. Therefore, the desired uniform
convergence on any set of the form $[0,1]\times[x_{1},A]$ follows
from Proposition \ref{prop:Stieltjesconv}. To extend it to
$[0,1]\times[x_{1},+\infty]$, we only have to notice that for $n$
large enough, we have that for all $t\in[0,1]$ and $z>x_{1}$: 
\[
|G_{t}^{(n)}(z)|\leq\frac{1}{z-(x_{1}+r_{\sigma})/2}
\]
goes to $0$ as $z$ goes to $\infty$. The uniform convergence of
$(t,z)\mapsto\frac{G_{t}^{(n)}(z)}{G^{(n)}(z)}$ toward $(t,z)\mapsto\frac{G_{t}(z)}{G(z)}$
on $[0,1]\times[x_{1},M]$ follows  since $G(z)$ is bounded away from zero
uniformly on $[x_{1},M]$ for any finite $M$. Moreover, we notice
that for any $L<l_{\sigma}$ and $n$ large enough for any $t\in[0,1]$
and $z>x_{1}$: 
\[
\frac{1}{z-L}\leq G_{t}^{(n)}(z)\leq\frac{1}{z-(x_{1}+r_{\sigma})/2}\mbox{ and }\frac{1}{z-L}\leq G^{(n)}(z)\leq\frac{1}{z-(x_{1}+r_{\sigma})/2}\,,
\]
implies 
\[
\frac{z-(x_{1}+r_{\sigma})/2}{z-L}\leq\frac{G_{t}^{(n)}(z)}{G^{(n)}(t)}\leq\frac{z-L}{z-(x_{1}+r_{\sigma})/2}
\]
Both upper and lower bound go to $1$ at infinity, which implies again
the uniform convergence on $[0,1]\times[x_{1},+\infty[$. 
\end{proof}
We let $\mathcal{Q}(\theta,x,\psi,\sigma)$ be given by

\[
\mathcal{Q}(\theta,x,\psi,\sigma):=\theta^{2}\int\int\sigma(s,t)d\phi(\theta,x,\psi)(s)d\phi(\theta,x,\psi)(t)\,.
\]
Recalling the definition of $H$ from \eqref{defKH} and that $\phi(\theta,x,\psi)=\phi_{\sigma}(\theta,x,\psi)$
of \eqref{eqn:phicontinuous-1} depends on $\sigma$, we recall that
\begin{equation}
\mathcal{F}(\theta,x,\psi,\sigma_{n})=J_{\mu_{\sigma_{n}}}(\theta,x)-\mathcal{Q}(\theta,x,\psi,\sigma_{n})+H(\phi_{\sigma_{n}}(\theta,x,\psi))\,.\label{defF}
\end{equation}
To finish the proof of Theorem \ref{thmconv}, we are going to prove
the uniform convergence separately on the three terms of this sum.
We have, with $v=G^{-1}(2\theta)\vee x$, 
\begin{eqnarray*}
\mathcal{Q}(\theta,x,\psi,\sigma) & = & \int\int G_{t}(v)G_{s}(v)\sigma^{2}(s,t)dsdt\\
 &  & +2(\theta-G(v))\int\int G_{t}(v)d\psi(s)\sigma^{2}(s,t)dt\\
 &  & +(\theta-G(v))^{2}\int\int\sigma^{2}(s,t)d\psi(s)d\psi(t)
\end{eqnarray*}
and the same is true replacing $\sigma$ by $\sigma_{n}$, $G$ by
$G^{(n)}$, $G_{t}$ by $G_{t}^{(n)}$ and $v$ by $v_{n}=G_{n}^{-1}(2\theta)\vee x$.
By Lemma \ref{lem:Gconvergence}, $G_{t}^{(n)}(v)$ converges uniformly
toward $G_{t}(v)$ of $[0,1]\times[x_{1},x_{2}]$, $G^{(n)}(v)$ converges
uniformly toward $G(v)$ and $\sigma_{n}$ converges uniformly toward
$\sigma$. Moreover $(G_{t}^{(n)})'$ and $G'_{t}$ are uniformly
bounded above by a negative real number on compacts subsets of $[x_{1},+\infty)$
so that they are invertible. As a consequence, $v_{n}=G_{n}^{-1}(2\theta)\vee x$
converges uniformly towards $v=G^{-1}(2\theta)\vee x$ and belongs
to $[x_{1},\infty)$. We deduce that the function $(\theta,x,\psi)\mapsto\mathcal{Q}(\theta,x,\psi,\sigma_{n})$
converges uniformly toward $(\theta,x,\psi)\mapsto\mathcal{Q}(\theta,x,\psi,\sigma)$
on $[x_{1},x_{2}]\times[0,M]\times\mathcal{P}([0,1])$. The uniform
convergence of $(\theta,x)\mapsto J_{\mu_{\sigma_{n}}}(\theta,x)$
is proven in \cite{HussonVar}. We will now prove that for every $\psi\in\mathcal{P}([0,1])$,
the uniform convergence of $(\theta,x)\mapsto H(\phi_{\sigma_{n}}(\theta,x,\psi))$.
Let us denote in short $H_{n}(\phi(\theta,x,\psi))=H(\phi_{\sigma_{n}}(\theta,x,\psi))$.
We want to prove that for every $(\theta_{n},y_{n})$ such that $\theta_{n}\leq M$,
$y_{n}\in[x_{1},x_{2}]$ and $\lim_{n\to\infty}(\theta_{n},y_{n})=(\theta,y)$,
we have $\lim_{n\to\infty}H_{n}(\phi(\theta_{n},y_{n},\psi))=H(\phi(\theta,y,\psi))$.
First, using the convergence from Lemma \ref{lem:Gconvergence}, we
have that for every $t\in[0,1]$, $\frac{d\phi(\theta_{n},y_{n},\psi)}{dt}$
converges toward $\frac{d\phi(\theta,y,\psi)}{dt}$. Furthermore since
$(t,x)\mapsto G_{t}(x)/G(x)$ is bounded below on $[0,1]\times[x_{1},x_{2}]$,
we can find $a>0$, such that for $n$ large enough $G_{t}^{(n)}(x)/G^{(n)}(x)\geq a$
and the same thing is true for $(t,x)\mapsto G_{t}^{(n)}(x)$ so that
$G_{t}^{(n)}(x)\geq a$. Moreover, there exists a finite constant
$b$ such that for all $x>x_{1}$, $G_{t}^{(n)}(x)\leq b$. This implies
that for $(\theta,x,t)\in[0,M]\times[x_{1},x_{2}]\times[0,1]$, $\frac{d\phi(\theta_{n},y_{n},\psi)}{dt}\geq a(1\wedge1/2M)=:c_{M}^{{-1}}$.
Therefore,

\[
\Big|\log\Big(\frac{d\phi_{\sigma_{n}}(\theta_{n},y_{n},\psi)}{dt}\Big)-\log\Big(\frac{d\phi_{\sigma}(\theta,y,\psi)}{dt}\Big)\Big|\leq c_{M}|\frac{d\phi_{\sigma_{n}}(\theta_{n},y_{n},\psi)}{dt}-\frac{d\phi_{\sigma}(\theta,y,\psi)}{dt}|.
\]
Therefore, using the definition of \eqref{eqn:phicontinuous-1} ,
\begin{eqnarray*}
 &  & |\frac{d\phi_{\sigma_{n}}(\theta_{n},y_{n},\psi)}{dt}-\frac{d\phi_{\sigma}(\theta,y,\psi)}{dt}|\le\big(\frac{1}{2\theta}\Big|G_{\mu_{\sigma_{n}}}(x)-G_{\mu_{\sigma}}(x)\Big|+b|\frac{1}{\theta}-\frac{1}{\theta_{n}}|\big)\frac{d\psi}{dt}\\
 &  & +\big(\frac{1}{2\theta}|{G_{\sigma,t}(G_{\mu_{\sigma}}^{-1}(2\theta_{n})\vee x)}-G_{\sigma_{n},t}(G_{\mu_{\sigma_{n}}}^{-1}(2\theta_{n})\vee x)|+b|\frac{1}{\theta}-\frac{1}{\theta_{n}}|\big).
\end{eqnarray*}
If $\theta =0$, fix $\epsilon>0$ small. When $\theta_{n}\ge\epsilon$ for all $n$,
the previous considerations allow us to conclude that the above differences
go to zero when $n$ goes to infinity. Let us finally consider the
case where $\theta =0$. Up to taking subsequences we
can always assume that this is true for all $n$. But then, as we
discussed already 
\[
G_{\mu_{\sigma_{n}}}^{-1}(2\theta_{n})\simeq\frac{1}{2\theta_{n}}\mbox{ and }G_{\sigma_{n},t}(G_{\mu_{\sigma_{n}}}^{-1}(2\theta_{n})\vee x)\simeq2\theta_{n}
\]
so that $\frac{d\phi_{\sigma_{n}}(\theta_{n},y_{n},\psi)}{dt}$ is
close to one uniformly. The claim follows.

\end{proof}
\begin{thm}
Let $\sigma$ be piecewise continuous. Let us assume
that there is $\eta>0$ such that $\int\int\sigma(s,t)d\psi(s)d\psi(t)\geq\eta$.
Let $x_{1},x_{2},R,L$ be such that $r_{\sigma}<R<x_{1}<x_{2}$ and
$L<l_{\sigma}$. There exists $\theta_{*}>0$ that only depends on
$\eta,x_{1},x_{2},L,R$ such that for any $\psi\in\mathcal{P}([0,1])$,
$x\in[x_{1},x_{2}]$ 
\[
\sup_{\theta\geq0}\mathcal{F}(\theta,x,\psi,\sigma)=\sup_{\theta\in[0,\theta_{*}]}\mathcal{F}(\theta,x,\psi,\sigma).
\]
\end{thm}

\begin{proof}
We have the following upper bounds on each term on $\mathcal{F}(\theta,x,\psi,\sigma)$.
First, notice that for all $\theta\ge0$, all $x\le x_{2}$, 
\begin{equation}
J(\theta,x,\mu_{\sigma})\leq\theta x\leq\theta x_{2}\,.\label{bJ}
\end{equation}
Using the lower bound we assumed on $\sigma$, as well as the fact
that it is non-negative entrywise as well as $G(x)$, we have that
for $\theta\ge\theta_{x}=G_{\mu_{\sigma}}(x)/2$, 
\[
-\mathcal{Q}(\theta,x,\psi,\sigma)\leq-\theta^{2}(1-\frac{\theta_{x}}{\theta})^{2}\int\int\sigma(s,t)d\psi(s)d\psi(t)\le-\eta\theta^{2}(1-\frac{\theta_{x}}{\theta})^{2}
\]
Since $\theta_{x}\leq\frac{1}{x-r_{\sigma}}$, we have 
\begin{equation}
-\mathcal{Q}(\theta,x,\psi,\sigma)\leq-\eta(\theta-\frac{1}{(x-r_{\sigma})})^{2}.\label{bQ}
\end{equation}
Furthermore for any $\psi\in\mathcal{P}([0,1])$ with $v=x\vee G^{-1}(2\theta)$:

\[
\frac{d\phi(\theta,x,\psi)}{dt}=\Big(1-\frac{G(v)}{2\theta}\Big)_{+}\frac{d\psi}{dt}+\frac{G_{t}(v)}{2\theta}
\]
so 
\[
-\log\Big(\frac{d\phi(\theta,x,\psi)}{dt}\Big)\leq-\log(\frac{G_{t}(v)}{2\theta})=\log(2\theta)-\log(G_{t}(v))
\]
so that 
\begin{equation}
H(\phi(\theta,x,\psi))\le-\int_{0}^{1}\log(\frac{G_{t}(v)}{2\theta})dt\,.\label{gh}
\end{equation}
To bound the above right hand side, note that if $G^{-1}(2\theta)\ge x$
so that $\theta\le\theta_{x}\le\theta_{x_{1}}$ and $v=G^{-1}(2\theta)$,
we have that with $r_{\sigma}$ being the right edge of the support
of $\mu_{\sigma}$ and $l_{\sigma}$ its left edge: 
\[
\frac{1}{v-l_{\sigma}}\leq G_{t}(v)\leq\frac{1}{v-r_{\sigma}}\text{ and }\frac{1}{v-l_{\sigma}}\leq G(v)\leq\frac{1}{v-r_{\sigma}}
\]
so that 
\[
l_{\sigma}+\frac{1}{2\theta}\le v\le r_{\sigma}+\frac{1}{2\theta} \]
which implies that

\[ \frac{G_{t}(v)}{2\theta}\ge\frac{1}{2\theta(v-l_{\sigma})}\ge\frac{1}{2\theta_{x_{1}}(r_{\sigma}-l_{\sigma})+1}
\]
with $\theta_{x_{1}}=G(x_{1})/2\le\frac{1}{2(x_{1}-r_{\sigma})}$.
And so, one has that 

\[ \frac{G_{t}(v)}{2\theta}\ge \frac{x_1 -R }{ x_1 - L}\,.
\]
Therefore, we deduce from \eqref{gh} that 
\[
H(\phi(\theta,x,\psi))\leq \log( x_1 - L) - \log( x_1 - R).
\]

For $G^{-1}(2\theta)\le x$ ,so $v=G^{-1}(x)$, simply using that $G_t(v) = G_t(x) \geq 1/(x - L)  \geq (x_2 - L)$ we find similarly
that

\[
H(\phi(\theta,x,\psi))\le\log(2\theta) + \log( x_2 - L).
\]
Therefore, we deduce from \eqref{defF}, \eqref{bJ} and \eqref{bQ}
that

\[
\mathcal{F}(\theta,x,\psi,\sigma)\leq\theta x_{2}-\eta(\theta-\frac{1}{x_{1}-R})^{2}+(\log(2\theta) \vee 0)+C(x_{1},x_2,L,R)
\]
where 
$ C(x_{1},x_2,L,R) = \max( \log( x_1 - L) - \log( x_1 - R) , \log( x_2 - L))$.
Therefore, there exists a $\theta_{*}$ that depend only on $x_{1},x_{2},\epsilon,R,L$
such that the right hand side in negative on $[\theta_{*},+\infty[$.
Since $\mathcal{F}(0,x,\psi,\sigma)=0$ this implies that

\[
\sup_{\theta\geq0}\mathcal{F}(\theta,x,\psi,\sigma)=\sup_{\theta\in[0,\theta_{*}]}\mathcal{F}(\theta,x,\psi,\sigma).
\]
\end{proof}
\begin{prop}
Le $(\sigma_{n})$ be a sequence of variance profiles that converges
uniformly toward a continuous variance profile $\sigma\neq0$ . Let
$x_{1},x_{2}$ be such that $r_{\sigma}<x_{1}<x_{2}$. Then $I_{\sigma_{n}}$
converges uniformly toward $I_{\sigma}$ on $[x_{1},x_{2}]$. 
\end{prop}

\begin{proof}
First, for $\epsilon>0$, Let us define 
\[
I_{\sigma}^{\epsilon}(x):=\inf_{{\psi\in\mathcal{P}([0,1])\atop \langle\psi,\sigma_{n}\psi\rangle\geq\epsilon}}\sup_{\theta\geq0}\mathcal{F}(\theta,x,\psi,\sigma).
\]

Let $R$ be such that $r_{\sigma}<R<x_{1}$ and $L$ such that $L<l_{\sigma}$
and let $\eta$ such that $0<\eta<\min_{s,t}\sigma(s,t)$ and $A$
such that $A>\max\sigma$. For $n$ large enough we have $r_{\sigma_{n}}<R$
and $L<l_{\sigma_{n}}$, and $0<\eta<\int\int\sigma_{n}(s,t)dsdt$
and $\max\sigma_{n}<A$. Using Proposition \ref{prop:epsilon}, one
can find an $\epsilon_{x_{2}}$ depending only on $R,A$ and $\eta$
such that for all $n$ large enough, we have that for all $x\in[x_{1},x_{2}]$

\[
I_{\sigma_{n}}(x)=I_{\sigma_{n}}^{\epsilon_{x_{2}}}(x)
\]
Let us define $\mathcal{S}(\sigma,\epsilon):=\{\psi\in\mathcal{P}([0,1]):\langle\psi,\sigma\psi\rangle\geq\epsilon\}$.
Using the uniform convergence of $\sigma_{n}$ toward $\sigma$, we
have that for $n$ large enough 
\[
\mathcal{S}(\sigma_{n},\epsilon_{x_{2}})\subset\mathcal{S}(\sigma,\epsilon_{x_{2}}/2)\subset\mathcal{S}(\sigma_{n},\epsilon_{x_{2}}/4).
\]
Therefore for $n$ large enough, we have that for every $x\in[x_{1},x_{2}]$

\[
I_{\sigma_{n}}(x)=I_{\sigma_{n}}^{\epsilon_{x_{2}}}(x)\geq\inf_{\psi\in\mathcal{S}(\sigma,\epsilon_{x_{2}}/2)}\sup_{\theta\geq0}\mathcal{F}(\theta,x,\psi,\sigma_{n}) \geq I_{\sigma_n}(x). 
\]

Since $\mathcal{S}(\sigma,\epsilon_{x_{2}}/2)\subset\mathcal{S}(\sigma_{n},\epsilon_{x_{2}}/4)$
using the last result one can find $\theta_{*}$ depending only on
$A,\eta,\epsilon_{x_{2}},L,R,x_{1},x_{2}$ such that for $n$ large
enough we have that for every $x\in[x_{1},x_{2}]$.

\begin{eqnarray*}
I_{\sigma_{n}}(x) & = & \inf_{\psi\in\mathcal{S}(\sigma,\epsilon_{x_{2}}/2)}\sup_{\theta\geq0}\mathcal{F}(\theta,x,\psi,\sigma_{n})\\
 & = & \inf_{\psi\in\mathcal{S}(\sigma,\epsilon_{x_{2}}/2)}\sup_{0 \leq \theta\leq\theta_{*}}\mathcal{F}(\theta,x,\psi,\sigma_{n})
\end{eqnarray*}

and also

\begin{eqnarray*}
I^{\epsilon_{x_{2}}/2}_{\sigma}(x) & = & \inf_{\psi\in\mathcal{S}(\sigma,\epsilon_{x_{2}}/2)}\sup_{\theta\geq0}\mathcal{F}(\theta,x,\psi,\sigma)\\
 & = & \inf_{\psi\in\mathcal{S}(\sigma,\epsilon_{x_{2}}/2)}\sup_{0 \leq \theta\leq\theta_{*}}\mathcal{F}(\theta,x,\psi,\sigma).
\end{eqnarray*}

Then since $(\theta,x,\psi)\mapsto\mathcal{F}(\theta,x,\psi,\sigma_{n})$
converges uniformly on $[0,\theta_{*}]\times[x_{1},x_{2}]\times\mathcal{P}([0,1])$
toward $(\theta,x,\psi)\mapsto\mathcal{F}(\theta,x,\psi,\sigma)$,
we have that $x\mapsto I_{\sigma_{n}}(x)$ converges toward $x\mapsto I_{\sigma}^{\epsilon_{x_{2}}/2}(x)$
uniformly on $[x_{1},x_{2}]$. We then notice that the same proof
holds for any $\epsilon'\in(0,\epsilon_{x_{2}})$, implying that for
every $0<\epsilon'\leq\epsilon_{x_{2}}$, $I_{\sigma}^{\epsilon'/2}(x)=I_{\sigma}^{\epsilon_{x_{2}}/2}(x)$.
As a consequence, 
\[
I_{\sigma}^{\epsilon_{x_{2}}/2}(x)=\lim_{\epsilon'\to0}I_{\sigma}^{\epsilon'/2}(x)=I_{\sigma}(x)\,.
\]
Moreover, $I_{\sigma_{n}}(x)$ converges toward $I_{\sigma}^{\epsilon_{x_{2}}/2}(x)$
uniformly on $[x_{1},x_{2}]$. 
\end{proof}
Therefore, we can now prove Theorem \ref{thm-ratefuct}: 
\begin{proof}[Proof of Theorem \ref{thm-ratefuct}:]
Let $\sigma$ be a continuous variance profile such that $\sigma\neq0$
and let $(\sigma_{n})$ be a sequence of piecewise constant variance
profiles that converges uniformly toward $\sigma$. Let $x\in\mathbb{R}$. 
\begin{itemize}
\item If $x<r_{\sigma}$, for $\delta<(r_{\sigma}-x)/2$, $r_{\sigma_{n}}>(x+r_{\sigma_{n}})/2$
for $n$ large enough and then $\inf_{|y-x|\leq\delta}I_{\sigma_{n}}(y)=+\infty$
so $\hat{I}(x)=+\infty=I_{\sigma}(x)$. 
\item If $x=r_{\sigma}$, for any $\delta>0$, $|r_{\sigma_{n}}-x|\leq\delta$
for $n$ large enough. Since $I_{\sigma_{n}}(r_{\sigma_{n}})=0$ and
$\hat{I}(x)\geq0$, we have $\hat{I}(x)=0=I_{\sigma}(x)$. 
\item If $x>r_{\sigma}$, and $\delta<(x-r_{\sigma})/2$, since $I_{\sigma_{n}}(.)$
converges uniformly toward $I_{\sigma}(.)$ on $[x-\delta,x+\delta]$,
we have

\[
\lim_{n\to\infty}\inf_{|y-x|\leq\delta}I_{\sigma_{n}}(y)=\inf_{|y-x|\leq\delta}I_{\sigma}(y).
\]
Since $I_{\sigma}(.)$ is lower semi-continuous, we have 
\[
\sup_{\delta>0}\inf_{|y-x|\leq\delta}I_{\sigma}(y)=I_{\sigma}(x).
\]
This implies 
\[
\hat{I}(x)=I_{\sigma}(x).
\]

\end{itemize}
\end{proof}
{
\textit{Remark}: Regarding the lack of continuity of the rate function in $r_{\sigma}$, it is believed to be a mere technical obstruction rather than a serious one. Indeed following the proof of the approximation argument, to prove the continuity of the rate function, we would essentially only need to prove that the following limit holds:
\[ \lim_{n \to \infty} G_{\sigma_n}( r_{\sigma_n}) = G_{\sigma}( r_{\sigma})  .\]
However we could not find an argument to prove this point. }
		
	\appendix

\section{Appendix}

\subsection{Proof of Lemma \ref{lem:theta}}

\label{secsimp}
\begin{proof}
[Proof of Lemma \ref{lem:theta}]One can re-express $K(\theta,\phi(\theta,x,\psi))$
and $J_{\mu_{\sigma}}(x,\theta)$ in terms of $v=G_{\mu_{\sigma}}^{-1}(2\theta)>x$.
Indeed, because $G_{\mu_{\sigma}}$ is decreasing, this implies $G_{\mu_{\sigma}}(x)>2\theta$
and therefore, if we denote $\phi(t)$ the Radon-Nikodym derivative
of $\phi(\theta,x,\psi)$ with respect to the Lebesgue measure on
$[0,1]$, we find that 
\[
\phi(t):=\frac{d\phi(\theta,x,\psi)}{dt}(t)=\frac{G_{\sigma,t}(v)}{2\theta}\,.
\]
As a consequence, for $x>r_{\sigma}$ and $\theta$ such that $\theta<G_{\mu_{\sigma}}(x)/2$
we get: 
\begin{eqnarray*}
 &  & K(\theta,\phi(\theta,x,\psi))=\theta^{2}\int_{0}^{1}\int_{0}^{1}\phi(t)\phi(s)\sigma(t,s)dtds+\frac{1}{2}\int_{0}^{1}\log\phi(t)dt\\
 &  & \qquad=\frac{1}{4}\int_{0}^{1}\int_{0}^{1}G_{\sigma,t}(v)G_{\sigma,s}(v)\sigma(s,t)dsdt+\frac{1}{2}\int_{0}^{1}\log G_{\sigma,t}(v)dt-\frac{1}{2}\log G_{\mu_{\sigma}}(v)
\end{eqnarray*}
where we used that $2\theta=G_{\mu_{\sigma}}(v)$. We also have: 
\begin{eqnarray*}
J_{\mu_{\sigma}}(x,\theta) & = & \theta v-\frac{1}{2}\int\log(v-\lambda)d\mu_{\sigma}(\lambda)-\frac{1}{2}\log(2\theta)-\frac{1}{2}\\
 & = & \frac{1}{2}\Big[vG_{\mu_{\sigma}}(v)-\int\log(v-\lambda)d\mu_{\sigma}(\lambda)-\log G_{\mu_{\sigma}}(v)-1\Big]
\end{eqnarray*}
Now differentiating $h(v):=K(\theta,\psi(\theta,x,\psi))-J_{\mu_{\sigma}}(x,\theta)$
with respect to $v$ we get 
\begin{eqnarray*}
\frac{d}{dv}h(v) & = & \frac{1}{2}\Big[\int_{0}^{1}\int_{0}^{1}G'_{\sigma,t}(v)G_{\sigma,s}(v)\sigma(s,t)dsdt+\int_{0}^{1}\frac{G'_{\sigma,t}(v)}{G_{\sigma,t}(v)}dt-vG'_{\mu_{\sigma}}(v)\Big]\\
 & = & \frac{1}{2}\Big[\int_{0}^{1}G'_{\sigma,t}(v)\Big(\int_{0}^{1}G_{\sigma,s}(v)\sigma(s,t)ds+\frac{1}{G_{\sigma,t}(v)}-v\Big)dt\Big].
\end{eqnarray*}
Using that $G_{\sigma,t}$ satisfies \eqref{eq:Dyson}, we conclude
that $h'(v)=0$ for $v>x$. Furthermore when $v$ goes to $+\infty$,
$\theta$ goes to $0$, and $h(v)$ goes to $K(0,\psi(0,x,v))-J_{\mu_{\sigma}}(x,0)=0$
so we deduce that $h(v)=0$ for $v>x$, which proves the Lemma. 
\end{proof}
With this lemma, we have the the following simplification on the rate
function, that is we can take the supremum over $\theta\geq G_{\mu_{\sigma}}(x)/2$
which will be useful for the proof of the large deviation upper bound.

Finally let us prove the following proposition that justifies the
fact that in $r_{\sigma}$, $G_{\sigma,t}(r_{\sigma})$ is well defined: 
\begin{prop}
For every variance profile $\sigma$ piecewise continuous, one has
that when $x\to r_{\sigma}^{-}$, $t\mapsto G_{\sigma,t}(x)$ converges
toward some $t\mapsto G_{\sigma,t}(r_{\sigma})$ and furthermore this
convergence is in $L^{2}$ and Lebesgue almost everywhere. 
\end{prop}

\begin{proof}
First, since for all $t$, $x\mapsto G_{\sigma,t}(x)$ the Stieltjes
transform of a measure on $]r_{\sigma},+\infty]$, it is positive
and decreasing. Therefore, using the monotone convergence theorem,
one only needs that $t\mapsto G_{\sigma,t}(x)$ is bounded in $L^{2}$,
which is given for instance by \cite[Proposition 2.1]{AEK2}. 
\end{proof}

\subsection{Proof of Lemma \ref{lem:Ratefunctionbound}}
We give the proof in
the piecewise continuous case which of course implies the piecewise
constant case. Denoting $\text{Leb}$ the Lebesgue measure on $[0,1]$,
we have by definition of $I_{\sigma}$ 
\[
I_{\sigma}(x)\leq\sup_{\theta\geq0}\mathcal{F}(\theta,x,\text{Leb},\sigma)\,.
\]
Moreover, using \eqref{eqn:Fupperbound} with $\psi=\text{Leb}$, we find

\[
\mathcal{F}(\theta,x,\text{Leb},\sigma)\leq-a(\theta-\theta_{x})^{2}+(\theta-\theta_{x})x\leq\frac{x^{2}}{4a}
\]
where we finally optimized over all possible values of $\theta-\theta_{x}$.
Therefore, we have 
\[
I_{\sigma}(x)\leq\frac{x^{2}}{4a}=\frac{x^{2}}{4\iint_{[0,1]^{2}}\sigma(u,t)dudt}\,.
\]

\subsection{Convergence of the largest eigenvalue}

\label{convlambda}

In this section, we prove Lemma \ref{lem:Measure-Concentration}.\ref{enu:edge} for completeness. For this we are first going to prove
it in the piecewise constant case. 
\begin{proof}[Proof of Lemma \ref{lem:Measure-Concentration} in the piecewise constant
case]
First since a.s. $\mu_{H^{N}}$ goes to $\mu$, we must have that
almost surely $\liminf_{N\to\infty}\lambda_{1}(H^{N})\geq r_{\sigma}$.
For the complementary upper bound, we use \cite[ Theorem 4.7]{Alt18}.
Our model satisfies the assumptions of this theorem with $L=1,\ell=1,a_{i}=0,\tilde{a}_{i}=0,\tilde{\alpha}_{1}=1,\tilde{\beta}_{1}=\tilde{\gamma}_{1}=0$.
Therefore, \cite[ Theorem 4.7]{Alt18} implies that for any $\epsilon>0$,
if we denote $\text{Spec}(H^{N})$ the spectrum of $H^{N}$, there
is $C_{\epsilon}>0$ such that

\begin{equation}
\mathbb{P}\left(\text{Spec}(H^{N})\subset[-r_{\sigma_{N}}-\epsilon,r_{\sigma_{N}}+\epsilon]\right)\geq1-\frac{C_{\epsilon}}{N^{100}}\label{alt}
\end{equation}
where $r_{\sigma_{N}}$ is obtained as the rightmost point of the
support of the measure $\mu_{\sigma_{N}}=\sum_{i=1}^{N}\alpha_{i}^{N}\mu_{i}^{N}$
so that the Stieltjes transforms $m_{i}^{N}(z)=\int(z-x)^{{-1}}d\mu_{i}^{N}(x)$
satisfy the Dyson-Schwinger equations

\[
\forall i=1,\dots,p,\text{ }\frac{1}{m_{i}^{N}(z)}=z-\sum_{j=1}^{p}\alpha_{j}^{N}\sigma_{i,j}m_{j}^{N}(z)\,.
\]
Here, $\alpha_{i}^{N}:=\#I_{i}^{N}/N$. Note that this equation only
depends on the function $\sigma_{N}$ equal to $\sigma_{i,j}$ on
$\tilde{I}_{i}^{N}\times\tilde{I}_{j}^{N}$ with $\tilde{I}_{k}^{N}:=[\sum_{i=1}^{k-1}\alpha_{i}^{N},\sum_{i=1}^{k}\alpha_{i}^{N}-1]$.

Our main goal is to show that $r_{\sigma_{N}}$ converges towards
$r_{\sigma}$ as $N$ goes to infinity.  This is not obvious and may become false if the size of one interval would go to zero, as the Baik-Ben Arous-P\'ech\'e transition shows. 
\begin{lem}
\[
\lim_{N\to\infty}r_{\sigma_{N}}=r_{\sigma}\,.
\]
\end{lem}

\begin{proof}
It turns out that the Stieltjes transform of $\mu_{\sigma}$ $m(z):=\int(z-x)^{-1}d\mu_{\sigma}(x)=\sum_{i=1}^{p}\alpha_{i}m_{i}(z)$
where the $m_{i}$ satisfy the following system of equations

\[
\forall i=1,\dots,p,\text{ }\frac{1}{m_{i}(z)}=z-\sum_{j=1}^{p}\alpha_{j}\sigma_{i,j}m_{j}(z).
\]
Hence, we only need to show that the solutions to these equations
are continuous in the $\alpha_{j}$, at least as far as the rightmost
point of the support is concerned. Now we define $M(z)=\frac{m(1/z)}{z}$
the moment generating functions of $\mu$ and for all $i$, $M_{i}(z)=\frac{m_{i}(1/z)}{z}$.
We define $M^{N}$ and $M_{i}^{N}$ the same way with $\mu_{N}$.
Therefore we have that 
\begin{equation}
\forall i=1,\dots,p,\text{ }M_{i}(z)=1+z^{2}M_{i}(z)\sum_{j=1}^{p}\alpha_{i}\sigma_{i,j}M_{j}(z)\label{dys}.
\end{equation}
Expanding $M_{i}(z):=\sum_{k=1}^{+\infty}c_{i}(k)z^{k}$, we find
that $c_{i}(2k+1)=0$ and that \eqref{dys} implies the following
recursion relationship

\[
c_{i}(2k)=\sum_{l=0}^{k-1}c_{i}(2l)\sum_{j=1}^{p}\alpha_{j}\sigma_{i,j}c_{j}(2(k-1-l))\,.
\]
Similarly, for every $N\in\mathbb{N}$, we denote $M_{i}^{N}(z)=\frac{m_{i}^{N}(1/z)}{z}=\sum_{k=1}^{+\infty}c_{i}^{N}(k)z^{k}$
and find the same relationship: 
\[
c_{i}^{N}(2k)=\sum_{l=0}^{k-1}c_{i}^{N}(2l)\sum_{j=1}^{p}\alpha_{j}^{N}\sigma_{i,j}c_{j}^{N}(2(k-1-l))
\]

Let $\epsilon>0$, and $N_{0}\in\mathbb{N}$ such that for $N\geq N_{0}$,
for all $i\in[1,p]$, 
\[
(1-\epsilon)\alpha_{i}\leq\alpha_{i}^{N}\leq(1+\epsilon)\alpha_{i}\,.
\]
then by an recursion on $k$, we easily see that for $N\geq N_{0}$:

\[
(1-\epsilon)^{k}c_{i}(2k)\leq c_{i}^{N}(2k)\leq(1+\epsilon)^{k}c_{i}(2k)
\]
Moreover, because $\mu_{\sigma}(x^{k})=\sum_{j=1}^{p}\alpha_{i}c_{i}(k)$
and $\mu_{\sigma^{N}}(x^{k})=\sum_{j=1}^{p}\alpha_{i}^{N}c_{i}(k)$,
we deduce that 
\[
(1-\epsilon)^{k+1}\mu(x^{2k})\leq\mu_{N}(x^{2k})\leq(1+\epsilon)^{k+1}\mu(x^{2k})
\]
then, using that the law of the entries is symmetric so that the spectrum
is symmetric, $r_{\sigma}=\lim_{k\to\infty}\mu(x^{2k})^{1/2k}$ we
have that

\[
(\sqrt{1-\epsilon})r_{\sigma_{N}}\leq r_{\sigma}\leq(\sqrt{1+\epsilon})r_{\sigma_{N}}\,.
\]
\end{proof}
\end{proof}
Let us now look at the piecewise continuous case. 
\begin{proof}[Proof of Lemma \ref{lem:Measure-Concentration} in the piecewise continuous
case.]
We consider a matrix $H^{N}$ with piecewise continuous covariance
and its approximation $(H^{N})^{(p)}$ as defined in Section \ref{sec:approximation}.
By Lemma \ref{approx}, we have

\[
\mathbb{P}[||H^{N}-(H^{N})^{(p)}||\geq K\epsilon]\leq e^{-N}
\]
so that 
\begin{equation}
\mathbb{P}[\exists i\in[1,N]\text{ s.t }|\lambda_{i}(H^{N})-\lambda_{i}((H^{N})^{(p)})|\geq K\epsilon]\leq e^{-N}\,.\label{cv}
\end{equation}
Applying Lemma \ref{lem:Measure-Concentration} in the piecewise constant
case to $(H^{N})^{(p)}$, we see that we only need to prove that

\begin{equation}
\lim_{p\to\infty}r_{\sigma^{(p)}}=r_{\sigma}.\label{cvr}
\end{equation}
But it is easy to see that we have the weak convergence

\[
\lim_{p\to\infty}\mu_{\sigma^{(p)}}=\mu_{\sigma}\,.
\]
This can be seen directly form the Dyson equation (see Theorem A.4
in \cite{HussonVar}) or from the fact that $\mu_{H^{N}}$ goes to
$\mu_{\sigma}$, $\mu_{(H^{N})^{(p)}}$ goes to $\mu_{\sigma^{(p)}}$
and following the argument above, $\mu_{H^{N}}$ and $\mu_{(H^{N})^{(p)}}$
are close with overwhelming probability. This implies 
\[
\liminf_{p\to\infty}r_{\sigma^{(p)}}\geq r_{\sigma}\,.
\]
To show that $r_{\sigma^{(p)})}$ goes to $r_{\sigma}$ we may again
use \eqref{cv} as follows. It indeed implies that for $\delta>0$,
$p\geq p_{0}$, we have for $N$ large enough: 
\[
\mathbb{P}[|\lambda_{\lceil\delta N\rceil}(H^{N})-\lambda_{\lceil\delta N\rceil}((H^{N})^{(p)})|\geq\epsilon]\leq e^{-N}.
\]
Importantly here, $p_{0}$ does not depend on $\delta$. Therefore,
if we assume that $\delta$ is such that $(1-\delta)$ is not a point
of discontinuity for neither $Q$ nor $Q^{(p)}$, we have that since
a.s. $\lim_{N\to\infty}\mu_{H^{N}}=\mu_{\sigma}$ and $\lim_{N\to\infty}\mu_{(H^{N})^{(p)}}=\mu_{\sigma^{(p)}}$,
a.s. we have that $\lim_{N\to\infty}\lambda_{\lceil\delta N\rceil}(H^{N})=Q_{\mu_{\sigma}}(1-\delta)$
and $\lim_{N\to\infty}\lambda_{\lceil\delta N\rceil}^{(p)}(H^{N})=Q_{\mu_{\sigma^{(p)}}}(1-\delta)$.
So with the inequality above, we conclude that

\[
|Q_{\mu_{\sigma^{(p)}}}(1-\delta)-Q_{\mu_{\sigma}}(1-\delta)|\leq\epsilon
\]
for every such $\delta>0$. Since $r_{\sigma}=\lim_{\delta\to0^{+}}Q(1-\delta)$
and $r(\sigma^{(p)})=\lim_{\delta\to0^{+}}Q^{(p)}(1-\delta)$ we conclude,
taking the limit in the preceding inequality for $\delta$ going to
$0^{+}$ that are points of discontinuity neither for $Q$ nor $Q^{(p)}$,
that for $p\geq p_{0}$:

\[
|r_{\sigma^{(p)}}-r_{\sigma}|\leq\epsilon
\]
which completes the proof of \eqref{cvr}.

\subsection{Convergence of the empirical measure}

Let us prove that the empirical measure of $H^{N}$ converges. In this section we will denote 
\[ \mathbb{H}^{\pm} :=\{ z \in \mathbb C : \pm \Im z > 0 \} .\]

Let us consider $(m^N_i(z))_{ 1 \leq i \leq N}$ the unique solutions to the following Dyson equation for $z \in \mathbb{H}^+$ and $ m^N_i(z) \in  \mathbb{H}^-$: 
\[ \forall i = 1, \dots, N \,\quad  z - \frac{1}{N}\sum_{j=1}^{N} \Sigma_{i,j}^N m^N_j(z) = \frac{1}{m^N_i(z)} \]

and $m_t(z)$ the solution of the limit version of this Dyson equation:

\[ z - \int_0^1 \sigma(s,t) m_s(z)= \frac{1}{m_t(z)}\,. \]

For $z \in \mathbb{H}^+$, both those equations have unique solution $(m_i^N(z))_{ i\in [1,N]}$ and $(m_t(z))_{t \in [0,1]}$ in respectively $(\mathbb{H}^-)^N$ and $L^{\infty} ( [0,1], \mathbb{H}^-)$, furthermore, for every $i \in [1, N]$ (resp. $t \in [0,1]$), $ z \mapsto m_i^N(z)$ (resp. $z \mapsto m_t(z)$) are Stieltjes transform of probability measures that we will denote respectively $\mu_{\Sigma^N,i}$ and $\mu_{\sigma,t}$. From those one can define: 
\[ \mu_{\Sigma^N} = \frac{1}{N} \sum_{i=1}^N \mu_{\Sigma^N,i} \text{ and } \mu_{\sigma} = \int_{0}^1 \mu_{\sigma,t} dt .\]

Using \cite{girko2012} we have the following convergence result that links the empirical measure of $H^N$ and the $\mu_{\Sigma^N}$:

\begin{prop} 
	For almost all $x \in \mathbb R$, we have that with probability $1$
	\[ \lim_{N \to \infty}  \Big| \frac{1}{N} \# \{ \lambda_i(H^N) \leq x \} - \mu_{\Sigma^N}((- \infty, x]) \Big| =0. \]
\end{prop}

However, we still need a little work to link the "deterministic equivalent" of the empirical measure $H^N$, that is $\mu_{\Sigma^N}$ and its natural continuous limit $\mu_{\sigma}$. We will actually prove a slightly finer result, in that we are going to prove that $\mu_{\Sigma^N,i}$ become uniformly close in $i$ to $\mu_{\sigma, t_i^N}$. 
For convenience sake, for $z \in \mathbb{C} $ we will denote $m^N(z)$ the $N$-tuple $( m_i^N(z))_{i \in [1,N]}$ and $m(z)$ the function $t \mapsto m_t(z)$. 
We have the following a priori bounds on the $m_i^N$ an $m_t$ for $z \in \mathbb{H}^+$

\[ \forall i = 1, \dots, N \, |m_i^N(z)| \leq \frac{1}{\Im z} \]

\[ \forall t \in [0,1] \, |m_t(z)| \leq \frac{1}{\Im z} .\]

%

We will prove the following 
\begin{prop}\label{prop:muconv}
	For all $z \in \mathbb C \setminus \mathbb R$
	\[\lim_{N \to \infty} \sup_{1 \leq i \leq N} | m^N_i(z) -m_{t_i^N}(z)| =0 \]
	\end{prop}

For this, let us consider $S^N$ and $\Phi^N$ the following operators defined for any $N$-tuple $u$ in $\mathbb{H}^{-}$ and $z \in \mathbb{H}^+$:

\[ (S^N u)_i = \frac{1}{N}\sum_{j=1}^{N} \Sigma_{i,j}^N u_j \]

\[ (\Phi^N(u)(z))_i = \frac{1}{z - (S^N u)_i } \]

and their  equivalent in the large $N$ limit for $u : [0,1] \to \mathbb{H}^-$

\[ (S u)(t) = \int_{0}^1 \sigma(s,t) u(s) ds \]

\[ (\Phi(u)(z))(t) = \frac{1}{z - (S u)(t) }. \]
Let us consider the function $D$ defined by 
\[ D(u,v) = \frac{|u -v|^2}{\Im u \Im v} \]

and 
\[ D^N(u,v) = \max_{i=1,\dots, N} D( u_i,v_i). \]

Let us consider $d$ the distance defined on $\mathbb{H}^-$ by 
\[ d(u,v) = arcosh \Big(\frac{ 1 + D(u,v) }{2} \Big) \]
and $d^N$ the distance defined on $(\mathbb{H}^-)^N$ by 
\[ d^N(u,v) = \max_{i=1,\dots, N} d( u_i,v_i) =arcosh \Big(\frac{ 1 + D^N(u,v) }{2} \Big) . \]
Remind that, following the proof of \cite[Proposition 2.1]{AEK2}, with $A = \sup_N \sup_{i,j} \Sigma^N_{i,j}$, we have that for $u,v \in (\mathbb{H}^-)^N$ 
\[ D^N ( \Phi^N(u)(z), \Phi^N(v)(z)) \leq \Big( 1 + \frac{(\Im z)^2}{A} \Big) ^{-2} D^N ( u, v) .\]
Therefore if we fix $z \in \mathbb{H}^+$ and if we let for $\delta > 0$

\[ F(\delta) = \sum_{n=0}^{+ \infty} arcosh \Big( \frac{1 + \Big( 1 + \frac{(\Im z)^2}{A} \Big) ^{-2n} \delta }{2} \Big) \]
	we have $\lim_{\delta \to 0} F( \delta) =0 $
	
	Since $m^N(z)$ is the fixed point of $\Phi^N(.)(z)$, and using the contraction property, 
	\[ d^N(m^N(z), u) \leq F( D^N( \Phi^N(u)(z),u)). \]
	Therefore to prove the result, if we let $\tilde{m}^N_i(z) = m_{t_i^N}(z)$, we only need to prove that 
	 \[ \lim_{N \to \infty} D^N( \Phi^N(\tilde{m}^N(z))(z),\tilde{m}^N(z)) =0. \]
	 
	 Using the fact that 
	 \[ D( u,v) =  D( - \frac{1}{u} , - \frac{1}{v}) \]
	 and for $\eta \in \mathbb R$
	 \[ D( u + \eta ,v + \eta ) = D( u,v) \]
	  we have that 
	  
	  \begin{eqnarray*}
	  	D( (\Phi^N(\tilde{m}^N(z))(z))_i,\tilde{m}_i^N(z)) &=& D(  z - S^N(\tilde{m}_i^N(z)) , z - \int_{0}^1 \sigma(t_i^N,s) m_s(z) ds )  \\
	  	& \leq & \frac{1}{(\Im z)^2} | S^N(\tilde{m}_i^N(z)) - \int_{0}^1 \sigma(t_i^N,s) m_s ds |.
	  	\end{eqnarray*}

Let $ \nu^N = \frac{1}{N} \sum_{i=1}^N \delta_{t_i^N} $. With this notation, it suffices to prove:

\[ \lim_{N \to \infty} \max_{i=1, \dots, N} |\int_{0}^1 \sigma(t_i^N,s) m_s ds - S^N (\tilde{m}^N(z))_i | = 0. \]

Using that $\tilde{m}_i^N(z)= m_{t_i^N}(z)$, we have 
\begin{eqnarray*}
	\int_{0}^1 \sigma(t_i^N,s) m_s ds - S^N (\tilde{m}^N(z))_i  &=& \int_0^1 \sigma(t_{i}^N,s) m_s(z) ds - \int_0^1 \sigma(t_{i}^N,s) m_s(z) d \nu^N(s) + A_i^N(z) 
	\end{eqnarray*}

where 
\[ A_i^N(z) =\frac{1}{N} \sum_{j=1}^N (\Sigma_{i,j}^N - \sigma(t_{i}^N,t_j^N)) \tilde{m}_j^N(z). \]

Using our a priori bound on $m$, and the fact the $\lim_{N \to \infty} \sup_{i,j} | \Sigma_{i,j}^N - \sigma(t_i^N,t_j^N)|= 0 $ we have that
\[ \limsup_{N \to \infty} \sup_i  |A_i^N(z)| =0 .\]

Using the fact that $\sigma$ can be continuously extended to $\overline{I_k} \times \overline{I_l}$ for any $k,l$, we have that the family of functions 
$\sigma_{s}: t \mapsto \sigma(s,t)$ for $s \in [0,1]$ is a family of uniformly continuous functions on each $I_k$. This implies that for every $\epsilon >0$, one can find $N_{\epsilon} < + \infty$ and a family $f_1, \dots f_{N_{\epsilon}}$ of $[0,1]$ that can be continuously extended to each $\overline{I}_k$. such that  
\[ \{ \sigma_{s} : s \in [0,1] \} \subset \bigcup_{i=1}^{N_{\epsilon}} B( f_{N_{\epsilon}}, \epsilon) \]
where $B$ are ball for the uniform distance on $[0,1]$. Since $t \mapsto m_t(z)$ can itself be continuously extended to each $\overline{I}_k$, we have that for every $\epsilon >0$, since $\nu^N$ goes weakly to $Leb_{[0,1]}$ and since the $f_i$ and $m_s(z)$ only have finitely many discontinuity points:

\[ \lim_{N \to \infty} \sup_{i=1,\dots, N_{\epsilon}} \Big|  \int_0^1 f_i(s) m_s(z) ds - \int_0^1 f_i(s) m_s(z) d \nu^N(s) \Big| = 0 .\]

Therefore, since $|m_i(z)| \leq 1/ \Im z$ it implies for any $\epsilon > 0$ that 
\[ \limsup_{N \to \infty} \sup_{t \in [0,1]}\Big| \int_0^1 \sigma(t,s) m_s(z) ds - \int_0^1 \sigma(t,s) m_s(z) d \nu^N(s) \Big| \leq \frac{2 A}{\Im z} \epsilon \]

which implies to desired limit.

\end{proof}
\subsection{Proof of Lemma \ref{lem:equalityRF}}\label{equalityRF}

	If we compute the derivatives of the two following functions for $\widehat{\theta} \geq 0$
	
	\[ \widehat{\theta} \mapsto J_{\mu_{\sigma}}(x,\widehat{\theta}+ G_{\mu_{\sigma}}(x)/2)-K(\theta,\phi(\widehat{\theta}+ G_{\mu_{\sigma}}(x)/2,x,\psi)) \text{ and } \widehat{\theta} \mapsto \widehat{\mathcal{F}}(\widehat{\theta},x,\psi,\sigma) \]
	we see that they coincide. Futhermore, for $\widehat{\theta}=0$, we see that both functions are $0$, using respectively Lemma \ref{lem:theta} and a straightforward computation. Therefore they must coincide for $\widehat{\theta} \geq 0$ and so after a change of variables we get 
	
	\[
	I_{\sigma}(x)=\inf_{\substack{\psi\in{\cal P}([0,1])\\
		}
	}\sup_{\theta\geq G_{\mu_{\sigma}}(x)/2 }\{J_{\mu_{\sigma}}(x,\theta)-K(\theta,\phi(\theta,x,\psi))\}\,.
	\]

We can  relax the condition $\theta \geq G_{\mu_{\sigma}}(x)/2$ into $\theta \geq 0$ by using  Lemma \ref{lem:theta}. To impose the condition $\langle\psi,S\psi\rangle\neq0$, we can notice that if $\langle\psi,S\psi\rangle=0$, we have
	
	\[ \widehat{\mathcal{F}}( \widehat{\theta},x,\psi,\sigma) := \widehat{\theta}\langle\frac{dt}{\varphi_{\star}(x)},\psi\rangle-\frac{1}{2}\int_{0}^{1}\log\left(\frac{2\widehat{\theta}}{\varphi_{\star}(x)}\frac{d\psi}{dt}+1\right)dt \]
	which has an infinite supremum in $\widehat\theta$ if $\langle\frac{dt}{\varphi_{\star}(x)},\psi\rangle\neq 0$, and a null supremum if also $\langle\frac{dt}{\varphi_{\star}(x)},\psi\rangle= 0$. None of these cases will contribute to the infimum over $\psi$.

\subsection{The block diagonal case} \label{sec:block}

Let us consider the case of a block diagonal random matrix
\[ H^N = \begin{pmatrix} H^N_1 & 0 \\
	0 & H^N_2 \end{pmatrix} \]
such that $H^N_1$ is $\lfloor \alpha N \rfloor \times\lfloor \alpha N \rfloor $ and $H^N_2$ is $N - \lfloor \alpha N \rfloor \times(N- \lfloor \alpha N \rfloor )$ for some $0 <\alpha < 1$ and such that $H^N$ has a variance profile $\sigma$ that has the following form
\[ \sigma(t,s)= \sigma_1(\alpha^{-1} t, \alpha^{ -1} s)\mathds{1}_{[0,\alpha]^{2}}(s,t)+\sigma_2( (1-\alpha)^{-1} (s -\alpha),(1-\alpha)^{-1} (t -\alpha)) \mathds{1}_{[\alpha,1]^{2}}(s,t) \]
 where $\sigma_{1}$ and $\sigma_2$ are two nonnegative functions of $[0,1]^2$.
 
 Then, one can apply our main result to $\lambda_1(H^N_1)$ and $\lambda_1(H^N_2)$. They then both satisfy large deviation principles in speed $N$ with the respective rate functions $x \mapsto \alpha I_{\sigma_1}( x / \sqrt{\alpha} )$ and $x \mapsto (1- \alpha) I_{\sigma_2}( x / \sqrt{1 - \alpha} )$. The factors $\alpha$ and $1- \alpha$ come from the fact that the dimension of the matrices $H^N_1$ and $H^N_2$ are respectively approximately $\alpha N$ and $(1- \alpha) N$, the renormalization is $1/ \sqrt{N}$ and we are looking for large deviation principle of speed $N$. 
 
 Since the functions $x \mapsto I_{\sigma_1}(x/ \sqrt{\alpha})$ and $x \mapsto I_{\sigma_2}(x/ \sqrt{1 -\alpha})$ are increasing function on $[ r_{\sigma}, + \infty[$ and since $\lambda_1( H^N) = \max \{\lambda_1(H^N_2), \lambda_1(H^N_1) \}$, it satisfies a large deviation principle with rate function
 
 \[ x \mapsto \min\{ \alpha I_{\sigma_1} \Big( \frac{x}{\sqrt{\alpha}} \Big ), (1-\alpha) I_{\sigma_2} \Big( \frac{x}{\sqrt{1 - \alpha}} \Big ) \} .\]
 We will now prove the following Proposition  which states this result at the level of the expression of the rate function $I_{\sigma}$ given in Theorem \ref{maintheo}. 

\begin{prop}\label{prop:block}
	If $\sigma(s,t)= \sigma_1(\alpha^{-1} t, \alpha^{ -1} s)\mathds{1}_{[0,\alpha]^{2}}(s,t)+\sigma_2( (1-\alpha)^{-1} (s -\alpha),(1-\alpha)^{-1} (t -\alpha)) \mathds{1}_{[\alpha,1]^{2}}(s,t)$ for some $0 <\alpha < 1$ and where $\sigma_{1}$ and $\sigma_2$ are two functions of $[0,1]^2$, then for every $x > r_{\sigma}$,
	\[ I_{\sigma}(x) = \min\{ \alpha I_{\sigma_1}\Big( \frac{x}{\sqrt{\alpha}} \Big) , (1-\alpha) I_{\sigma_2}\Big( \frac{x}{\sqrt{1 -\alpha}} \Big) \}. \]
	\end{prop}

 \begin{proof}[Proof of Proposition \ref{prop:block}]
 Let us look first at the Dyson equation from Lemma \ref{lem:Measure-Concentration}. If we define for any $t \in [0,1]$, $G_{1,t}(z) = G_{\mu_{\sigma, \alpha t}}(z)$ and $G_{2,t}(z) = G_{\mu_{\sigma, (1- \alpha)t + \alpha}}(z)$, we have 
 
 \begin{equation}
 \frac{1}{G_{1,t}(z)} = z - \alpha \int_{0}^1 \sigma_1(t,s) G_{1,s}(z) ds 
 \end{equation}
 	and 
 	 \begin{equation}
 \frac{1}{G_{2,t}(z)} = z - (1-\alpha) \int_{0}^1 \sigma_2(t,s) G_{2,s}(z) ds.
\end{equation}
Using the uniqueness of the solution to the Dyson equation, this implies that 
\[ G_{1,t}(z) =\frac{1}{\sqrt{\alpha}} G_{\mu_{\sigma_1,t}} \Big( \frac{z}{\sqrt{\alpha}} \Big) 
\mbox{ and }
  G_{2,t}(z) =\frac{1}{\sqrt{1 -\alpha}} G_{\mu_{\sigma_2,t}} \Big( \frac{z}{\sqrt{1 -\alpha}} \Big). \]
In particular this implies that 
\[ \mu_{\sigma} = \alpha (m_{\sqrt{\alpha}})\# \mu_{\sigma_1} +(1- \alpha) (m_{\sqrt{1 - \alpha}})\#\mu_{\sigma_2} \]
where $m_{a}: x \mapsto a x$ is the multiplication by $a$ and $f\#\mu$ denotes the push forward of $\mu$ by $f$. In particular

\[r_{\sigma} = \max \{ \sqrt{\alpha} r_{\sigma_1}, \sqrt{1 - \alpha} r_{\sigma_2} \}. \]
Furthermore, if let $h$ be the function defined on $[0,1]$ by 
 $h_{\alpha}(x) = \alpha + (1- \alpha)x=\alpha+m_{1-\alpha}(x)$, we can always write any measure $\psi \in \mathcal{P}([0,1])$ in the form 

\[ \psi = a m_{\alpha}\#\psi_1 + (1-a) h_{\alpha}\#\psi_2 \]
where $\psi_1,\psi_2 \in \mathcal{P}([0,1])$ and $a \in [0,1]$. Therefore, 

\begin{eqnarray*}
	I_{\sigma}(x) &=& \inf_{\substack{\psi\in{\cal P}([0,1])
		}
	}\sup_{\widehat{\theta}\geq0} \widehat{\mathcal{F}}( \widehat{\theta},x,\psi,\sigma) \\
&=& \inf_{a \in [0,1]} \inf_{\substack{\psi_1, \psi_2\in{\cal P}([0,1])
	}
}\sup_{\widehat{\theta}\geq0} \widehat{\mathcal{F}}( \widehat{\theta},x, a m_{\alpha}\#\psi_1 + (1-a) h_{\alpha}\#\psi_2 ,\sigma)
\end{eqnarray*}
If we let $\varphi^i_{\star}(z)$ be the function $t \mapsto G_{\mu_{\sigma_i,t}}(z)$ for $i=1,2$, we have 

\[ \varphi_{\star}(z)(t) = \frac{1}{\sqrt{\alpha}} \varphi^1_{\star}\Big(\frac{z}{\sqrt{\alpha}} \Big)((\alpha)^{-1}t) \mathds{1}_{ [0,\alpha] }(t) + \frac{1}{\sqrt{1 - \alpha}} \varphi^2_{\star}\Big(\frac{z}{\sqrt{ 1- \alpha}} \Big)((1- \alpha)^{-1}(t-\alpha)) \mathds{1}_{ [\alpha,1] }(t) \]
Therefore, if we define $S_{1},S_{2}$ as the linear operators with kernel $\sigma_{1}$ and $\sigma_{2}$ respectively, we get that
\begin{eqnarray*}
	 \langle\frac{dt}{\varphi_{\star}(x)},S(a m_{\alpha}\#\psi_1 + (1-a) h_{\alpha}\#\psi_2)\rangle &=& a \alpha^{3/2}  \langle\frac{dt}{\varphi^1_{\star}(\frac{x}{\sqrt{\alpha}})},S_1 \psi_1 \rangle\\
	  &+& (1-a) (1-\alpha)^{3/2}  \langle\frac{dt}{\varphi^2_{\star}(\frac{x}{\sqrt{\alpha}})},S_2 \psi_2 \rangle
	\end{eqnarray*}
and also 
$$
	\langle a m_{\alpha}\#\psi_1 + (1-a) h_{\alpha}\#\psi_2, S (a m_{\alpha}\#\psi_1 + (1-a) h_{\alpha}\#\psi_2)\rangle =  a^2  \langle \psi_1, S_1 \psi_1 \rangle + (1-a)^2 \langle \psi_2, S_2 \psi_2 \rangle\,.
$$
Finally, we compute
 \begin{eqnarray*}
&&\hspace{-0.8cm} \int_{0}^{1}\log\left(\frac{2\widehat{\theta}}{\varphi_{\star}(x)}\frac{d(a m_{\alpha}\#\psi_1 + (1-a) h_{\alpha}\#\psi_2)}{dt}+1\right)dt = \alpha 	\int_{0}^{1}\log\left(\frac{2a\widehat{\theta}}{\sqrt{\alpha}\varphi^1_{\star}(x/ \sqrt{\alpha})}\frac{d\psi_1}{dt}+1\right)dt
 	\\
 	&&\quad+ (1-\alpha) 	\int_{0}^{1}\log\left(\frac{2(1-a)\widehat{\theta}}{\sqrt{1-\alpha}\varphi^2_{\star}(x/ \sqrt{1- \alpha})}\frac{d\psi_2}{dt}+1\right)dt
 	\end{eqnarray*}
 where we used that 
\[ \frac{d (a m_{\alpha}\#\psi_1 + (1-a) h_{\alpha}\#\psi_2)}{dt}=  \frac{a}{\alpha}\frac{d \psi_1}{dt}(g^{-1}(t))\mathds{1}_{[0, \alpha]}(t)  +\frac{1- a}{1- \alpha}  \frac{d \psi_2}{dt}(h^{-1}(t)) \mathds{1}_{[\alpha,1]}(t)\,. \]
Thus,  we have:
$$
	\widehat{\mathcal{F}}( \widehat{\theta},x,a m_{\alpha}\#\psi_1 + (1-a) h_{\alpha}\#\psi_2,\sigma) = \alpha \widehat{\mathcal{F}}\Big( \frac{a\widehat{\theta}}{\sqrt{\alpha}},\frac{x}{\sqrt{\alpha}},\psi_1,\sigma_1 \Big) + (1-\alpha) \widehat{\mathcal{F}}\Big( \frac{(1-a)\widehat{\theta}}{\sqrt{1- \alpha}},\frac{x}{\sqrt{1- \alpha}},\psi_2,\sigma_2 \Big).$$
Therefore, choosing $a=0$ and $a=1$ we get

\[ I_{\sigma}(x) \leq \min\{ \alpha I_{\sigma_1} \Big( \frac{x}{\sqrt{\alpha}} \Big ), (1-\alpha) I_{\sigma_2} \Big( \frac{x}{\sqrt{1 - \alpha}} \Big ) \}. \]
For the reverse bound, we now need the following lemma 
\begin{lem}\label{quasiconcave}
		Let $\sigma :[0,1]^2 \rightarrow \mathbb{R}^+$ be a variance profile, $\psi \in \mathcal{P}([0,1])$, $x > r_{\sigma}$ the function $\widehat{\theta} \mapsto \widehat{\mathcal{F}}( \widehat{\theta} ,x, \psi, \sigma)$ is quasiconcave, meaning that for every $t \in \mathbb{R}$, the level set $\{ \widehat{\theta} \in ]r_{\sigma}, + \infty[, \widehat{\mathcal{F}}( \widehat{\theta} ,x, \psi, \sigma) \leq t \}$ are intervals.
\end{lem}
\begin{proof}[Proof]
	Using the formula \eqref{defF2} for $\widehat{\mathcal{F}}$, we have:
	
	\[ \frac{\partial }{\partial \widehat{\theta}}\widehat{\mathcal{F}}( \widehat{\theta} ,x, \psi, \sigma)|_{\widehat{\theta}=0} =0 \]
	 and 
	\[ \frac{\partial^3 }{\partial \widehat{\theta}^3}\widehat{\mathcal{F}}( \widehat{\theta} ,x, \psi, \sigma)= -8 \int_0^1 \Big(\frac{d \psi}{dt} \Big)^3 \frac{1}{( 2 \widehat{\theta} \frac{d \psi }{dt} + \varphi_{\star}(t))^3 } dt <0 .\] 
	Therefore $\widehat{\theta} \mapsto \frac{\partial }{\partial \widehat{\theta}}\widehat{\mathcal{F}}( \widehat{\theta} ,x, \psi, \sigma)$ is concave. Putting together these two pieces of informations,  we deduce that  this function  is positive on some interval $]  0, \widehat{\theta}_c[$, zero at $\widehat{\theta}_c$ and negative on $] \widehat{\theta}_c, + \infty[$. This means that $\widehat{\theta} \mapsto \widehat{\mathcal{F}}( \widehat{\theta} ,x, \psi, \sigma)$ is increasing on $[  0, \widehat{\theta}_c]$ and decreasing on $[ \widehat{\theta}_c,0]$ which proves the lemma.
	 
	\end{proof}

Let $\psi_1, \psi_2 \in \mathcal{P}( [0,1])$ and $a \in[0,1]$. Let $\widehat{\theta}_{c,1}, \widehat{\theta}_{c,2}$ be  the respective maxima of 
$ \widehat{\theta} \mapsto \widehat{\mathcal{F}}( \widehat{\theta} ,\frac{x}{\sqrt{\alpha}}, \psi_1, \sigma_1)$ and $ \widehat{\theta} \mapsto \widehat{\mathcal{F}}( \widehat{\theta} ,\frac{x}{\sqrt{1 - \alpha}}, \psi_2, \sigma_2)$. Let us assume first that $ \sqrt{\alpha} \widehat{\theta}_{c,1} / a < \sqrt{1 -\alpha} \widehat{\theta}_{c,2} / (1-a) $ then since on $[0, \theta_{c,2}]$ the function $ \widehat{\theta} \mapsto \widehat{\mathcal{F}}( \widehat{\theta} ,\frac{x}{\sqrt{1 - \alpha}}, \psi_2, \sigma_2)$ is nonnegative we have 
	\begin{eqnarray*}\sup_{\widehat{\theta} \geq 0} \widehat{\mathcal{F}}( \widehat{\theta},x,a m_{\alpha}\#\psi_1 + (1-a) h_{\alpha}\#\psi_2,\sigma) \hspace{-0.3cm}&\geq& \hspace{-0.3cm}  \alpha \widehat{\mathcal{F}}\Big( \widehat{\theta}_{c,1},\frac{x}{\sqrt{\alpha}},\psi_1,\sigma_1 \Big) \\
		&\geq& \hspace{-0.3cm} \alpha \sup_{\widehat{\theta} \geq 0} \widehat{\mathcal{F}}\Big( \widehat{\theta},\frac{x}{\sqrt{\alpha}},\psi_1,\sigma_1 \Big)  \geq  \alpha I_{\sigma_1} \Big( \frac{x}{\sqrt{\alpha}} \Big) \\
		& \geq  & \hspace{-0.3cm}\min\{ \alpha I_{\sigma_1} \Big( \frac{x}{\sqrt{\alpha}} \Big ), (1-\alpha) I_{\sigma_2} \Big( \frac{x}{\sqrt{1 - \alpha}} \Big ) \}.
		\end{eqnarray*}
If $ \sqrt{\alpha} \widehat{\theta}_{c,1} / a \geq  \sqrt{1 -\alpha} \widehat{\theta}_{c,2} / (1-a) $ the same reasoning yields the same lower bound by exchanging $\alpha$ and $1-\alpha$. 
	Taking the infimum on $a,\psi_1,\psi_2$, we conclude that
\[I_{\sigma}(x) \geq \min\{ \alpha I_{\sigma_1} \Big( \frac{x}{\sqrt{\alpha}} \Big ), (1-\alpha) I_{\sigma_2} \Big( \frac{x}{\sqrt{1 - \alpha}} \Big ) \}\]
which finishes the proof.

  \end{proof}

\subsection{The Wishart case}\label{sec:Wishart}

In this section we prove that the formula from Theorem \ref{maintheo} corresponds to the well known formula in the case of a Wishart matrix.

 We can use the following linearization trick. For a Wishart random matrix $W_N= \frac{1}{L} X X^T$ such that $X$ is a rectangular $L \times M$ random matrix with centered sharp sub-Gaussian coefficients of variance $1$, we can define the following $N \times N$ matrix:
\[ H^N = \begin{pmatrix} 0_{L} & \frac{1}{\sqrt{N}} X \\ \frac{1}{\sqrt{N}} X^T & 0_{M} \end{pmatrix} \]
where $N = L+M$. Then, if $ \lambda_1( W_N), \dots,  \lambda_L( W_N)$ is the spectrum of $W_N$, $ \lambda_{ i} ( W_N) =0$ for $i=1,\dots, L \wedge M$, and 
\[ \forall i =1, \dots , L \wedge M, \lambda_{ i} ( H^N) = \frac{\sqrt{L}}{\sqrt{N}} \sqrt{\lambda_{i}(W_N)}, \lambda_{ N -i+1} ( H^N) = - \frac{\sqrt{L}}{\sqrt{N}} \sqrt{\lambda_{i}(W_N)} \]
and the remaining eigenvalues of $H^N$ are $0$. Therefore if we assume that $M/L\to \alpha \in (1, + \infty)$, getting a large deviation principle for $\lambda_1(W_N)$ is essentially equivalent to getting a large deviation principle for $\lambda_1(H^N)$ whose variance profile is \[  \sigma(s,t) = \mathds{1}_{[0, 1/(1+ \alpha) ]}(s) \mathds{1}_{[1/(1+ \alpha),1  ]}(t) + \mathds{1}_{[0, 1/(1+ \alpha) ]}(s) \mathds{1}_{[1/(1+ \alpha),1  ]}(t)  \] 
and the formula we want for its rate function $I^{GH}$ is given by Theorem 1.7. of \cite{HuGu1}. Furthermore let us mention that
\[ r_{\sigma} = \frac{1 + \sqrt{\alpha}}{1+\alpha} .\]
Let us first begin by proving that the supremum and infimum in $I_{\sigma}$ can be exchanged. This is a consequence by Sion's minmax theorem  of the more general assumption made in \cite{HussonVar} that  $\psi \mapsto \langle \psi, S \psi \rangle$ is concave.

\begin{lem} \label{lem:concave} If $\sigma: [0,1]^2 \to \mathbb{R}^+$ is piecewise constant such that $\psi \mapsto \langle \psi, S \psi \rangle$ is concave on $\mathcal{P}([0,1])$ then for $x > r_{\sigma}$:
	
	\[ I_{\sigma}(x) = \sup_{\theta \geq 0 } \Big( J_{\mu_{\sigma}}(x,\theta) - \tilde{F}(\theta,x,\sigma) \Big) \]
	where 
	\[ \tilde{F}(\theta,x,\sigma)  = \sup_{\psi \in \mathcal{P}([0,1])} K(\theta, \phi(\theta,x,\psi)). \]
\end{lem}
\begin{proof}
	This is a straightforward application of Sion's minmax theorem. Indeed, using Lemma \ref{quasiconcave} and Lemma \ref{lem:theta} it is easy to see that for any $\psi,x$, 
	\[ \theta \mapsto \mathcal{F}( \theta,x,\psi,\sigma) .\]
	is quasi-concave. Furthermore, using that $H$ is concave as well as the assumption, we can see that 
	\[ \psi\mapsto \mathcal{F}( \theta,x,\psi,\sigma) \]
	is convex. Therefore using the minmax theorem, we get
	
	\begin{eqnarray*}
		I_{\sigma}(x) &=& \inf_{\substack{\psi\in{\cal P}([0,1])\\
			}
		}\sup_{\theta\geq0} \mathcal{F}( \theta,x,\psi,\sigma) \\
&=& \sup_{\theta\geq0} \inf_{\substack{\psi\in{\cal P}([0,1])}}\Big( J_{\mu_{\sigma}}(x,\theta) - K( \theta, \phi(\theta,x,\psi) )\Big) \\
&=& \sup_{\theta\geq0} \Big( J_{\mu_{\sigma}}(x,\theta) - \sup_{\substack{\psi\in{\cal P}([0,1])}}K( \theta, \phi(\theta,x,\psi)) \Big)\,. \\
		\end{eqnarray*}
	\end{proof}

In section 4.2 of \cite{HuGu1}, it is proved that the rate function for the Wishart case can also be expressed in the following way 
  
  \[ I^{GH}(x) = \sup_{\theta \geq 0} \Big( J_{\mu_{\sigma}}(x,\theta) - F(\theta,\sigma) \Big) \]
  where it is shown in Lemma 3.1 that 
  \[ F(\theta,\sigma) = \sup_{\psi \in \mathcal{P}([0,1]) }\Big( \theta^2 \langle \psi, S \psi \rangle - H(\psi) \Big) = \sup_{\psi \in \mathcal{P}([0,1]) } K(\theta,\psi).\]
One has to remember that since the variance profile is piecewise constant, this supremum can basically be reduced to the set of probability measures of the form $\psi = (1+ \alpha) a \mathds{1}_{[0,\frac{1}{1+ \alpha}]} dt + \frac{1+ \alpha}{\alpha}(1-a) \mathds{1}_{[\frac{1}{1+ \alpha},1]}  dt$ where $a \in[0,1]$. 
Therefore to complete the proof, we only need the following lemma
\begin{lem}
	In the Wishart case, for every $x > r_{\sigma}$, $\theta > 0$, we have
	\[ \tilde{F}( \theta,x,\sigma) =F(\theta,\sigma) .\]
	\end{lem}
\begin{proof}
	We assume $\alpha \geq 1$. Let us first establish what is $\phi_{\star}(z)$ is this case. Actually we can represent $\phi_{\star}(z)$ in the following way:
	
	\[ \phi_{\star}(z)(t) = G_1(z) \mathds{1}_{[0, \frac{\alpha}{1+ \alpha}]}(t) + G_2(z) \mathds{1}_{[\frac{\alpha}{1+ \alpha}, 1]}(t) \]
	where, if we denote $G_{\pi_{\alpha}}$ the Stieltjes transform of the Marcenko-Pastur law of parameter $\alpha$ we have if $G_{\pi_{\alpha}}(z) = \frac{z - \alpha + 1 - \sqrt{(z-\alpha -1)^2 -4 \alpha }}{2 z },$
	
	\[ G_1(z) = z (1+ \alpha) G_{\pi_{\alpha}}((1+ \alpha) z^2) \mbox{ and }	 G_2(z)  = z \frac{1+ \alpha}{\alpha} G_{\pi_{1/\alpha}}(\frac{1+ \alpha}{\alpha} z^2)\,.\]
		This comes from the fact that for large $z$, using the expansion of  $( z- H^N)^{-1}$ in $1/z$: 
	
	\[  ( z- H^N)^{-1} = \begin{pmatrix} z ( z^2 - \frac{XX^T}{N})^{-1} & * \\
		* &  z ( z^2 - \frac{X^TX}{N})^{-1} \end{pmatrix} .\]
	
Furthermore, we have 

\[ G_{\mu_{\sigma}}(z)= 2 \alpha (1 + \alpha)z G_{\pi_{\alpha}}((1+\alpha)z^2) + \frac{1 -  \alpha}{(1+ \alpha)z} .\]
We will also denote $\theta_c = G_{\mu_{\sigma}}( r_\sigma)$. We then find our formulas using that asymptotically the empirical measure of $(1+ \alpha)XX^{T}$ is a Marcenko Pastur of parameter $\alpha$ and the empirical measure of $\frac{(1+ \alpha)}{\alpha} X^TX$ is a Marcenko Pastur of parameter $1/\alpha$. Then let us define $\overline{G_{\pi_{\alpha}}}$ simply by changing the sign in the square root 
		
		\[ \overline{G_{\pi_{\alpha}}}(z) = \frac{z  \alpha +1 + \sqrt{(z-\alpha -1)^2 -4 \alpha }}{2 z }. \]
	and  define similarly $\overline{G_1},\overline{G_2},\overline{G_{\mu_{\sigma}}}$ and $ \overline{\phi_{\star}}$ by switching each time  $\overline{G_{\pi_{\alpha}}}$ for $G_{\pi_{\alpha}}$.
	If we denote $\mathcal{P}(\theta,x,\sigma)$ the image of the set $\mathcal{P}([0,1])$ by $\psi \mapsto \phi(\theta,x,\psi)$, we get
	\[ \tilde{F}(\theta,x,\sigma) = \sup_{ \psi \in\mathcal{P}(\theta,x,\sigma) } K( \theta,\phi(\theta,x,\psi)) \leq F( \theta,x) \]
	Let us first describe the minimizer in the expression of $F(\theta,\sigma)$ that we will denote $\psi(\theta)$. It satisfies the critical point equation in $\mathcal{P}([0,1])$: 
	
	\[ 2 \theta^2 S \psi(\theta) + \frac{1}{2} \Big(\frac{d \psi(\theta)}{dt} \Big)^{-1} dt  = \lambda dt  \]
	where $\lambda \in \mathbb{R}$ is some Lagrange multiplier. Since we are dealing with a piecewise constant variance profile, we can consider this to be an equation in $\mathbb R^2$.  Furthermore, by strict concavity this critical point equation is sufficient. Therefore our result boils down to proving that for any $x > r_{\sigma}$ and $\theta \geq 0$, there exoists some $\overline{\psi}$ such that $\psi(\theta) = \phi(\theta, x, \overline{\psi})$.
	
	One can then notice that if there is $z \in [r_{\sigma}, + \infty[$ such that $2\theta = G_{\mu_{\sigma}}(z)$, the Dyson equation can take the form
	\[ S (\phi_{\star}(z) dt) + \frac{1}{\phi_{\star}(z)} dt  = z dt \]
	meaning that choosing $\psi(\theta) = \frac{\phi_{\star}(z)}{G_{\mu_{\sigma}}(z)}dt$ does yield the desired critical point equation.
 In particular, it shows the desired  result  for any $x > r_{\sigma}$ and $\theta \leq \theta_c$. Indeed if $2 \theta \leq G_{\mu_{\sigma}}(x)$, then we simply have $\phi(\theta, x, \overline{\psi}) = \psi(\theta)$ for any $\overline{\psi}$ and if  $G_{\mu_{\sigma}}(x) \leq 2 \theta \leq 2 \theta_c $ there is $x'\leq x$ such that $2 \theta= G_{\mu_{\sigma}}(x')$ and we have with 
	 \[ \overline{\psi} = \frac{\phi_{\star}(x') -\phi_{\star}(x) }{2 \theta - G_{\mu_{\sigma}}(x)}dt \]
	 that 
	 \[ \psi(\theta) = \phi(\theta,x,\overline{\psi}) .\]
	 
However, one can look at the Dyson equation and deduce that
	\[ S (\phi_{\star}(z) dt) + \frac{1}{\phi_{\star}(z)} dt  - z dt =0.\]
	We can then see the righ hand side as an algebraic expression in $z$ that is equal to $0$ on an open set. So if we see it as a formal algebraic expression in the inderminate $Z$, it is equal to $0$. Using the morphism the $\mathbb{R}(\alpha)$-morphism that sends $\sqrt{(1+\alpha)^2(z-1) - 4 \alpha}$ on $-\sqrt{(1+\alpha)^2(z-1) - 4 \alpha}$, we have that 
	\[ S (\overline{\phi_{\star}}(z) dt) + \frac{1}{\overline{\phi_{\star}}(z)} dt  - z dt =0 \]
	Therefore, if $2 \theta = \overline{G}_{\mu_{\sigma}}(z)$, $\psi( \theta) = \frac{\overline{\phi}_{\star}(z)}{\overline{G}_{\mu_{\sigma}}(z)}dt$
	does satisfy the critical point equation. 
Furthermore, one can see that the functions $\overline{G}_{\pi_\alpha}$ and $\overline{G}_{\pi_\alpha}$ are increasing functions and therefore so are $\overline{G}_1,\overline{G}_1, \overline{G}_{\mu_{\sigma}}$.  Furthermore $\overline{G}_{\mu_{\sigma}}(\theta)$ goes to $+\infty$ as $\theta$ goes to $+\infty$. Therefore for $\theta \geq \theta_c$ we can find $x'$ such that $2 \theta = \overline{G}_{\mu_{\sigma}}(x')$. Then with 
\[ \overline{\psi} =  \frac{\overline{\phi_{\star}}(x') -\phi_{\star}(x) }{
	\overline{G}_{\mu_{\sigma}}(x') - G_{\mu_{\sigma}(x)}}dt \] 
we have $\psi(\theta)= \phi(\theta,x,\overline{\psi})$. The fact that $\overline{\psi}$ is indeed a probability measure, that is that $\frac{\overline{\phi_{\star}}(x') -\phi_{\star}(x) }{
	\overline{G}_{\mu_{\sigma}}(x') - G_{\mu_{\sigma}}(x)}$ comes from the fact that since $G_{\mu_\sigma}, G_1,G_2$ are decreasing on $(r_{\sigma},+ \infty)$ and $\overline{G}_{\mu_\sigma}, \overline{G }_1,\overline{G}_2$  are increasing and therefore 
\[ \overline{G}_{\mu_{\sigma}}(x') \geq \overline{G}_{\mu_{\sigma}}(r_{\sigma})= G_{\mu_\sigma}(r_{\sigma}) \geq  G_{\mu_\sigma}(x) \]
\[ \overline{G}_{1}(x') \geq \overline{G}_{1}(r_{\sigma})= G_{1}(r_{\sigma}) \geq  G_{1}(x) \]
and
\[ \overline{G}_{2}(x') \geq \overline{G}_{2}(r_{\sigma})= G_{2}(r_{\sigma}) \geq  G_{2}(x) \]
	
\end{proof}

 \bibliographystyle{acm}
\bibliography{bibLDP}

\end{document}